\documentclass[11pt]{article}
\usepackage{amsthm}
\usepackage{amsmath}
\usepackage{thmtools,thm-restate}
\usepackage{amssymb}
\usepackage{graphicx}
\usepackage{geometry}
\usepackage{commath}
\usepackage{xcolor}
\usepackage{color}
\usepackage{url}
\usepackage{algorithm}
\usepackage{subcaption}
\usepackage{enumitem}
\usepackage{authblk}
\usepackage{algorithmic}
\usepackage[title]{appendix} 

\usepackage{hyperref}   
\hypersetup{
    colorlinks=true,
    linkcolor=blue,
    filecolor=magenta,      
    urlcolor=cyan,
    pdftitle={Geometric Hermite Interpolation},
    }

\newcommand{\av}[3]{\operatorname{av}_{#1} \left(#2, #3\right)}

\numberwithin{equation}{section}
\newtheorem{thm}{Theorem}
\numberwithin{thm}{section} 

\newtheorem{prop}[thm]{Proposition}
\newtheorem{lemma}[thm]{Lemma}
\newtheorem{result}[thm]{Result}

\newtheorem{definition}[thm]{Definition}
\newtheorem{remark}[thm]{Remark}

\pdfoutput=1
\begin{document}

\title{Geometric Hermite Interpolation in $\mathbb{R}^n$ by Refinements}
\author{Hofit Ben-Zion Vardi, Nira Dyn, Nir Sharon\thanks{Corresponding author: nsharon@tauex.tau.ac.il}}
 
\date{}
\maketitle

\begin{abstract}
We describe a general approach for constructing a broad class of operators approximating high-dimensional curves based on geometric Hermite data. The geometric Hermite data consists of point samples and their associated tangent vectors of unit length. Extending the classical Hermite interpolation of functions, this geometric Hermite problem has become popular in recent years and has ignited a series of solutions in the 2D plane and 3D space. Here, we present a  method for approximating curves, which is valid in any dimension. A basic building block of our approach is a Hermite average, a notion introduced in this paper. We provide an example of such an average and show, via an illustrative interpolating subdivision scheme, how the limits of the subdivision scheme inherit geometric properties of the average. Finally, we prove the convergence of this subdivision scheme, whose limit interpolates the geometric Hermite data and approximates the sampled curve. We conclude the paper with various numerical examples that elucidate the advantages of our approach.
\end{abstract}


\section{Introduction}  

The problem of geometric Hermite approximation is to estimate a curve from a finite number of its samples, consisting of both points on the curve and their associated, normalized tangent vectors. This problem is fundamental in Computer Aided Geometric Design (CAGD), see, e.g.,~\cite{de1987high, meek1997geometric, xu2001geometric}, but it also appears in various other applied topics such as biomechanical engineering~\cite{Biomechanical_Engineering}, marine biology~\cite{tremblay2006interpolation}, scientific simulations~\cite{vargas2019leapfrog}, CNC machining~\cite{beudaert20115} and more. The primary challenge in geometric Hermite approximation is to incorporate the additional information, given in the form of normalized tangent vectors, for obtaining a better approximation of the sampled curve than approximation  based on points only. Moreover, the fact that the curve lies in a high-dimensional space poses additional practical and theoretical difficulties. Our approach provides a unified solution that copes with the above, and generates approximation with many favorable geometric properties. 

Classical Hermite interpolation deals with the linear problem of interpolating a function to data consisting of its values and its derivatives' values at a finite number of points by a polynomial. In~\cite{merrien1992family}, a family of linear Hermite subdivision schemes is introduced. This work opened the door for solving Hermite-type approximation through refinement~\cite{dubuc2009hermite, dyn1995analysis, han2005noninterpolatory, jeong2017construction, merrien2012hermite}. Although all these subdivision schemes are linear, refinement is the approach taken in this paper to the nonlinear problem of geometric Hermite interpolation.

Recent years have given rise to many nonlinear subdivision schemes and, in particular, to operators refining 2D and 3D geometric Hermite data, e.g.,~\cite{aihua2016new, lipovetsky2016weighted, reif2021clothoid}. The design of such subdivision operators is typically based on the reconstruction of geometric objects from a certain class. Such operators were suggested for circle reconstruction~\cite{chalmoviansky2007non, lipovetsky2016weighted, romani2010circle}, ellipse reconstruction~\cite{conti2015ellipse}, local shape-preserving~\cite{costantini2008constrained,manni2010shape} and local fitting of clothoids~\cite{reif2021clothoid}. The relative simplicity and high flexibility in the design of refinement rules allow posing advanced solutions for modern instances of the Hermite problem, e.g., for manifold values~\cite{moosmuller2016c, moosmuller2017hermite}. The latter is an example of the agility of subdivision schemes, perhaps best illustrated when considering the variety of methods that serve as approximation operators over diverse nonlinear settings. Note that although the Euclidean space is linear, the problem of geometric Hermite approximation for data sampled from a curve in an Euclidean space is nonlinear. This is due to the fact that the space of pairs of point-normalized tangent (for short, point-ntangent) is not linear, as it is a subset of $\mathbb{R}^n \times S^{n-1}$. 

In our approach, we interpret the refinement rules of a subdivision scheme as a method of averaging, as in~\cite{dyn2011convergence, dyn2017manifold, schaefer2008nonlinear}. Thus, generating curves in a nonlinear environment boils down to properly defining and understanding the averaging in the particular space. While most of the work in this direction refers to non-Hermite problems, some recent papers such as~\cite{lipovetsky2021subdivision, lipovetsky2016weighted, reif2021clothoid} introduce different techniques for averaging two point-normal pairs in the plane. Therefore, as a first step, we formulate the notion of Hermite average in $\mathbb{R}^{n}$, which encapsulates the concept of averaging two point-ntangent pairs sampled from a curve in $\mathbb{R}^{n}$ for any $n\geq2$. We then propose a nonlinear average in $\mathbb{R}^{n}$ that is based on B\'{e}zier curves and satisfies the requirements of the Hermite average. Determining such a mean is crucial for constructing our subdivision operators. Our paper is an example of how the averaging approach can serve as a fundamental bridge for obtaining approximation operators in general metric spaces, see also, e.g.,~\cite{kels2013subdivision}.

The construction of our
approximation operators, begins with the study of our Hermite average termed hereafter B\'{e}zier average. We show that this average satisfies several geometric properties such as invariance under similarity transformations and preservation of lines and circles. Then, we form a two-point interpolatory subdivision scheme by repeatedly applying the B\'{e}zier average as an insertion rule. This subdivision scheme serves as an illustrative example, which demonstrates the benefits of using our approach for constructing approximation operators based on refinement. For example, the geometric subdivision scheme we form inherits the geometric properties of the B\'{e}zier average; it is invariant under similarity transformations, it reconstructs lines and circles. Furthermore, our general approach of using an appropriate average, enables us to obtain a large class of nonlinear subdivision schemes refining geometric Hermite data, by replacing the linear average in the Lane--Riesenfeld subdivision schemes~\cite{lane1980theoretical} by the B\'{e}zier average. 

As part of the analysis, we prove the convergence of our interpolatory subdivision scheme, namely that it refines geometric Hermite data and its limit is a smooth curve that interpolates both the data points and their normalized tangent vectors. Moreover, the limit normalized tangent vectors are tangent to the limit curve. The proof of convergence of our non-linear subdivision scheme uses an auxiliary result which we validate with computer-aided evidence, combining analysis of multivariate functions and exhaustive search done by a computer program. The details appear in Appendix A  and as a complementary software code.

One additional advantage of our method is that the new schemes based on the B\'{e}zier  average, allow us to use a more flexible sampling strategy than typically assumed when applying subdivision schemes to approximate curves from Hermite data~\cite{floater2006parameterization}. We demonstrate this property numerically, emphasizing the approximation capabilities and how it assists in avoiding geometric artifacts, see Section~\ref{sec:examples}. In addition, we present numerical evidence of fourth-order approximation of curves, which is higher than what current theory guarantees for non-uniform sampling~\cite{floater2006parameterization}. This fourth-order approximation of curves  generalizes similar results in classical Hermite interpolation of functions. We also compare numerically our schemes with other known methods and illustrate the superiority of techniques based on the  B\'{e}zier  average. The entire set of examples is available for reproducibility as open Python code in~\url{https://github.com/HofitVardi/Hermite-Interpolation}.   

The paper is organized as follows. Section~\ref {sec:HermiteAveraging} presents the precise statement of the fundamental approximation problem we solve and defines the notion of Hermite averaging. Then, we introduce our B\'{e}zier average and prove that it is  Hermite average. The section is concluded with several geometric properties of this new Hermite average. Section~\ref{sec: subdivision} shows how to form Hermite subdivision schemes based on the B\'{e}zier average. There we offer additional properties of the B\'{e}zier average and use them to prove the convergence of our interpolatory Hermite subdivision scheme. In Section~\ref{sec:examples} we explore our solution to the problem of geometric Hermite interpolation numerically. The numerical examples demonstrate the performance of our Hermite subdivision schemes compared to other approximants. In the last section, conclusions and future work are discussed. This paper is accompanied by three appendices that include the computer-aided evidence mentioned above, proofs of some claims from the main sections, and a further discussion on the B\'{e}zier average.

\section{Averaging in the Hermite setting} \label{sec:HermiteAveraging}

It is important to place the motivation behind the following discussion in the area of Hermite approximation. We thus begin with some essential notation and definitions, leading to the formulation of the problem of geometric Hermite interpolation. 

In the following subsections we state a few lemmas and a proposition. The proofs of these results are given in Appendix~\ref{proofs}. 

\subsection{Notation and definitions}

A geometric Hermite data is a sequence 
\begin{equation} \label{eqn:HermiteData}
\left( (p_{j},  v_{j}) \right)_{j \in J}\subseteq\mathbb{R}^{n}\times S^{n - 1} ,
\end{equation}
where $\mathbb{R}^{n}$ is the Euclidean $ n$ dimensional space equipped with the Euclidean metric, $ S^{n - 1} \subset \mathbb{R}^{n} $ is the $n - 1$ sphere equipped with the angular metric, and $J=0,1,\ldots,N$ where $N \in \mathbb{N}$. We view such a sequence as samples of a $G^1$ curve, consisting of points and tangent directions. Unless otherwise stated, we denote by $ d\left(\cdot ,\cdot\right)$ and $ g\left(\cdot ,\cdot\right)$ the Euclidean and angular metrics on $\mathbb{R}^{n}$ and $ S^{n - 1}$, respectively. Moreover, $\left\langle \cdot ,\cdot\right\rangle$ and $\left\Vert \cdot\right\Vert$ refer to the Euclidean inner product and norm in $\mathbb{R}^{n}$.

Let $\begin{pmatrix}
p_0\\v_0\\
\end{pmatrix},\begin{pmatrix}
p_1\\v_1\\
\end{pmatrix}\in\mathbb{R}^n\times S^{n-1}$ s.t. $p_0\neq p_1$.

Consider the following functions of $\left(\begin{pmatrix}
p_0\\v_0\\
\end{pmatrix},\begin{pmatrix}
p_1\\v_1\\
\end{pmatrix}\right)$:

The normalized vector of difference between $p_0$ and $p_1$ is
\begin{equation}
\label{u definition}
    u\left(\begin{pmatrix}
p_0\\v_0\\
\end{pmatrix},\begin{pmatrix}
p_1\\v_1\\
\end{pmatrix}\right):=\frac{p_1-p_0}{\norm{p_1-p_0}}\in S^{n-1},
\end{equation}

The distance between $v_0$ and $v_{1}$ in terms of their angular distance is
\begin{equation}
\label{theta definition}
    \theta\left(\begin{pmatrix}
p_0\\v_0\\
\end{pmatrix},\begin{pmatrix}
p_1\\v_1\\
\end{pmatrix}\right):=g\left(v_0,v_1\right)=\arccos\left(\left\langle v_0,v_1\right\rangle\right)\in [0,\pi],
\end{equation}

The deviations of $v_0$ and $v_1$ from $u$ (in terms of their angular distance) is
\begin{equation}
\label{theta j definition}
    \theta_{j}\left(\begin{pmatrix}
p_0\\v_0\\
\end{pmatrix},\begin{pmatrix}
p_1\\v_1\\
\end{pmatrix}\right):=g\left(v_j,u\right)= \arccos\left(\left\langle v_j,u\right\rangle\right)\in [0,\pi],\ \ j=0,1.
\end{equation}

The $L_2$ norm of $\begin{pmatrix}
\theta_{0},
\theta_{1}
\end{pmatrix}$, regarded as a vector in 
$\mathbb{R}^{2}$, is
\begin{equation}
\label{sigma definition}
    \sigma\left(\begin{pmatrix}
p_0\\v_0\\
\end{pmatrix},\begin{pmatrix}
p_1\\v_1\\
\end{pmatrix}\right):=
\sqrt{\theta_1^2+\theta_2^2}\in [0,\sqrt{2}\pi].
\end{equation}
The last quantity was first used in~\cite{reif2021clothoid} in the convergence analysis  of a Hermite subdivision scheme. 

When there is no possibility of ambiguity, we omit the notation of the variables and simply write $u,\theta,\theta_0,\theta_1$ and $\sigma$. 

We conclude with the definition of the problem we solve.
\begin{definition}
\label{def: GHI}
The problem of geometric Hermite interpolation is to find a $G^{1}$ curve $\widetilde{\gamma} \colon \mathbb{R} \to \mathbb{R}^{n}$ which interpolates the geometric Hermite data $\left( (p_{j},v_j) \right)_{j=1}^N \subseteq\mathbb{R}^{n}\times S^{n - 1}$, see~\eqref{eqn:HermiteData}, where $N\in \mathbb{N}$. We assume the data is sampled from a $G^{1}$ curve $\gamma \colon \mathbb{R} \to \mathbb{R}^{n}$, 
\begin{equation} \label{eqn:sampled_hermite_data}
    \gamma(t_j)=p_j ,  \quad T(\gamma,t_j)=v_j , \quad j = 1,\ldots,N ,
\end{equation}
with $t_j < t_{j+1}$ for all $1 \le j \le N-1$. Here $T$ operates on $G^1$ curves and returns the ntangent to a curve at a specified parameter value. The curve $\widetilde{\gamma}$ should satisfy the interpolation conditions, 
\begin{equation}
\label{interpolation}
    \widetilde{\gamma}(\rho({\widetilde{t}}_j))= p_j \quad\text{and}\quad T(\widetilde{\gamma},\rho(\widetilde{t}_j))=v_j , \quad j = 1,\ldots,N ,
\end{equation} 
where $\rho$ is a parametrization and ${\widetilde{t}}_j < {\widetilde{t}}_{j+1}$ for all $1 \le j \le N-1$.
\end{definition}

\subsection{Hermite Average}

We address the problem of geometric Hermite interpolation through a subdivision process which is defined via an average over $\mathbb{R}^{n}\times S^{n - 1}$. We call this average a Hermite average, define it, and later propose a method to construct such a mean. The new definition illustrates some of the significant differences between classical and geometric Hermite settings. 

The classical concept of average is perhaps best demonstrated by the case of positive real numbers, see ,e.g.,~\cite{itai2013subdivision}. An average is a bivariate function with a weight parameter
\[  
\operatorname{av}  \colon [0,1] \times \left( (0,\infty) \times (0,\infty) \right) \to (0,\infty) , \] 
For brevity, we use the form $\av{\omega}{x}{y}$ where $\omega\in[0,1]$ is the weight. It is worth noting that in some cases it is essential to extend the weight parameter beyond the $[0,1]$ segment, see e.g.,~\cite{dyn2017global}. The popular value $\omega = 0.5$ leads to the well-known average expressions such as $\frac{x+y}{2}$ or $\sqrt{xy}$, that correspond to the linear $(1-\omega)x+\omega y$ and geometric $x^{1-\omega} y^\omega $ averages. Many other common families of averages exist, for example the $p$-averages $\left( (1-\omega)x^p+\omega y^p \right)^{\frac{1}{p}}$, which include the above linear and geometric means as special cases. This example is itself a special case of a wider family of averages of the form $F^{-1}\left( (1-\omega)F(x)+\omega F(y) \right)$, when $F$ is an appropriate invertible function. The $p$-averages family comes up when using $F(x) = x^p$. For more details, see~\cite{dyn2011convergence}.

As we generalize the concept of averaging to more intricate settings, we wish to follow the averages' fundamental properties. Namely, the basic properties satisfied by the above bivariate function $\operatorname{av}$,
\begin{enumerate}
    \item Identity on the diagonal: $\av{\omega}{x}{x}=x$.
    \item Symmetry: $\av{\omega}{x}{y}=\av{1-\omega}{y}{x}$.
    \item End points interpolation: $\av{0}{x}{y}=x$ and $\av{1}{x}{y}=y$.
    \item Boundedness: $\min\{x,y\}\leq \av{\omega}{x}{y} \leq\max\{x,y\}$.
\end{enumerate}
Note that the last property, unlike the first three, requires ordering (partial ordering). This harsh requirement can be relaxed by using a metric and modifying the last property to ensure metric-related boundedness. 

Considering the classical case and the absence of a ``natural'' metric in the Hermite setting, we define a Hermite average by adjusting the first three properties listed above. Recall that a pair of elements in our domain is viewed as two samples of a regular, differentiable curve. Thus, there is a native hierarchy associated with the direction of the curve. Under this interpretation, it makes sense to consider an average that is sensitive to orientation. It also calls for a different understanding of the diagonal since an element $\left(p,v\right)\in \mathbb{R}^n\times S^{n-1}$ cannot represent two distinct, yet close enough samples of a regular curve. Therefore, we refer to the diagonal utilizing the notion of limit. The limit we consider is the most fundamental approximation as one approaches a point on a curve. Finally, we treat the symmetry as reversing the curve and so the orientation of the points. Next is the formal definition.
\begin{definition}
\label{Hermite average}
We term $H$ a Hermite average over $ X \subseteq \left(\mathbb{R}^{n}\times S^{n-1}\right)^2$ if $H:[0,1]\times X \to \left(\mathbb{R}^{n}\times S^{n-1}\right)$ and the following holds
for any $ \omega\in [0,1]$ and $\left(\begin{pmatrix}
    p_{0} \\ 
    v_{0} \\ 
    \end{pmatrix},\begin{pmatrix}
    p_{1} \\ 
    v_{1} \\ 
    \end{pmatrix} \right)\in X$, 
\begin{enumerate}
    \item Identity on the ``limit diagonal": 
    \begin{equation*}  H_{\omega }\left(\begin{pmatrix}
p \\ 
v \\ 
\end{pmatrix},\begin{pmatrix}
p+tv \\ 
v \\ 
\end{pmatrix}\right)\xrightarrow{t\longrightarrow 0^{ + }}\begin{pmatrix}
p\\ 
v\\ 
\end{pmatrix}  .\end{equation*}
    \item Symmetry with respect to orientation:
     \[ \text{If} \,\, H_{\omega }\left(
    \begin{pmatrix}
    p_{0} \\ 
    v_{0} \\ 
    \end{pmatrix},\begin{pmatrix}
    p_{1} \\ 
    v_{1} \\
    \end{pmatrix}
    \right) =\begin{pmatrix}
    p \\ 
    v \\ 
    \end{pmatrix} \, \text{then} \, \, H_{1 - \omega }\left(
    \begin{pmatrix}
    p_{1} \\ 
     - v_{1} \\ 
    \end{pmatrix},\begin{pmatrix}
    p_{0} \\ 
     - v_{0} \\
    \end{pmatrix}
    \right) =\begin{pmatrix}
    p \\ 
     - v \\ 
    \end{pmatrix}. \]
    \item End points interpolation: 
    \[ H_{j}\left(\begin{pmatrix}
    p_{0} \\ 
    v_{0} \\ 
    \end{pmatrix},\begin{pmatrix}
    p_{1} \\ 
    v_{1} \\
    \end{pmatrix}\right)
    =\begin{pmatrix}
    p_{j} \\ 
    v_{j} \\ 
    \end{pmatrix}, \quad j=0,1.\]
\end{enumerate} 
\end{definition}

\begin{remark}
\label{metric}
In the presence of an appropriate metric $\tilde{d}$, the classical boundedness property can emerge as 
\[ \tilde{d}\left(\begin{pmatrix}
    p_{j} \\ 
    v_{j} \\ 
    \end{pmatrix},\begin{pmatrix}
p_{\omega } \\ 
v_{\omega } \\ 
\end{pmatrix}\right)\leq \tilde{d}\left(\begin{pmatrix}
    p_{0} \\ 
    v_{0} \\ 
    \end{pmatrix}, \begin{pmatrix}
    p_{1} \\ 
    v_{1} \\ 
    \end{pmatrix}\right), \quad j=0,1, \]
    where $\begin{pmatrix}
p_{\omega } \\ 
v_{\omega } \\ 
\end{pmatrix}=H_\omega\left(\begin{pmatrix}
    p_{0} \\ 
    v_{0} \\ 
    \end{pmatrix},\begin{pmatrix}
    p_{1} \\ 
    v_{1} \\
    \end{pmatrix}\right).$
\end{remark}

\subsection{B\'{e}zier Average}

In this section we introduce the ``B\'{e}zier average". While the definition considers a weighted average for any weight $ \omega \in\left[0,1\right]$, we mainly focus on the special case  $ \omega  = \frac{1}{2}$, which is required for our refinements.

We start our construction with an \textit{admissible} pair $\begin{pmatrix}
\begin{pmatrix}
p_{0} \\ 
v_{0} \\ 
\end{pmatrix},\begin{pmatrix}
p_{1} \\ 
v_{1} \\ 
\end{pmatrix}
\end{pmatrix}$ that satisfies the following two conditions
\begin{equation} \label{eqn:conditions_on_pts}
\begin{split}
    p_0 &\neq p_1, \\
    v_0=v_1=u \text{ or } v_0,v_1,u &\text{ are pair-wise linearly independent,}
\end{split}
\end{equation}
where $u$ is defined in \eqref{u definition}.\\
The full reasoning behind the conditions of~\eqref{eqn:conditions_on_pts} is revealed in the following discussion. Our next step is to generate the B\'{e}zier curve, which we denote by $ b\left(t\right) $, based on the following control points:
\begin{equation} \label{control points}
p_{0}, p_{0} + \alpha v_{0}, p_{1} - \alpha v_{1}, p_{1} .
\end{equation}
Here,
\begin{equation} \label{eqn:alpha}
    \alpha:=\frac{d\left(p_{0},p_{1}\right)}{3\cos^{2}\left(\frac{\theta_{0} + \theta_{1}}{4}\right)},
\end{equation}
and $\theta_0,\theta_1$ are as defined in~\eqref{theta j definition}. Note that the second condition of~\eqref{eqn:conditions_on_pts} guarantees that~\eqref{eqn:alpha} is well-defined since $\theta_0 + \theta_1<2\pi$.

The B\'{e}zier curve $b(t),\ t\in [0,1]$ is given explicitly by
\begin{equation}
\label{bezier curve}
\begin{split}
    b\left(t\right) & = \left(1 - t\right)^{3}p_{0} + 3t\left(1 - t\right)^{2}\left(p_{0} + \alpha v_{0}\right) + 3t^{2}\left(1 - t\right)\left(p_{1} - \alpha v_{1}\right) + t^{3}p_{1} \\
    & =\left(t - 1\right)^{2}\left(2t + 1\right)p_{0} + t^{2}\left(3 - 2t\right)p_{1} + 3t\left(1 - t\right)\alpha\left(\left(1 - t\right)v_{0} - tv_{1}\right),
\end{split}
\end{equation}
and its derivative is thus,
\begin{equation}
\label{tangents curve}
\frac{d}{dt}b\left(t\right) = 6t\left(1 - t\right)\left(p_{1} - p_{0}\right) + 3\left(3t - 1\right)\left(t - 1\right)\alpha v_{0} + 3t\left(3t - 2\right)\alpha v_{1}.
\end{equation}
\begin{remark} \label{remark: linearly independence}
Note that if $ u,v_{0},v_{1}$ are linearly independent as vectors in $\mathbb{R}^{n}$ then $ \frac{d}{dt}b(t)\neq 0$ for any $ t\in\left(0,1\right)$.\ \ Moreover, it can be shown that if two of those vectors are linearly dependent while the third is independent of each one of them, then $ \frac{d}{dt}b(t)\neq 0$ for any $ t\in\left(0,1\right)$.\ \ As well, it can be verified that if $ v_{0} = v_{1} =u$ then $ \frac{d}{dt}b(t) = p_{1} - p_{0}=d\left(p_0,p_1\right)u$.
\end{remark}

We are now ready to present the B\'{e}zier average.
\begin{definition}
\label{def: average def}
The B\'{e}zier average of $\begin{pmatrix}
p_{0} \\ 
v_{0} \\ 
\end{pmatrix}$ and $\begin{pmatrix}
p_{1} \\ 
v_{1} \\ 
\end{pmatrix}$ with weight $ \omega$ is defined as the point and  its normalized tangent vector of the B\'{e}zier curve \eqref{bezier curve} at $t=\omega$.  Namely,
\begin{equation} \label{eqn:bezierAvg}
B_{\omega }\left(\begin{pmatrix}
p_{0} \\ 
v_{0} \\ 
\end{pmatrix},\begin{pmatrix}
p_{1} \\ 
v_{1} \\ 
\end{pmatrix}\right) = \begin{pmatrix}
p_{\omega } \\ 
v_{\omega } \\ 
\end{pmatrix} ,
\end{equation}
where  $ p_{\omega } = b\left(\omega\right)$ and $ v_{\omega } =\frac{b^{'}(\omega )}{\norm{b^{'}(\omega)}}$ (or $ v_{\omega } = 0$ if $ b^{'}\left(\omega\right) = 0$).
\end{definition}

In the spirit of Definition~\ref{def: GHI}, the above means that if $b(t),\ t\in [0,1]$ of~\eqref{bezier curve} is a regular curve, that is $\frac{d}{dt}b(t)$ does not vanish, then $v_{\omega}=T(b(\omega),\omega)$ for any $\omega\in[0,1]$.

Figure~\ref{fig:average} demonstrates the procedure of averaging according to Definition~\ref{def: average def}. In particular, we present two examples that correspond to two different cases. On the left the B\'{e}zier curve which we sample according to Definition~\ref{def: average def} is convex, while on the right it is non-convex. 
\begin{figure}
    \centering
    \includegraphics[width=0.45\textwidth]{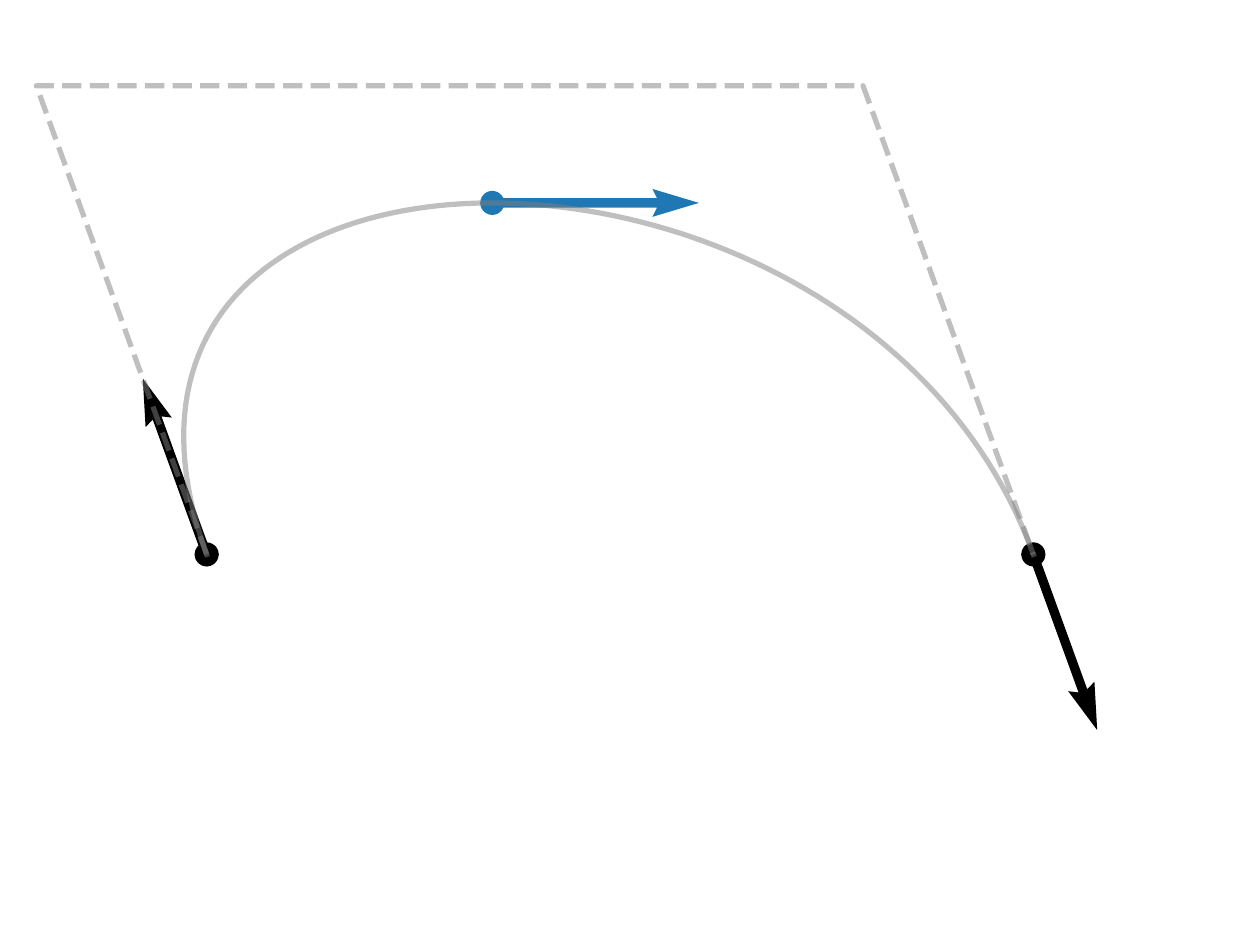}
    \quad
    \includegraphics[width=0.45\textwidth]{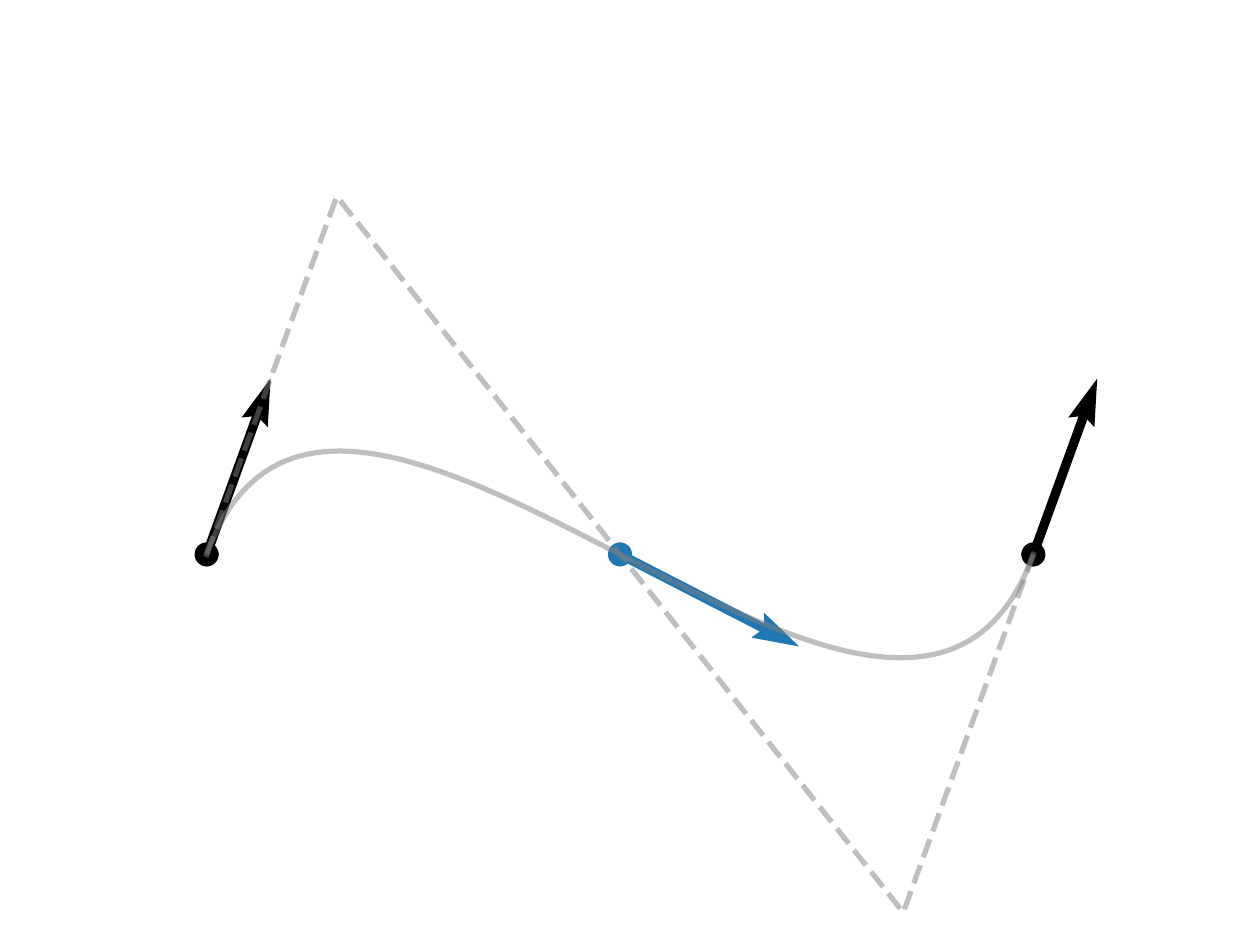}
    \caption{Computing B\'{e}zier average in two examples. In each example the data is in black, the averaged element is in blue, the control polygon and its associated B\'{e}zier curve are in gray.}
    \label{fig:average}
\end{figure}

Definition~\ref{def: average def} does not guarantee that the tangent vector is non-vanishing. That is to say, one can find examples of two pairs of Hermite data where  $ B_{\omega }\left(\begin{pmatrix}
p_{0} \\ 
v_{0} \\ 
\end{pmatrix},\begin{pmatrix}
p_{1} \\ 
v_{1} \\ 
\end{pmatrix}\right) =\begin{pmatrix}
p_{\omega } \\ 
0 \\ 
\end{pmatrix}\not\in \mathbb{R}^n\times S^{n-1}$. However, as is stated in the next lemma, this is not the case for $ \omega  = \frac{1}{2}$. 
\begin{lemma}
\label{non vanishing vectors}
Under the notation of Definition~\ref{def: average def}, 
\begin{equation}
p_{\frac{1}{2}} =\frac{1}{2}\left(p_{0} + p_{1}\right) +\frac{3}{8}\alpha\left(v_{0} - v_{1}\right),
\label{mid-point}
\end{equation}
\begin{equation}
v_{\frac{1}{2}}  =\frac{p_{1} - p_{0} - 0.5\alpha\left(v_{0} + v_{1}\right)}{\norm{ p_{1} - p_{0} - 0.5\alpha\left(v_{0} + v_{1}\right)} }.
\label{mid-vector}
\end{equation}
In particular, $ v_{\frac{1}{2}} \neq 0$.
\end{lemma}

Next we claim that the B\'{e}zier average (Definition~\ref{Hermite average}) is a Hermite average over $X$ (Definition~\ref{def: average def}). For the B\'{e}zier average we define $X$ to be the set of all pairs in $\left(\mathbb{R}^{n}\times S^{n-1}\right)^2$ satisfying conditions~\eqref{eqn:conditions_on_pts} and producing non vanishing averaged vectors for any weight in $[0,1]$. Note that by Remark~\ref{remark: linearly independence} $X$ is not empty. A discussion about the set $X$ is differed to Appendix~\ref{app:subsec_conjecture}. 
\begin{prop}
\label{Bezier average is Hermite average}
    The B\'{e}zier average is a Hermite average over $X$.
\end{prop}
\begin{remark} \label{alternative average}
The value of $ \alpha$ we chose for the definition of the B\'{e}zier average, is suggested in [4] in order to approximate circular arcs in $\mathbb{R}^{2}$ by B\'{e}zier curves. In case $n=2$ and under certain assumptions about the input, this choice coincides with the choice of $\alpha=\frac{d\left(p_{0},p_{1}\right)}{3\cos^{2}\left(\frac{\theta}{4}\right)}$. The latter choice of $\alpha$ yields the average defined in \cite{lipovetsky2021subdivision}, which can also be shown to be a Hermite average over a suitable set.
\end{remark}

\subsection{Geometric properties of the B\'{e}zier average}

We conclude this section by introducing three additional properties of the B\'{e}zier average, which are of great value in the context of generating curves. The first property shows that averaging two point-ntangent pairs is invariant under similarities appropriate for Hermite data. This essential property of the B\'{e}zier average follows from the invariance of B\'{e}zier curves under affine transformation.
\begin{lemma} [Similarity invariance]
\label{lemma: similarity invariance}
 The B\'{e}zier average is invariant under similarities in the following sense:
\begin{enumerate}
    \item The B\'{e}zier average is invariant under isometries applied to both $\mathbb{R}^{n}$ and $S^{n-1}$. That is, given an isometry 
    $\Phi:\mathbb{R}^{n}\longrightarrow\mathbb{R}^{n}$,
    \begin{equation}
    \label{similarity1}
        \tilde{\Phi}\left(B_\omega\left(\begin{pmatrix}
        p_0\\ v_0\\
        \end{pmatrix},\begin{pmatrix}
        p_1\\ v_1\\
        \end{pmatrix}\right)\right)=B_\omega\left(\tilde{\Phi}\begin{pmatrix}
        p_0\\ v_0\\
        \end{pmatrix},\tilde{\Phi}\begin{pmatrix}
        p_1\\ v_1\\
        \end{pmatrix}\right),
    \end{equation}
    where $\tilde{\Phi}\begin{pmatrix}
    p\\ v\\
    \end{pmatrix}:=\begin{pmatrix}
    \Phi(p)\\ \Phi(v)-\Phi(0)\\
    \end{pmatrix}$.
    \item The B\'{e}zier average is invariant under scaling: if $B_\omega\left(\begin{pmatrix}    p_0\\ v_0\\
        \end{pmatrix},\begin{pmatrix}
        p_1\\ v_1\\
        \end{pmatrix}\right)=\begin{pmatrix}
        p_\omega\\v_\omega\\
        \end{pmatrix}$, for any $a>0$, then,    \begin{equation}
        \label{similarity 2}
        B_\omega\left(\begin{pmatrix}
        a\cdot p_0\\ v_0\\
        \end{pmatrix},\begin{pmatrix}
        a\cdot p_1\\ v_1\\
        \end{pmatrix}\right)=\begin{pmatrix}
        a\cdot p_\omega\\ v_\omega\\
        \end{pmatrix}.
    \end{equation}
\end{enumerate}
\end{lemma}

The following two results show two important geometric properties of the average.
\begin{lemma}[Lines preservation]
\label{line reconstruction lemma}
    Let $ \begin{pmatrix}
    p_{0} \\ 
    v_{0} \\ 
    \end{pmatrix},\begin{pmatrix}
    p_{1} \\ 
    v_{1} \\ 
    \end{pmatrix}$ be two geometric Hermite samples from a straight line, $L_u$. 
    Then, the B\'{e}zier average coincides with the linear average of each component, namely, 
    \begin{equation}   \label{lines reconstraction}
    B_{\omega } \left( \begin{pmatrix}
    p_{0} \\ 
    v_0\\ 
    \end{pmatrix},\begin{pmatrix}
    p_{1} \\ 
    v_{1} \\ 
    \end{pmatrix} \right) =\begin{pmatrix}
    \left(1 - \omega\right)p_{0} + \omega p_{1} \\ 
    u \\ 
    \end{pmatrix}, \quad \omega \in [0,1].
    \end{equation}
\end{lemma}
Note that the average in~\eqref{lines reconstraction} is a point-ntangent pair 
from the line $L_u$. 

\begin{lemma}[Circles preservation]
\label{circle reconstruction lemma}
    Assume $\begin{pmatrix}
    p_{0} \\ 
    v_{0} \\ 
    \end{pmatrix},\begin{pmatrix}
    p_{1} \\ 
    v_{1} \\ 
    \end{pmatrix}$
    are samples of a circle and its ntangents. Then, their average at $\omega=\frac{1}{2}$ consists of the mid-point and its ntangent of the circle-arc, defined by $p_{0}$, $p_{1}$ and $v_{0}$, $v_{1}$.
\end{lemma}

\section{Subdivision schemes based on the B\'{e}zier Average}
\label{sec: subdivision}

 In the previous section we describe the concept of Hermite average and introduce such an average --- the B\'{e}zier Average. We also present some of the main properties of the B\'{e}zier average. This essential building block is the first step in constructing curves' approximations. The next step is to use an approximation operator, which can be defined in terms of binary averages, and replace the binary average by a Hermite average. As approximation operators, we use subdivision schemes, which enjoy simple implementation and are used extensively for modeling, approximation, and multiscale representations of curves. 

As an illustrative method, the main subdivision scheme we utilize is the interpolatory two-point scheme. This simple method demonstrates the application of the B\'{e}zier average and how the average's properties are inherited by the approximation process. We define the new refinement rules and prove the convergence of this subdivision scheme for any initial admissible data. 

\subsection{From Hermite average to Hermite subdivision schemes}
\label{subsec: IHB}

We start with the following lemma, which, together with Lemma~\ref{non vanishing vectors}, asserts that the process of midpoint averaging can be performed repeatedly.

\begin{lemma}
\label{lemma: reaveraging}
Assume the pair $\begin{pmatrix}
\begin{pmatrix}
p_{0} \\ 
v_{0} \\ 
\end{pmatrix},\begin{pmatrix}
p_{1} \\ 
v_{1} \\ 
\end{pmatrix}
\end{pmatrix}$ satisfies conditions~\eqref{eqn:conditions_on_pts}. Then, in the notations of Definition~\ref{def: average def}, the two pairs
$\left(\begin{pmatrix}
p_0\\v_0\\
\end{pmatrix},\begin{pmatrix}
p_{\frac{1}{2}}\\v_{\frac{1}{2}}\\
\end{pmatrix}\right)$ and $\left(\begin{pmatrix}
p_{\frac{1}{2}}\\v_{\frac{1}{2}}\\
\end{pmatrix},\begin{pmatrix}
p_1\\v_1\\
\end{pmatrix}\right)$ also satisfy conditions~\eqref{eqn:conditions_on_pts}.
\end{lemma}
The proof of Lemma~\ref{lemma: reaveraging} is given in Appendix~\ref{proofs}.

Lemma~\ref{lemma: reaveraging} allows us to design a Hermite interpolatory subdivision scheme utilizing the B\'{e}zier average with $\omega=\frac{1}{2}$ as an insertion rule. Namely, for initial Hermite data of the form $\begin{pmatrix}
p_{j}^{0} \\ 
v_{j}^{0} \\ 
\end{pmatrix}$, ${j \in \mathbb{Z}} $, we define for all $k=0,1,2,\ldots$:
\begin{equation} \label{eqn:IHB_scheme}
\begin{pmatrix}
p_{2j}^{k + 1} \\ 
v_{2j}^{k + 1} \\ 
\end{pmatrix} =\begin{pmatrix}
p_{j}^{k} \\ 
v_{j}^{k} \\ 
\end{pmatrix}    , \quad   \begin{pmatrix}
p_{2j + 1}^{k + 1} \\ 
v_{2j + 1}^{k + 1} \\ 
\end{pmatrix} = B_{\frac{1}{2}}\begin{pmatrix}
    \begin{pmatrix}
        p_{j}^{k} \\ 
        v_{j}^{k} \\ 
    \end{pmatrix},\begin{pmatrix}
        p_{j + 1}^{k} \\ 
        v_{j + 1}^{k} \\ 
    \end{pmatrix}
\end{pmatrix}, \quad  j\in \mathbb{Z}.
\end{equation}
We refer to the subdivision scheme~\eqref{eqn:IHB_scheme} as the interpolatory Hermite-B\'{e}zier scheme, in short \textit{IHB-scheme}. 

We view the IHB scheme as a modification of the piecewise linear subdivision scheme, which is also known as Lane-Riesenfeld of order $1$ subdivision scheme (LR1). This scheme serves as an illustrating example for our general method where we replace the linear average in any LR scheme of order $m\ge 1$, by the B\'{e}zier average with $\omega=\frac{1}{2}$. The modified scheme refines Hermite data and we term it Hermite-B\'{e}zier Lane-Riesenfeld of order $m$ (HB-LRm). The refinement step of this scheme is given in Algorithm~\ref{alg:HB_LRm}. 

For the HB-LRm scheme with $m>1$, we obtain a subdivision scheme which is not interpolatory, namely, the interpolation conditions~\eqref{interpolation} are not fulfilled. However, we expect the HB-LRm schemes to approximate the curve, as seen visually in the numerical examples in Section~\ref{sec:examples}. 

\begin{algorithm}[ht]
\caption{The refinement step of the Hermite-B\'{e}zier Lane-Riesenfeld of order $m$}
\label{alg:HB_LRm}
\begin{algorithmic}[1]
\REQUIRE The Hermite data to be refined $ \{ \left( p_i, v_i \right) \}_{i \in \mathbb{Z}}$. The order $m$.
\ENSURE The refined data.
\FOR{$i \in \mathbb{Z}$}
\STATE  $Q_{2i}^{0}  \gets \begin{pmatrix}
        p_{i} \\ 
        v_{i} \\ 
    \end{pmatrix} $    
\STATE  $Q_{2i+1}^{0}  \gets \begin{pmatrix}
        p_{i} \\ 
        v_{i} \\ 
    \end{pmatrix} $  
\ENDFOR
\FOR{$j=1$ \TO $m$}    
\FOR{$i \in \mathbb{Z}$}
\STATE $ Q_{2i}^{j} \gets Q_{i}^{j-1} $ 
\STATE $Q_{2i+1}^{j} \gets B_\frac{1}{2}\left(Q_{i}^{j-1},Q_{i+1}^{j-1}\right) $ 
\ENDFOR
\ENDFOR \\
\RETURN $\{Q_{i}^{m} \}_{i \in \mathbb{Z}}$
\end{algorithmic}
\end{algorithm}

As a first demonstration, we show  the repeated refinements of two elements in $\mathbb{R}^{n}\times S^{n-1}$ by the IHB-scheme, as appear on Figure~\ref{fig:average inserting} and Figure~\ref{fig:3D average inserting} for $n=2$ and $n=3$, respectively. These figures visually reveal some properties, in particular convergence of the IHB scheme, which we present and show in the sequel. 

\begin{figure}
    \centering
    \includegraphics[width=0.19\textwidth]{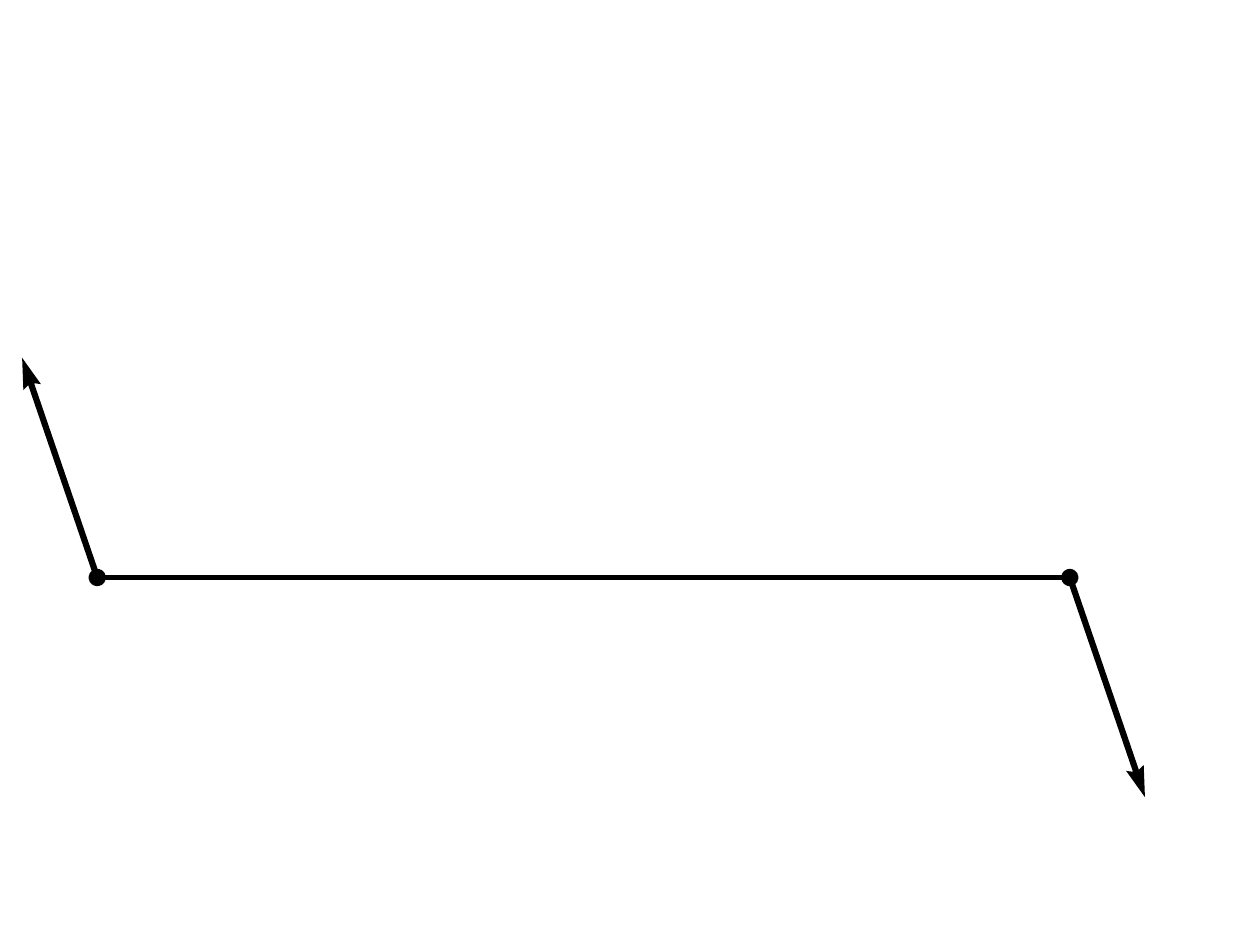}\hfill
    \includegraphics[width=0.19\textwidth]{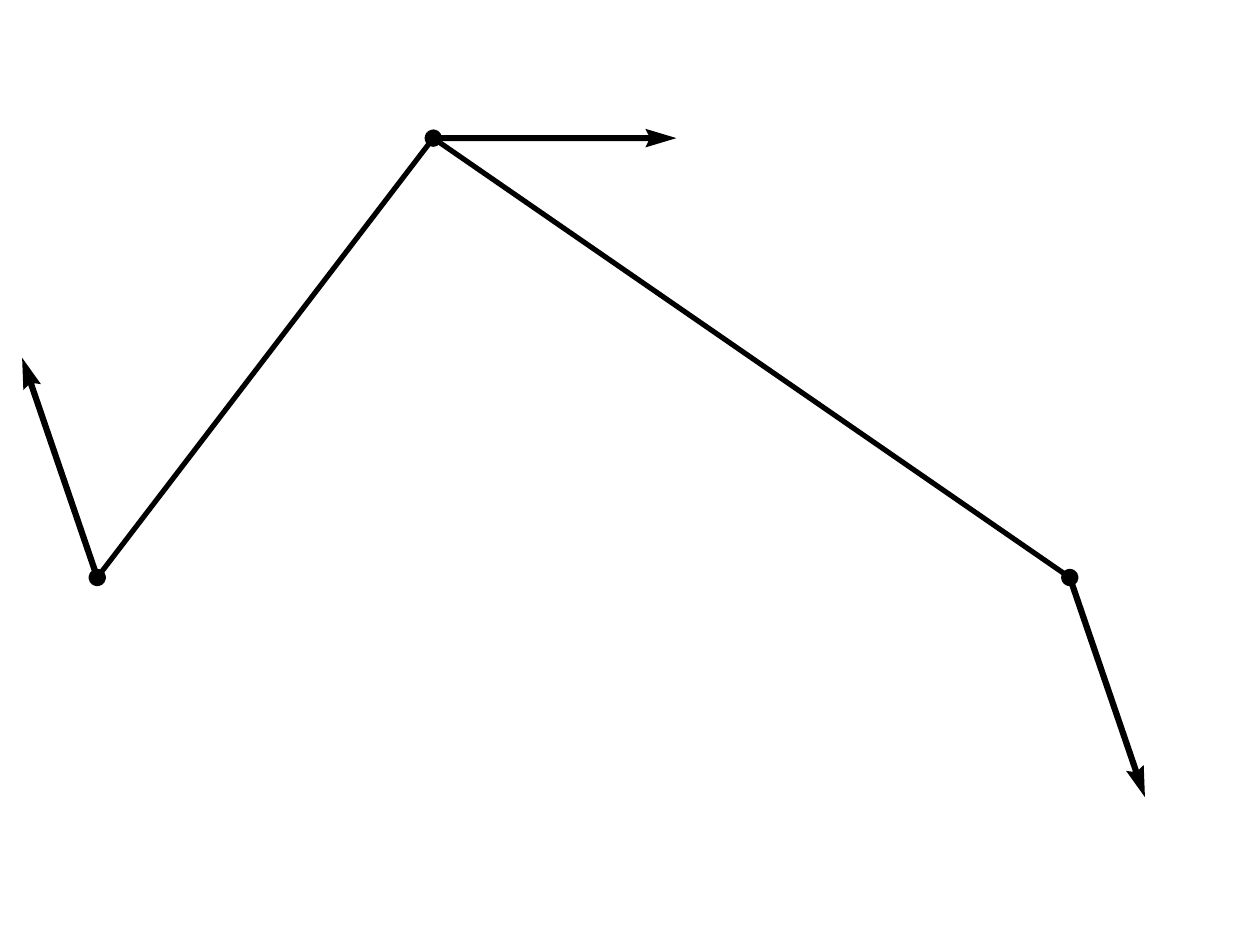}\hfill
    \includegraphics[width=0.19\textwidth]{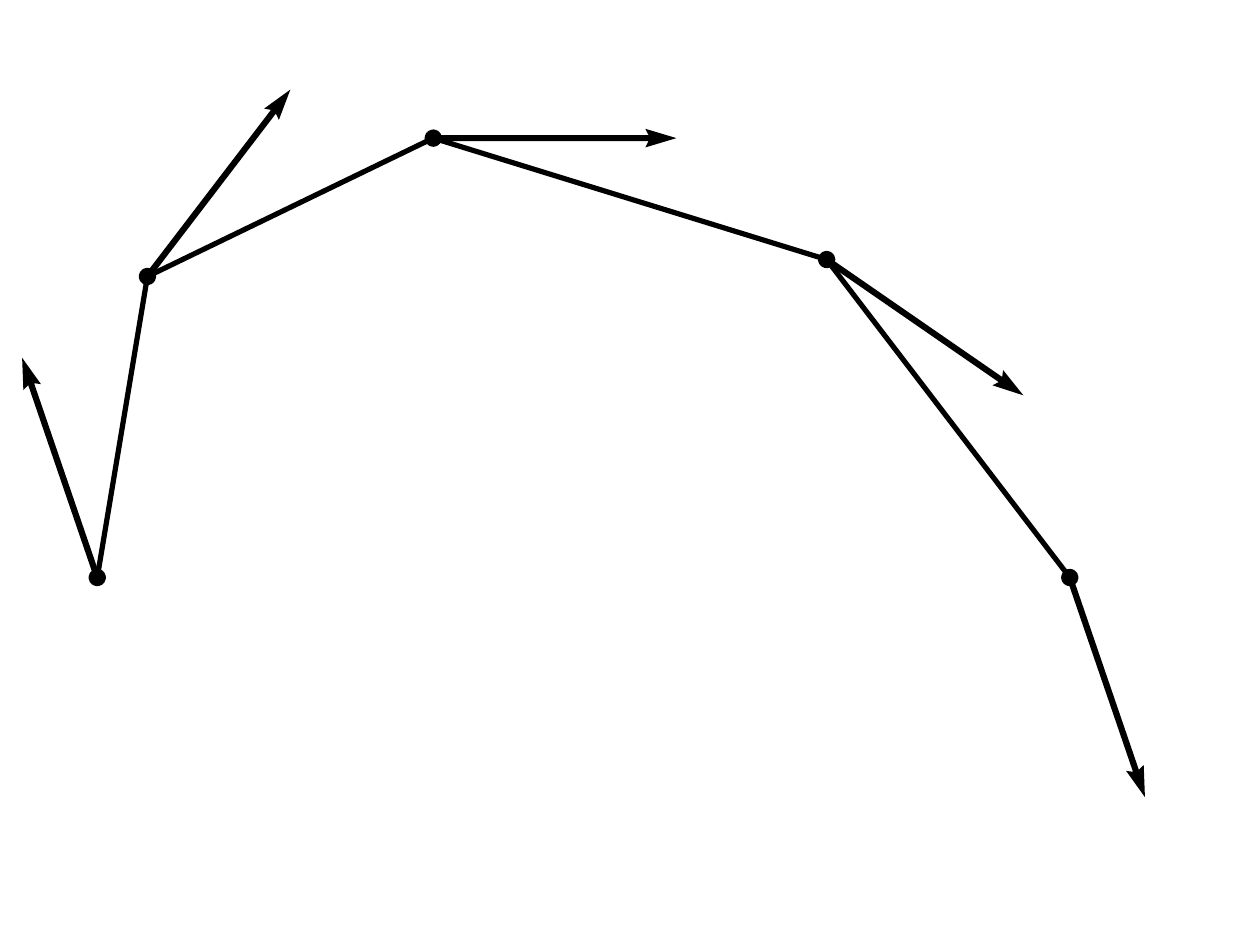}\hfill
    \includegraphics[width=0.19\textwidth]{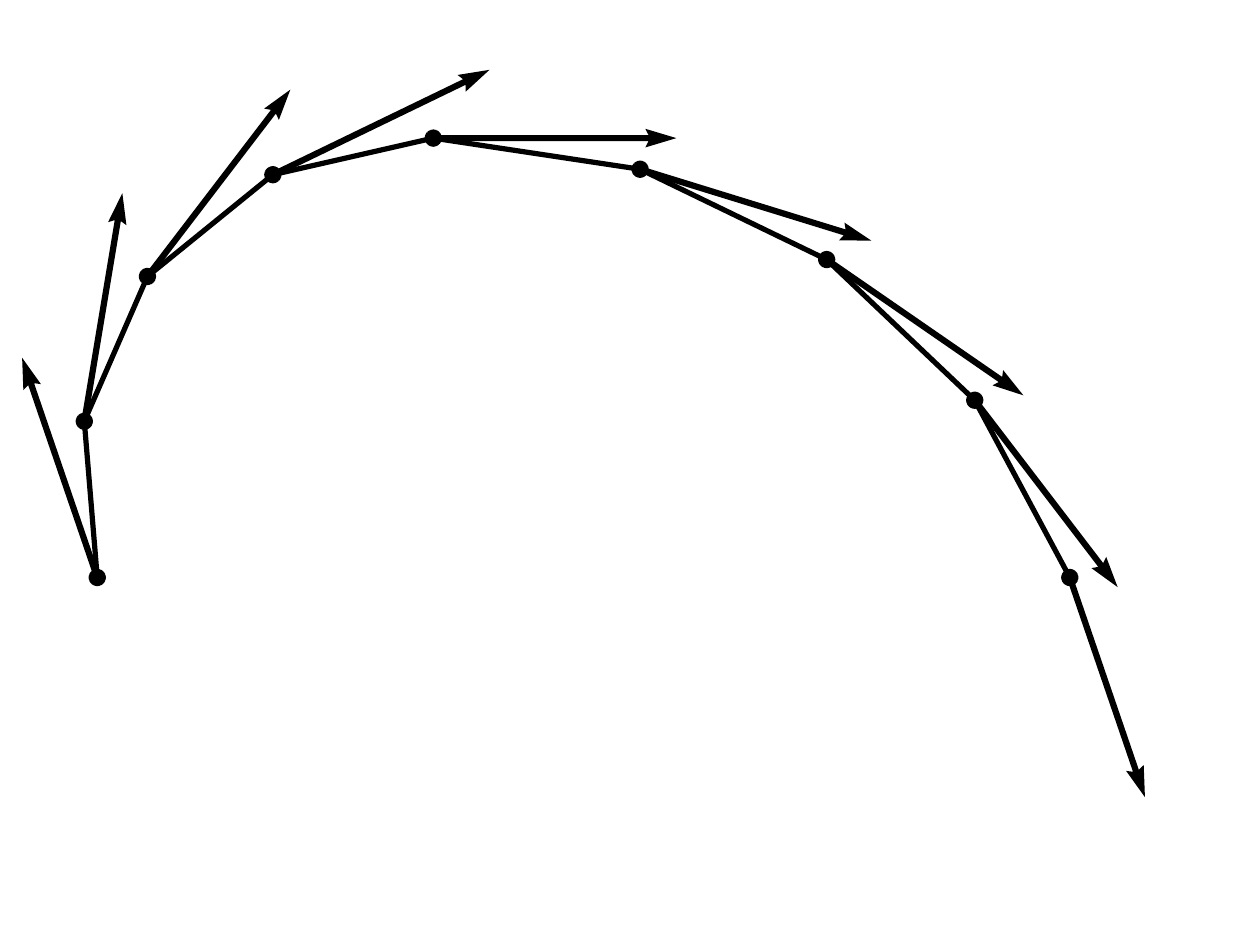}\hfill
    \includegraphics[width=0.19\textwidth]{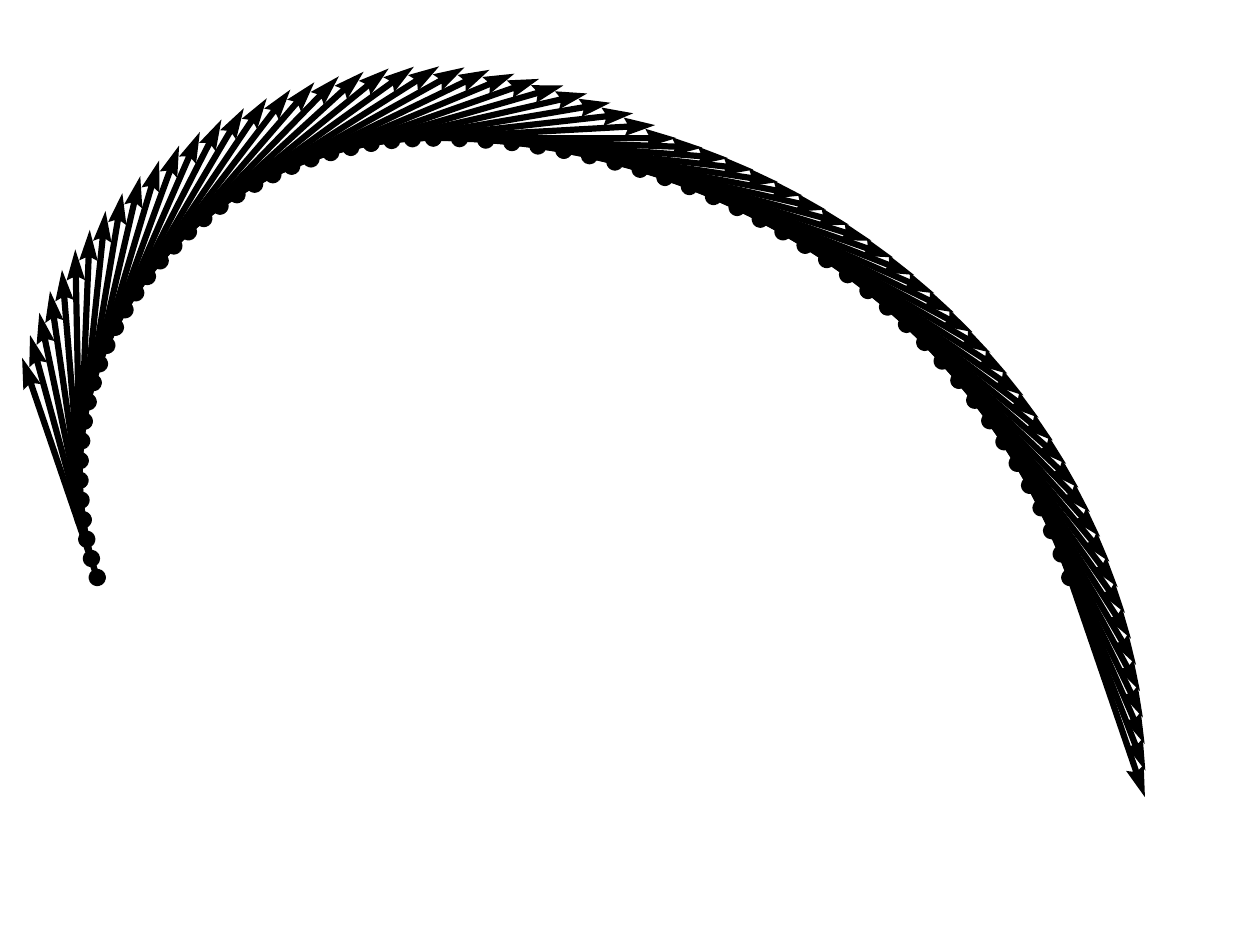}
    \includegraphics[width=0.19\textwidth]{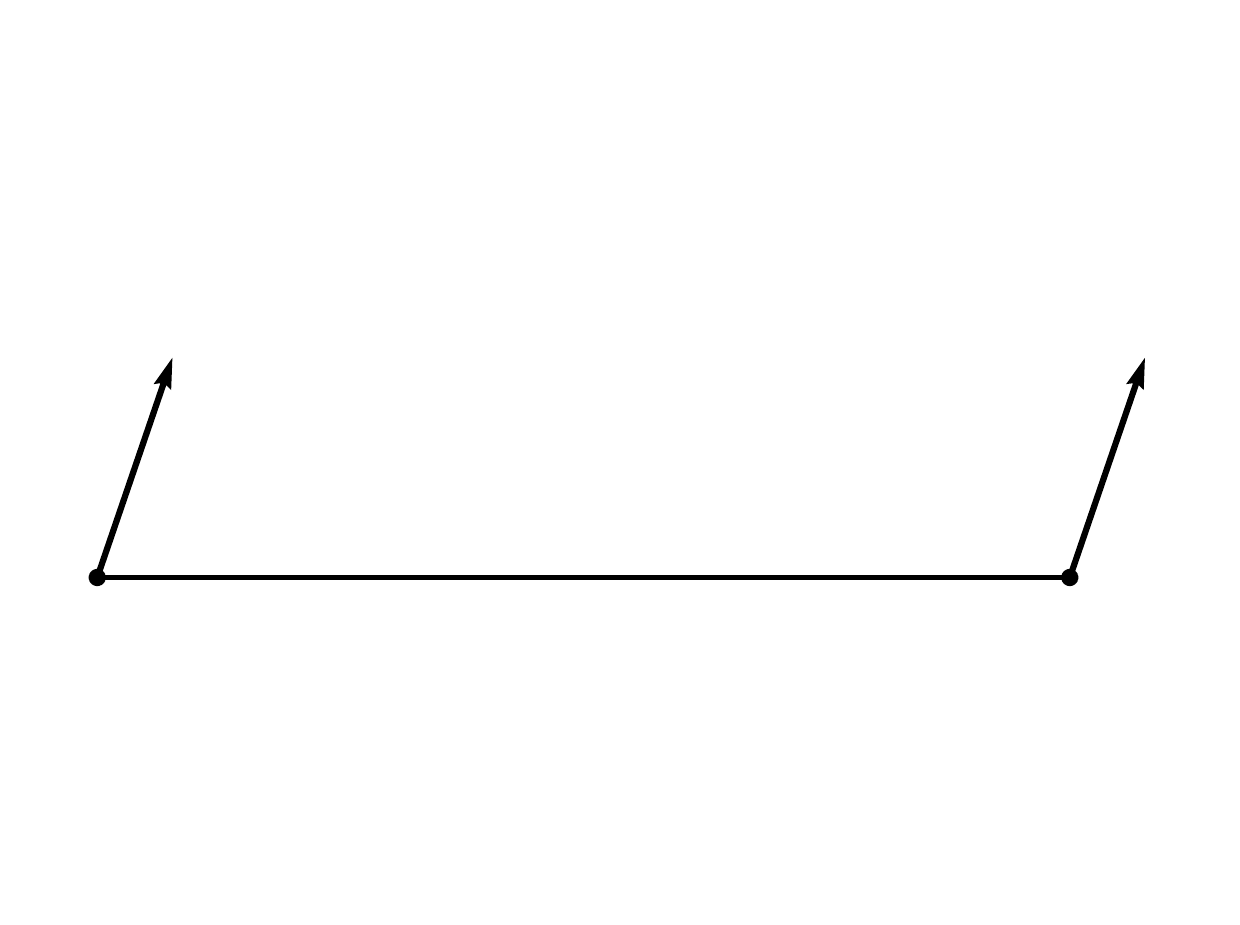}\hfill
    \includegraphics[width=0.19\textwidth]{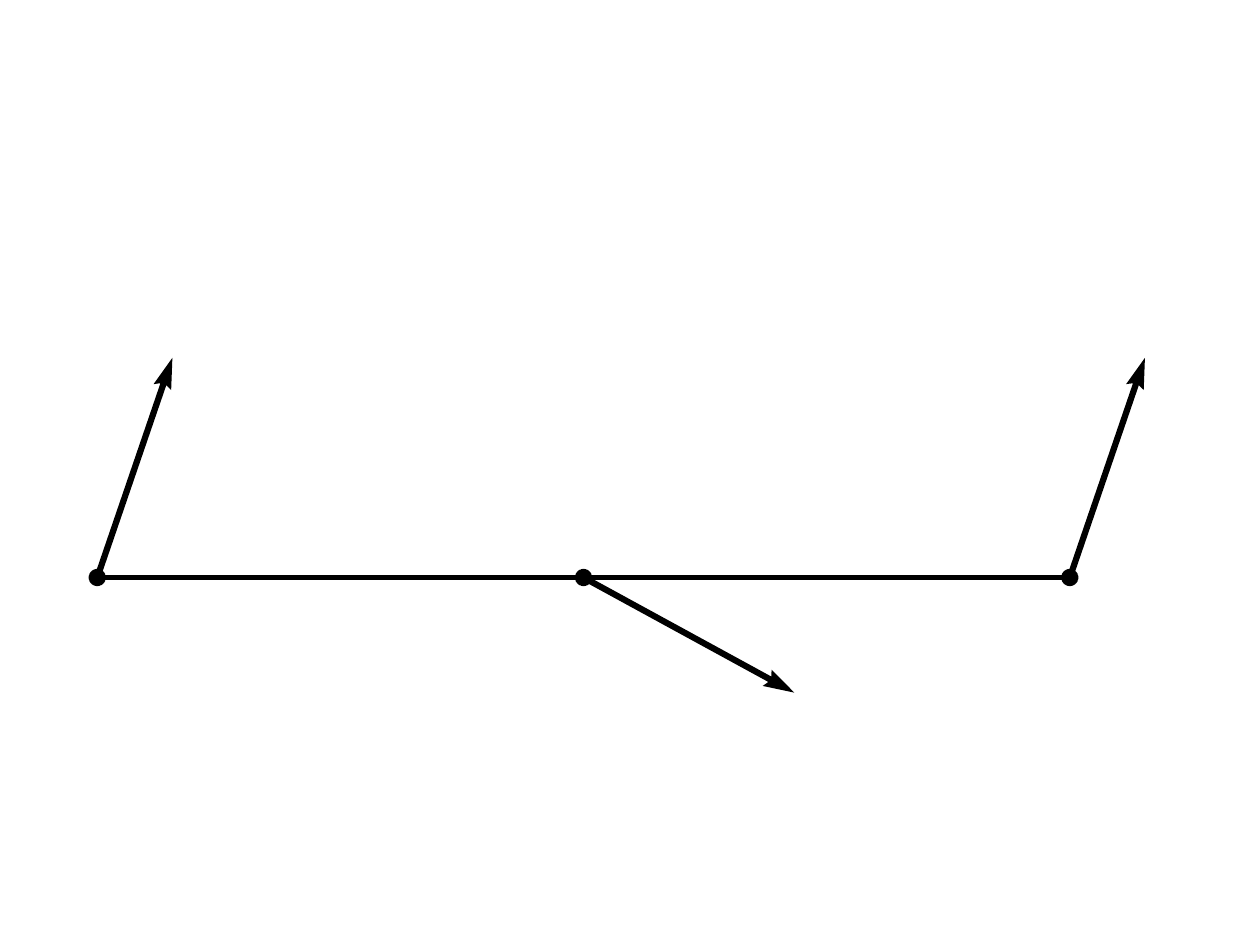}\hfill
    \includegraphics[width=0.19\textwidth]{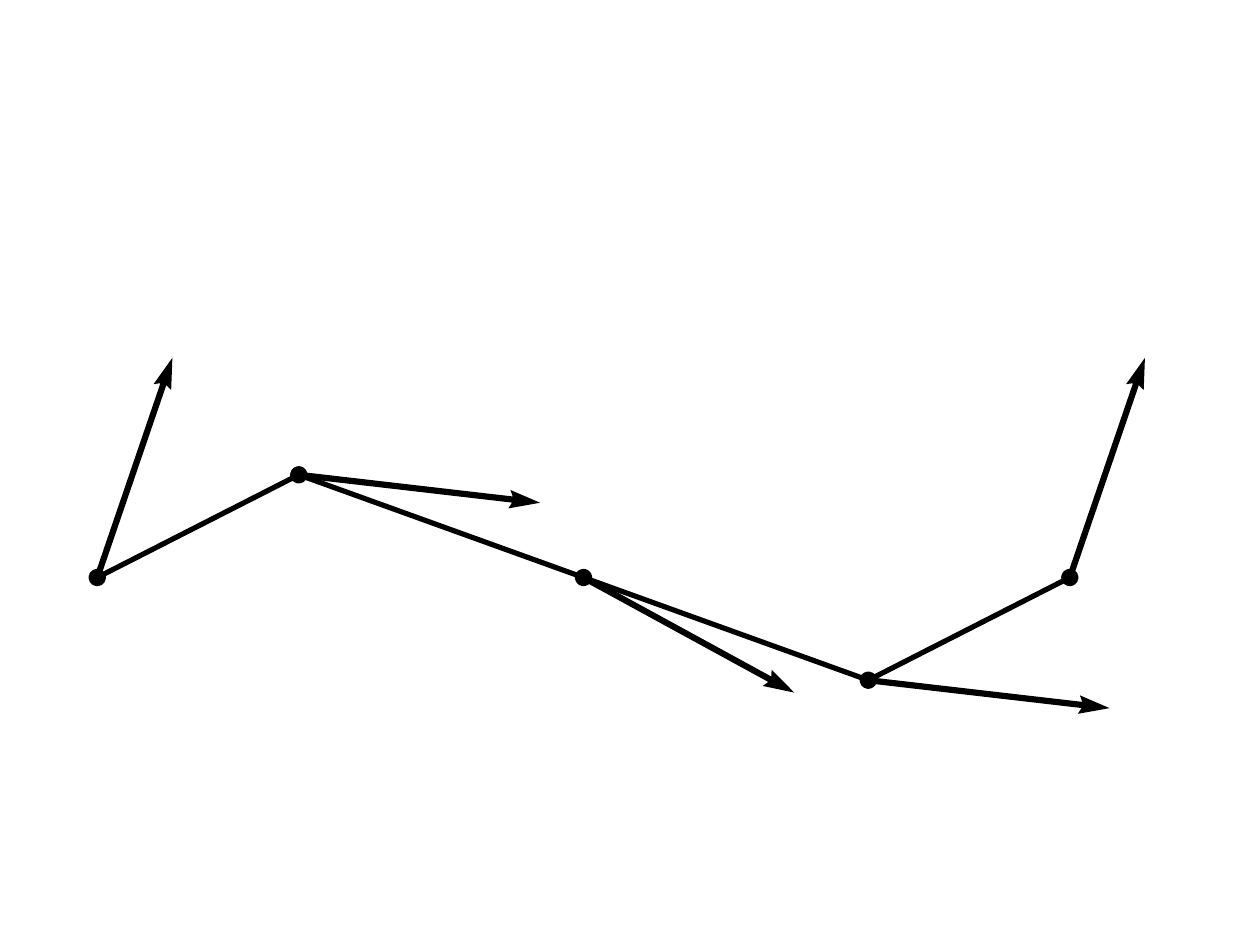}\hfill
    \includegraphics[width=0.19\textwidth]{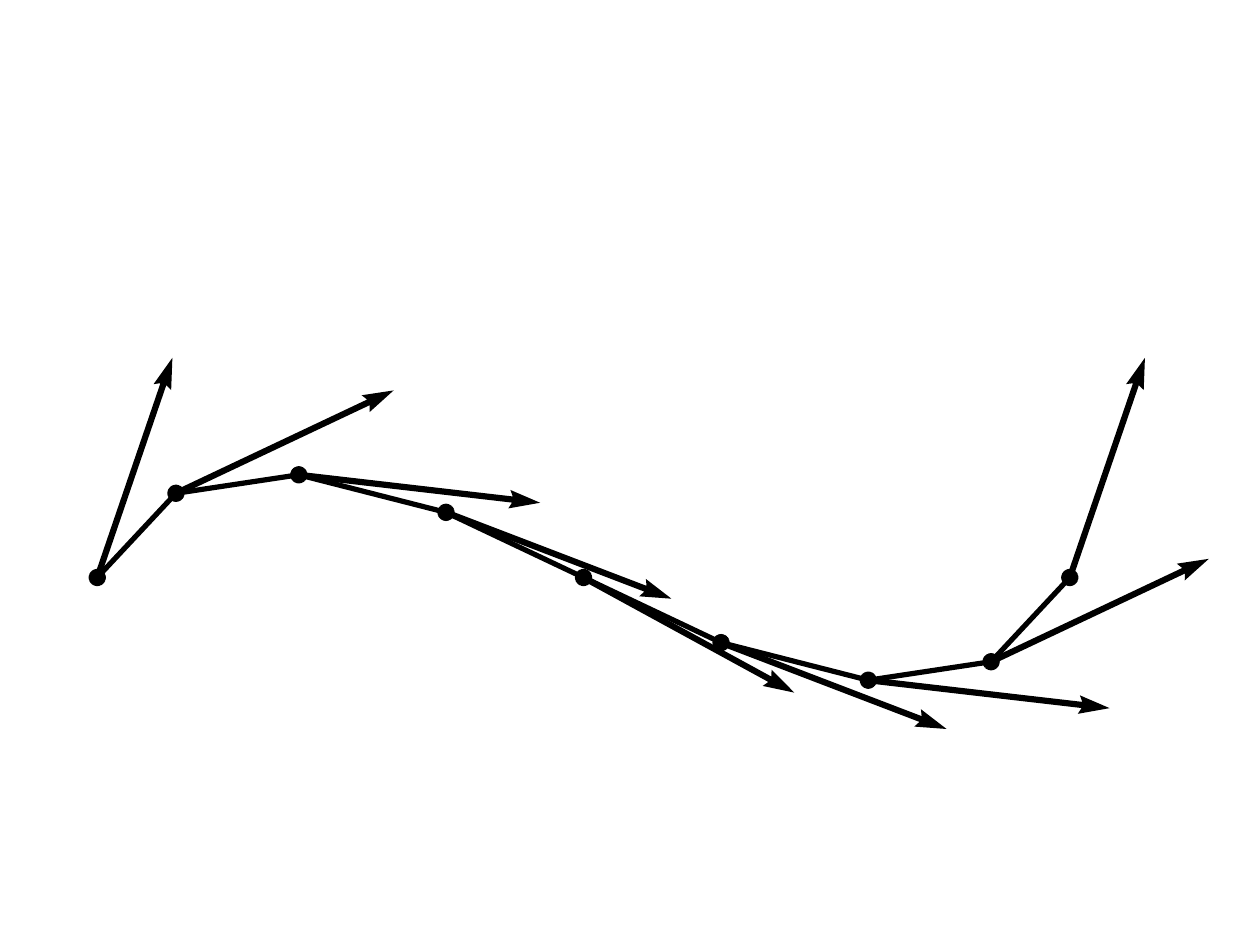}\hfill
    \includegraphics[width=0.19\textwidth]{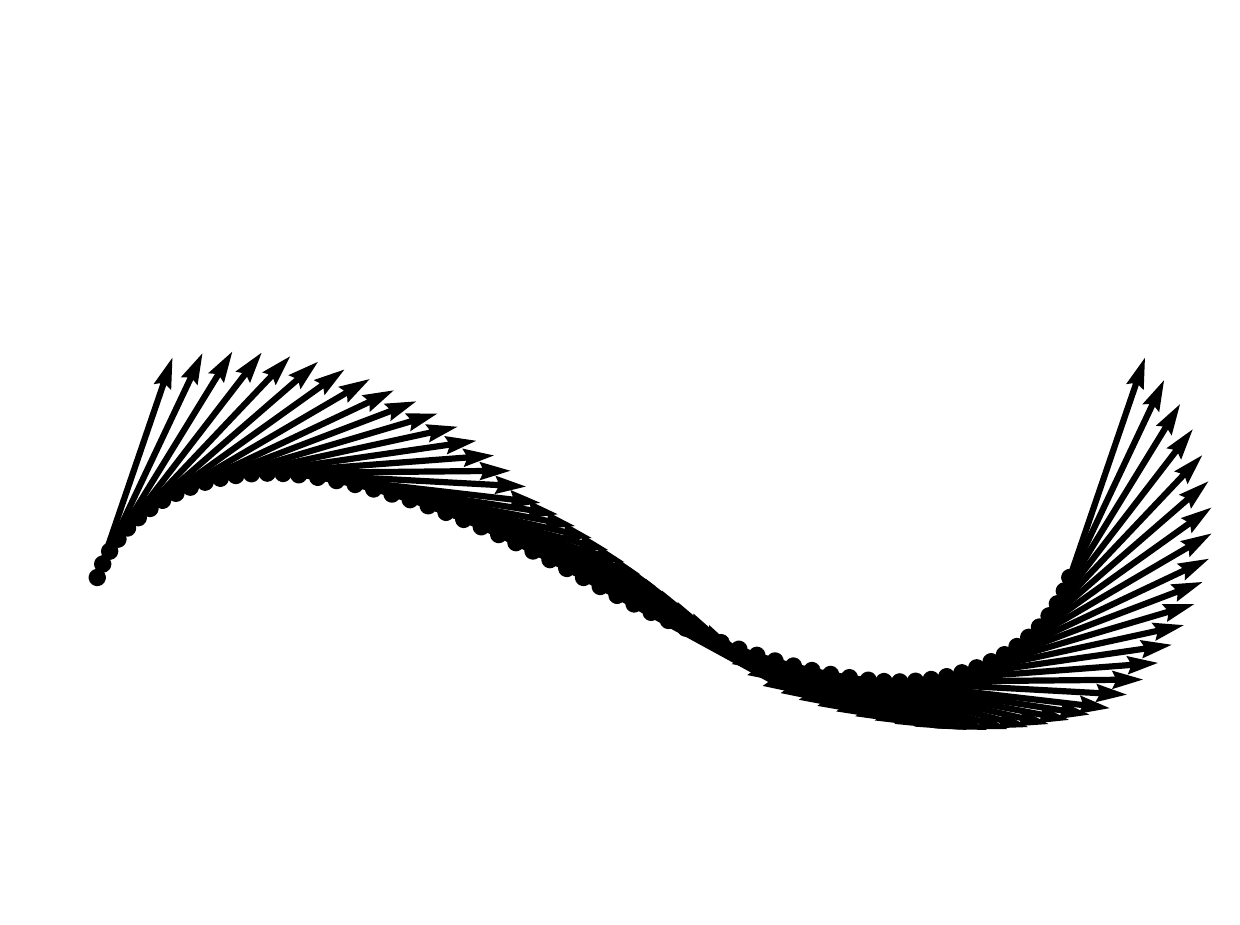}
    \caption{Two examples of applying the IHB-scheme in $\mathbb{R}^{2}$. From left to right: initial data, one iteration, two iterations, three iterations, six iterations.}
    \label{fig:average inserting}
\end{figure}

\begin{figure}
    \centering
    \begin{subfigure}[t]{0.28\textwidth}
        \includegraphics[width=\textwidth]{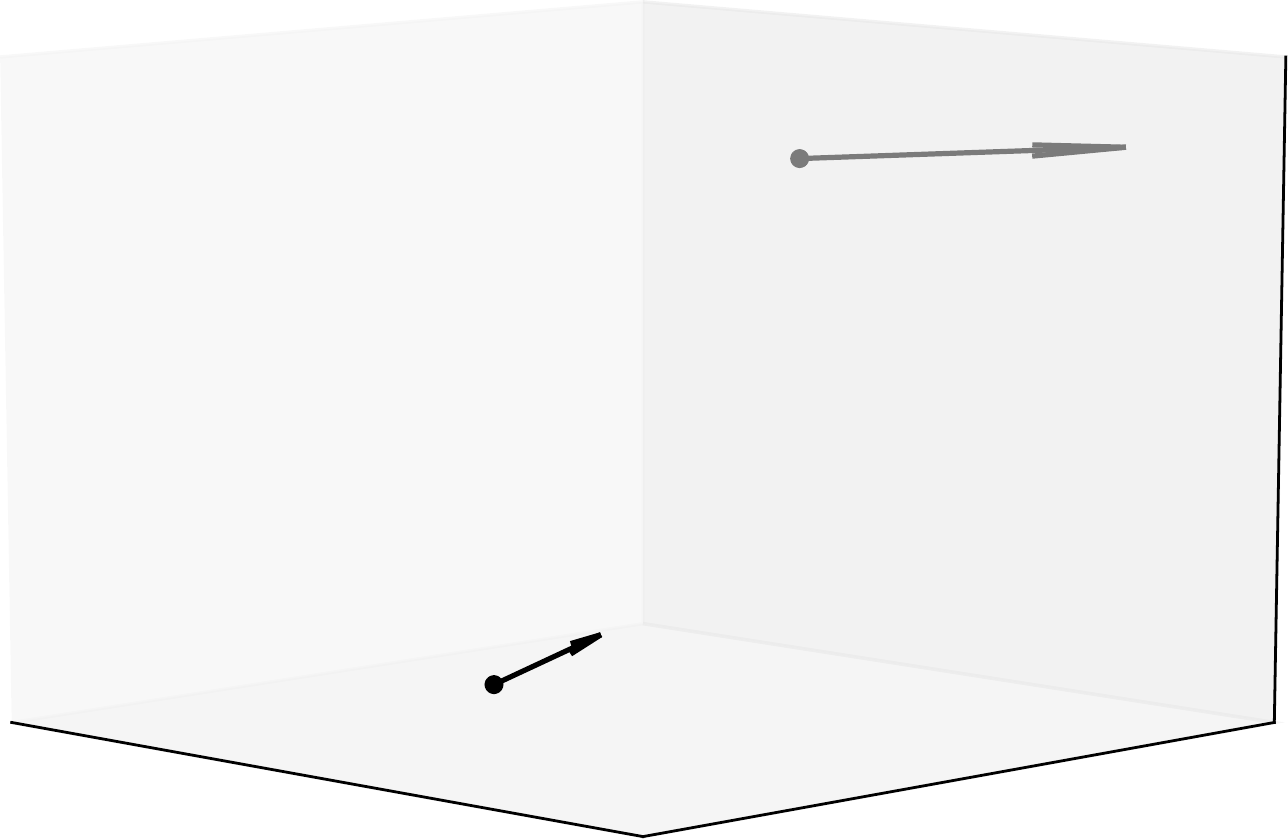}
        \caption{initial data}
        \end{subfigure}
        \quad\quad
    \begin{subfigure}[t]{0.28\textwidth}
        \includegraphics[width=\textwidth]{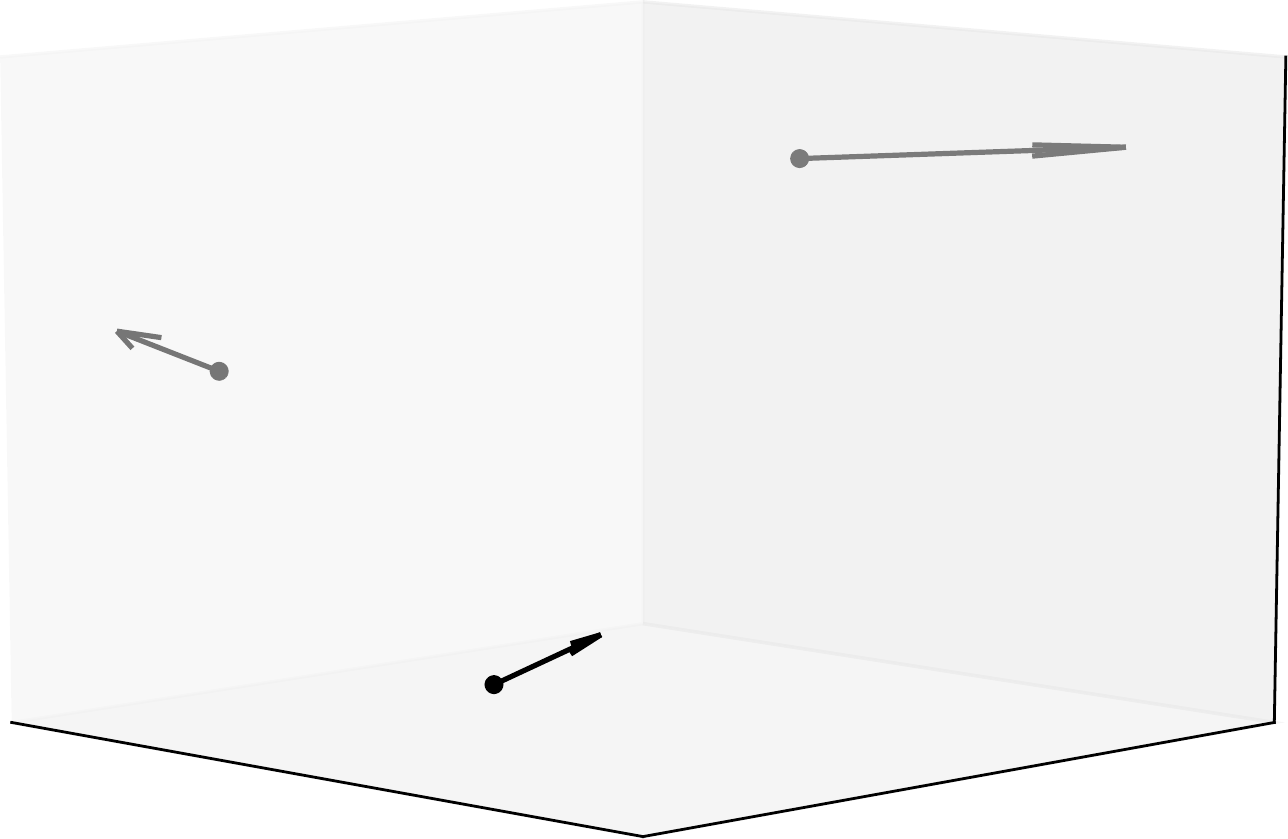}
        \caption{1 iteration}
        \end{subfigure}
        \quad\quad
    \begin{subfigure}[t]{0.28\textwidth}
        \includegraphics[width=\textwidth]{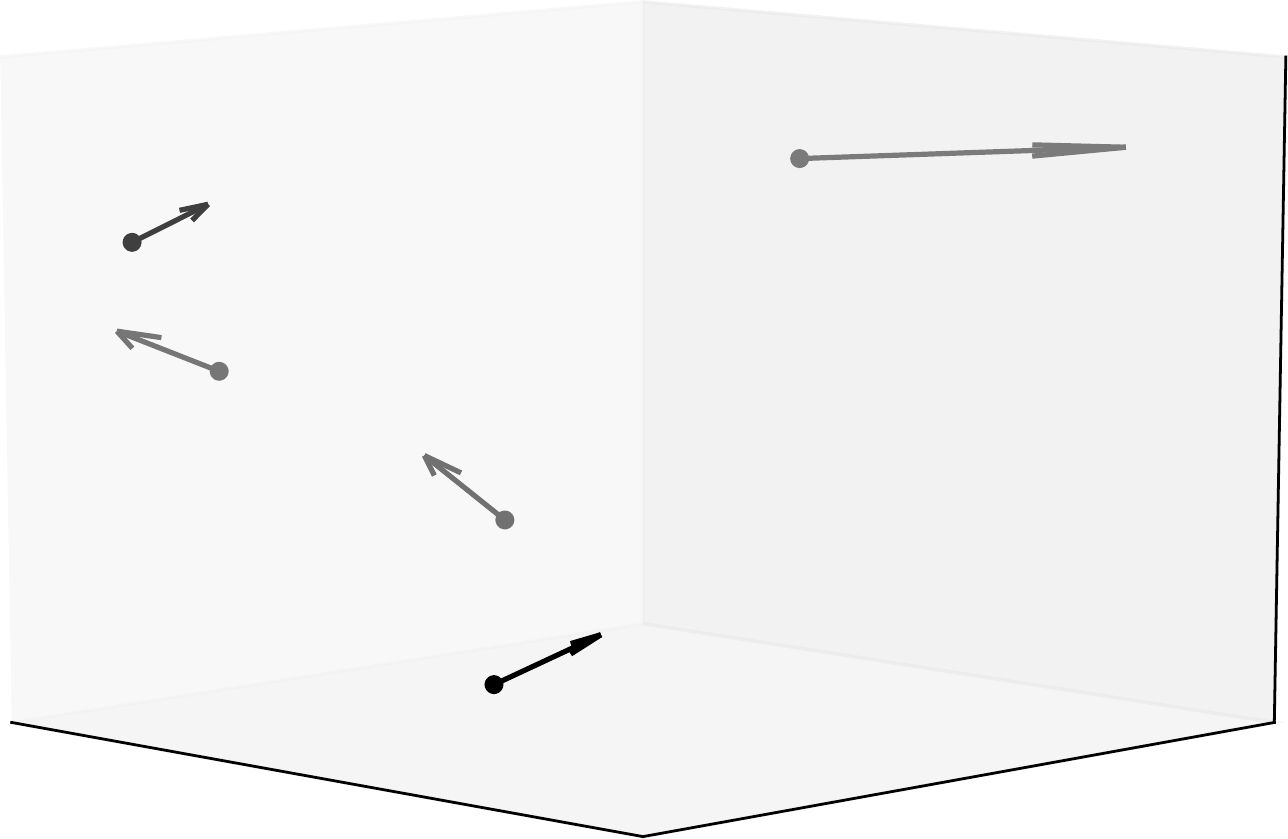}
        \caption{2 iterations}
        \end{subfigure}
    \begin{subfigure}[t]{0.28\textwidth}
        \includegraphics[width=\textwidth]{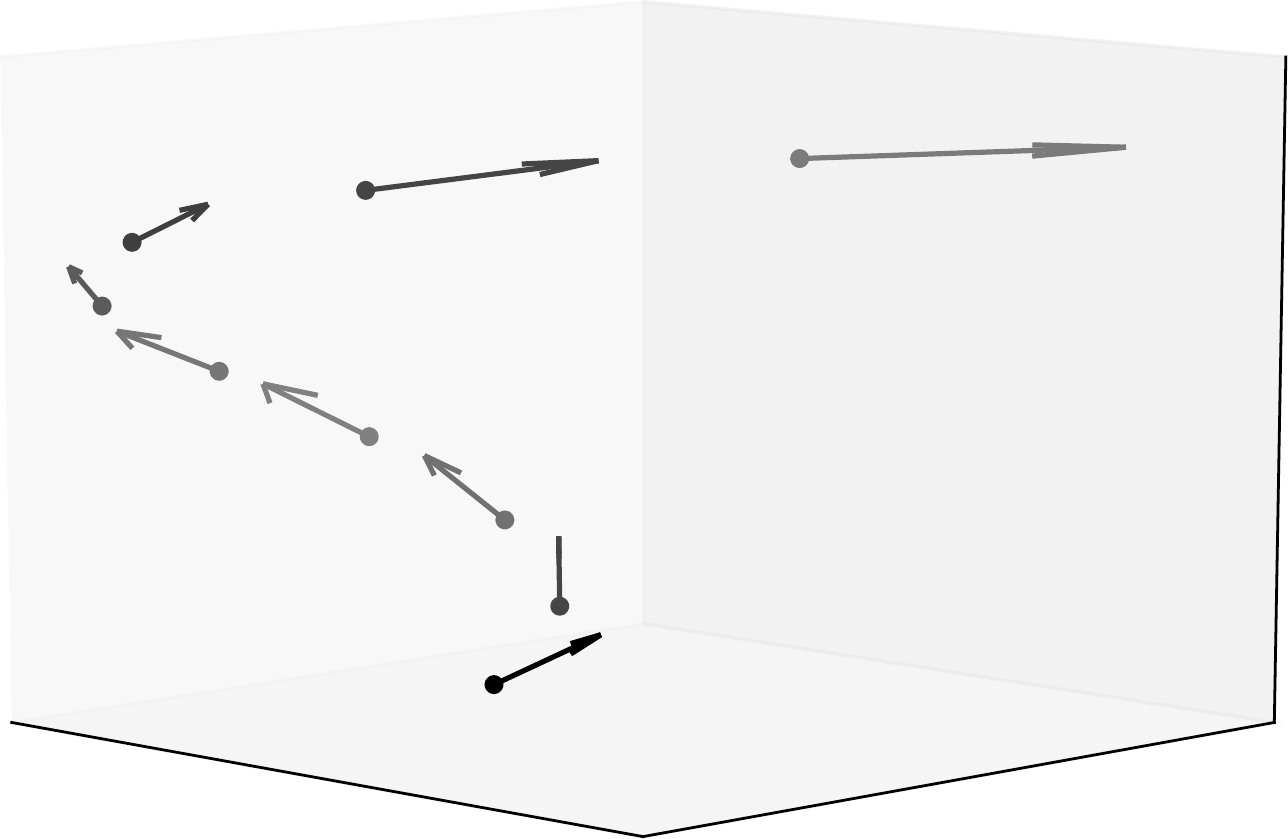}
        \caption{3 iterations}
        \end{subfigure}\quad\quad
    \begin{subfigure}[t]{0.28\textwidth}
        \includegraphics[width=\textwidth]{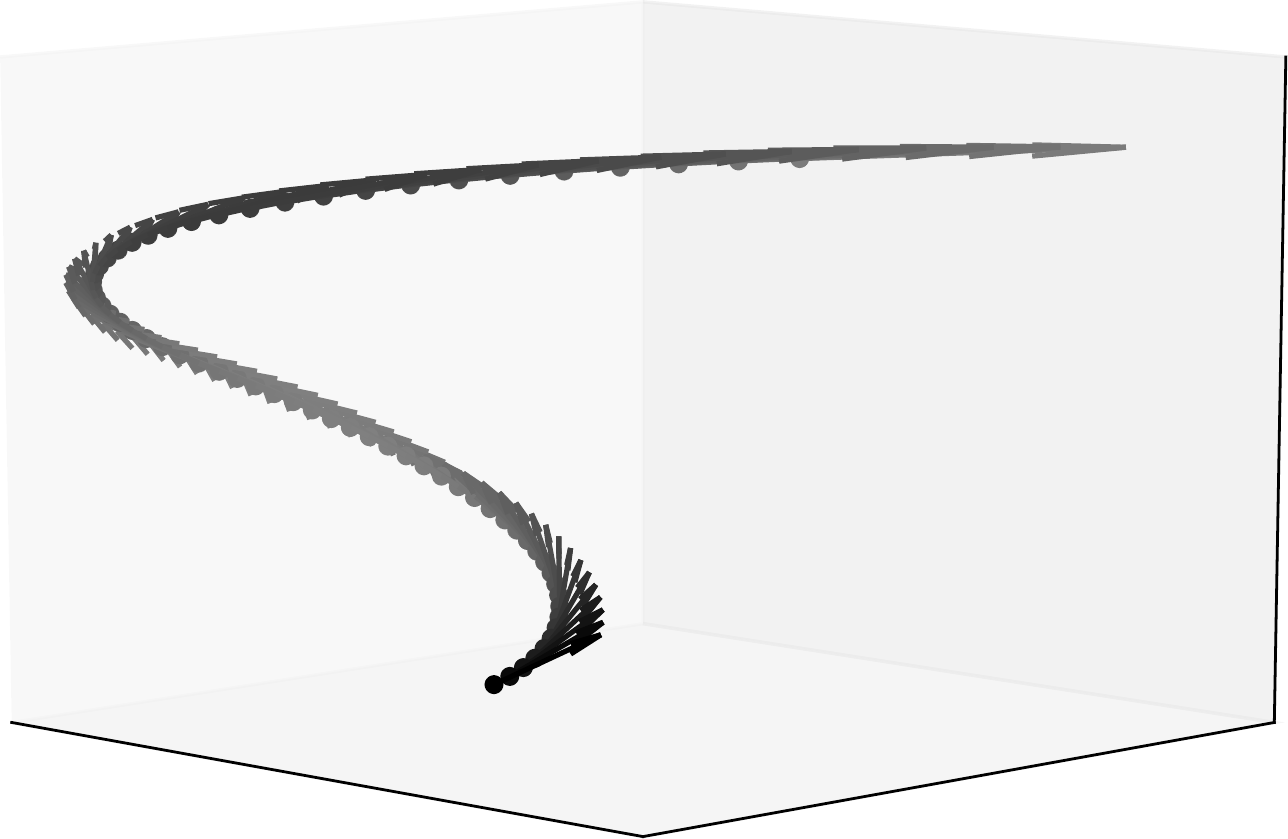}
        \caption{6 iterations}
    \end{subfigure}
    \caption{Example of the performance of the IHB-scheme in $\mathbb{R}^{3}$.}
    \label{fig:3D average inserting}
\end{figure}

Next, we present geometric properties of the HB-LRm schemes with $m \ge 1$ which are a direct consequence of geomeric properties of the B\'{e}zier average, as given in Lemma~\ref{lemma: similarity invariance}, Lemma~\ref{line reconstruction lemma}, and Lemma~\ref{circle reconstruction lemma}. In particular, the first part of the following result shows that preservation of geometric objects by the average becomes reconstruction of these objects by the modified subdivision schemes. 
\begin{thm}
The HB-LRm schemes of order $m \ge 1$ reconstruct lines and circles. For converging HB-LRm schemes, their limits are invariant under isometries and scaling transformations.
\end{thm}
From here and until the end of this section, we focus our attention on the IHB-scheme (HB-LR1). Next, we present auxiliary results on the B\'{e}zier average and then present the convergence analysis of the IHB-scheme.

\subsection{Contractivity of the B\'{e}zier Average}

Let $\left(\begin{pmatrix}
p_0\\v_0\\
\end{pmatrix},\begin{pmatrix}
p_1\\v_1\\
\end{pmatrix}\right)\in\left(\mathbb{R}^n\times S^{n-1}\right)^2$ be given. As in the previous section, we always assume that $p_0\neq p_1$ and we denote by $p_\frac{1}{2}$ and $v_\frac{1}{2}$ the point and vector obtained by the B\'{e}zier average with weight $\frac{1}{2}$ applied to the given data. The following two lemmas play a central role in the convergence analysis of the IHB-scheme.

\begin{lemma} 
\label{points distance contraction lemma}
If $ \theta_{0} + \theta_{1}\leq \pi$ then $d\left(p_{\frac{1}{2}},p_{j}\right)\leq d\left(p_{0},p_{1}\right), \quad j = 0,1$.

Furthermore, if $ \theta_{0} + \theta_{1}<\pi$ then there exists $\mu \in (0,1)$  such that $ d\left(p_{\frac{1}{2}},p_{j}\right)\leq \mu d\left(p_{0},p_{1}\right)$.
\end{lemma}
The proof of Lemma~\ref{points distance contraction lemma} is given in Appendix~\ref{proofs}.

\begin{lemma}
\label{angles contraction lemma}
If $ \sigma\begin{pmatrix}
\begin{pmatrix}
p_{0} \\ 
v_{0} \\ 
\end{pmatrix},\begin{pmatrix}
p_{1} \\ 
v_{1} \\ 
\end{pmatrix}
\end{pmatrix}\leq\frac{3\pi }{4}$. Then,  
\[ \max \Bigg\{ \sigma\begin{pmatrix}
\begin{pmatrix}
p_{0} \\ 
v_{0} \\ 
\end{pmatrix},\begin{pmatrix}
p_{\frac{1}{2}} \\ 
v_{\frac{1}{2}} \\ 
\end{pmatrix}
\end{pmatrix},\sigma\begin{pmatrix}
\begin{pmatrix}
p_{\frac{1}{2}} \\ 
v_{\frac{1}{2}} \\ 
\end{pmatrix},\begin{pmatrix}
p_{1} \\ 
v_{1} \\ 
\end{pmatrix}
\end{pmatrix} \Bigg\} \leq \sqrt{0.9}\sigma\begin{pmatrix}
\begin{pmatrix}
p_{0} \\ 
v_{0} \\ 
\end{pmatrix},\begin{pmatrix}
p_{1} \\ 
v_{1} \\ 
\end{pmatrix}
\end{pmatrix}.\]
\end{lemma} 
 We do not provide a formal  proof of Lemma~\ref{angles contraction lemma}. Instead, we give in Appendix~\ref{app:sigma proof} a comprehensive discussion where we show that proving the lemma is equivalent to showing the positivity of an explicit trigonometric function. Then, we introduce numerical indications for the positivity of this function and an outline for a proof based on a further exhaustive computer search.

\begin{remark}
The above two lemmas shed light on our exploration of a natural metric; see, e.g., Remark~\ref{metric}. On the one hand, Lemma~\ref{points distance contraction lemma} guarantees the boundedness of the Euclidean metric over a subset of $\left(\mathbb{R}^{n}\times S^{n - 1}\right)^2$. On the other hand, even though $ \sigma :X \to [0,\sqrt{2}\pi]$ is not a metric over $\mathbb{R}^{n}\times S^{n - 1}$, it gives a measure of how far an ordered couple of point-vector pairs is from being sampled from a line. It is the result of Lemma~\ref{angles contraction lemma} that the B\'{e}zier average yields two pairs that are ``closer" to samples from a line than the original pair.

In addition, a metric $\tilde{d}$ that satisfies the so-called metric property,
\[ \tilde{d}\left(\begin{pmatrix}
    p_{j} \\ 
    v_{j} \\ 
    \end{pmatrix},\begin{pmatrix}
p_{\frac{1}{2} } \\ 
v_{\frac{1}{2} } \\ 
\end{pmatrix}\right)= \frac{1}{2} \tilde{d}\left(\begin{pmatrix}
    p_{0} \\ 
    v_{0} \\ 
    \end{pmatrix}, \begin{pmatrix}
    p_{1} \\ 
    v_{1} \\ 
    \end{pmatrix}\right), \quad j=0,1. \]
would lead to the convergence of the HB-LRm subdivision schemes of any order $m \ge 1$ (see Algorithm~\ref{alg:HB_LRm}), which can be shown by the technique in~\cite{dyn2017global}.
\end{remark}

\subsection{G1 Convergence of the IHB-scheme}

First, we recall several definitions and notation needed in the convergence analysis.  Initial control points $ P^{0} = \left\{ p_{j}^{0}\right\}_{j\mathbb{\in Z}}   $ is a sequence of points in $\mathbb{R}^d$. A subdivision scheme refining points $ \mathcal{S}$ refines $P^{0} $ and generate the control points $\mathcal{S}^k(P^0)= P^{k} = \left\{ p_{j}^{k}\right\}_{j\mathbb{\in Z}}$, $k>0$. Associating the point $p^k_j$ with the parameter value  $ t^k_j=2^{-k}j$, we define at each level the piecewise linear interpolant to the points $(t^k_j,p^k_j),\ j\in \mathbb Z$, $f_k$ (also known as the \textit{control polygon at level $k$}).  With these notions we can define the convergence of $\mathcal S.$ We term $\mathcal{S}$ convergent if the sequence $\left\{ f_{k}\right\}_{k\geq 0}$ converges uniformly in the $ L_{\infty }$ norm. In addition, we follow the definition given in~\cite{dyn2012geometric} regarding G1 convergence of a subdivision scheme refining points, and say that $\mathcal{S}$ is G1 convergent if it is convergent and there exists a continuously varying directed tangent along its limit curve. 

A Hermite subdivision scheme is termed $G1$ convergent if it is convergent as a points refining scheme and is convergent as tangents refining scheme such that the limit tangents are tangent to the limit curve, see e.g.,~\cite{dyn2012geometric,lipovetsky2021subdivision, reif2021clothoid}. Note that the normalized tangents are points on the $S^{n-1}$ sphere, and that the definition of convergence can be generalized to manifold-valued data if instead of piecewise linear interpolants we use piecewise geodesic interpolants (see Definition~3.5. in~\cite{dyn2017manifold}).

By Theorem 3.6. in~\cite{dyn2017manifold}, sufficient conditions for a subdivision scheme refining points $\mathcal{S}$ to be convergent are displacement$-$safe and a contractivity factor $ \mu \in (0,1)$. The first requires a bound between two control polygons of consecutive refinement levels, and it trivially holds for interpolatory schemes. The second one means that 
\[ \Delta(\mathcal{S}(P^k)) \le \mu \Delta(P^k), \quad \text{ where } \quad \Delta(P) = \sup_j d(p_{j}^{k}, p_{j+1}^{k})  . \]

Back to the Hermite setting, with the initial data $ P^{0} =\left(\begin{pmatrix}
p_{j}^0 \\ 
v_{j}^0 \\
\end{pmatrix}\right)_{j\mathbb{\in Z}}$ and with $P^{k}=\left(\begin{pmatrix}
p_{j}^{k} \\ 
v_{j}^{k} \\ 
\end{pmatrix}\right)_{j\mathbb{\in Z}}$ the refined Hermite data at the $k$-th refinement level, and its associated
\begin{equation}
\sigma^{(k)} = \sup_{j}\sigma\left(\begin{pmatrix}
p_{j}^{k} \\ 
v_{j}^{k} \\ 
\end{pmatrix},\begin{pmatrix}
p_{j + 1}^{k} \\ 
v_{j + 1}^{k} \\ 
\end{pmatrix}\right) .
\end{equation}
In addition, we define the piecewise geodesic interpolant of the tangent vectors at the $k$-th refinement
level $\left\{\left(v_{j}^{k}\right)\right\}_{j\mathbb{\in Z}}$ as
\begin{equation}
PG_{k}\left(t\right) = M_{t2^{k} - j}\left(v_{j}^{k},v_{j + 1}^{k}\right), \quad   t\in\left[2^{ - k}j,2^{ - k}\left(j + 1\right)\right) , 
\end{equation}
where $ M_{\omega }$ is the sphere geodesic average, namely $ M_{\omega }(u,v) $ is the unique point on the great circle connecting two non-antipodal points $u$ and $v$, that divides the geodesic distance between $u$ and $v$ in a ratio of $1-\omega$ when measuring from $u$ and $\omega$ when measuring from $v$. Then, we conclude:
\begin{thm}
The IHB-scheme is $ G^{1}$ convergent whenever the initial data satisfies $ \sigma^{\left(0\right)}\leq\frac{3\pi }{4}$.
\end{thm}
\begin{proof}
Let 
$ P^{0}=\left(\begin{pmatrix}
p_{j}^0 \\ 
v_{j}^0 \\ 
\end{pmatrix}\right)_{j\mathbb{\in Z}}$
be the initial Hermite type data with $ \sigma^{\left(0\right)}\leq\frac{3\pi }{4}$. We begin with the convergence of the ntangents.  To this end, we prove that $\left\{ PG_{k}\left(t\right)\right\}_{k\mathbb{\in N}}$ is a Cauchy series for all $ t$. The proof follows closely the proof of Theorem~3.6. in~\cite{dyn2017manifold}. 

Recall $\theta$ of~\eqref{theta definition} and $\theta_{0}$ and $\theta_{1} $ as defined in~\eqref{theta j definition}. Then, we observe that for any $\begin{pmatrix}
p_{0} \\ 
v_{0} \\ 
\end{pmatrix},\begin{pmatrix}
p_{1} \\ 
v_{1} \\ 
\end{pmatrix}\in\mathbb{R}^{n}\times S^{n - 1}$,
\begin{align*}
\theta = g\left(v_{0},v_{1}\right) &\leq \theta_{0} + \theta_{1} = \sqrt{\theta_{0}^{2} + 2\theta_{0}\theta_{1} + \theta_{1}^{2}} \\ 
& \leq\sqrt{2\left(\theta_{0}^{2} + \theta_{1}^{2}\right)} =\sqrt{2}\sigma\left(\begin{pmatrix}
p_{0} \\ 
v_{0} \\ 
\end{pmatrix},\begin{pmatrix}
p_{1} \\ 
v_{1} \\ 
\end{pmatrix}\right) .
\end{align*}
Let $t\in\left[2^{ - k}j,2^{ - k}\left(j + 1\right)\right)$, $ j\mathbb{\in Z}$. Then, $g\left(PG_{k}\left(t\right),PG_{k + 1}\left(t\right)\right)$ is bounded by

\begin{align*}
 & g\left(PG_{k}\left(t\right),v_{j}^{k}\right) + g\left(v_{j}^{k},PG_{k + 1}\left(t\right)\right) \\
& = g\left(PG_{k}\left(t\right),v_{j}^{k}\right) + g\left(v_{2j}^{k + 1},PG_{k +1}\left(t\right)\right) \\
& \leq g\left(v_{j}^{k},v_{j + 1}^{k}\right) + g\left(v_{2j}^{k + 1},v_{2j + 1}^{k + 1}\right) + g\left(v_{2j+1}^{k + 1},v_{2j + 2}^{k + 1}\right) \\
&\leq\sqrt{2}\cdot \sigma^{\left(k\right)} + 2\sqrt{2}\cdot \sigma^{\left(k + 1\right)}\leq 3\sqrt{2}\cdot \mu^{k}\sigma^{\left(0\right)}
\end{align*}
with $\mu  =\sqrt{0.9}$. The last inequality follows from Lemma~\ref{angles contraction lemma}. 
Thus  $\left\{ PG_{k}\right\}_{k\mathbb{\in N}}$ converges uniformly and the IHB-scheme is convergent as a refining scheme of ntangents.

We now consider the IHB-scheme as a point refining scheme. It is interpolatory and hence  displacement-safe. To derive a contractivity factor
we consider levels $k$ with $k$ large enough, such that   $\sup_{j}\theta_{0}^{\left(j\right)} + \theta_{1}^{\left(j\right)}\leq\frac{2\pi }{3}$, where $\theta_i^{\left(j\right)}=\theta_i\left(\begin{pmatrix}
p_j^{k}\\v_j^{k}\\
\end{pmatrix},\begin{pmatrix}
p_{j+1}^{k}\\v_{j+1}^{k}\\
\end{pmatrix}\right)$ for $i=0,1$. By Lemma~\ref{points distance contraction lemma}, 
\[ d\left(p_{2j}^{k + 1},p_{2j + 1}^{k + 1}\right),d\left(p_{2j + 1}^{k + 1},p_{2j + 2}^{k + 1}\right)\leq\frac{5}{6}d\left(p_{j}^{k},p_{j + 1}^{k}\right) , \]
with $\gamma  =\frac{2\pi }{3}$ in the proof of Lemma~\ref{points distance contraction lemma}. It follows that the IHB-scheme has a contractivity factor of $\frac{5}{6}$  as a point refining scheme, and thus we deduce the convergence of the points.

Finally we observe that since $ \sigma^{\left(k\right)}\xrightarrow{k\longrightarrow \infty }0$ both the angles between $ v_{j}^{k}$ to $ p_{j + 1}^{k} - p_{j}^{k}$ and to $ p_{j}^{k} - p_{j - 1}^{k}$ approaches zero when $ k$ approaches infinity.\ \ This means that the limit tangents are tangent to the limit curve.

We conclude that the IHB-scheme is G1 for data satisfying $ \sigma^{\left(0\right)}\leq\frac{3\pi }{4}$.
\end{proof}

\begin{remark}  \label{remark: different modifications}
\phantom{The following}
\begin{enumerate}[label=(\roman*)]
    \item 
    Different modifications of LR1 can be obtained by using other Hermite averages, for example as done in~\cite{lipovetsky2021subdivision} in the 2D case. 
    \item
    Similar arguments, as in the proofs of this section, lead to an extension of the convergence result in~\cite{lipovetsky2021subdivision} for a wider class of initial data and for any dimension.
\end{enumerate}
\end{remark}

\section{Numerical examples} \label{sec:examples}

This section provides 2D and 3D examples of interpolating or approximating curves based on  geometric Hermite samples. Specifically, we test the performance of the IHB-scheme and the 
HB-LR3 scheme (modification, based on the B\'{e}zier average, of the LR3 algorithm). We then compare these results to approximations by related methods, such as other modifications of LR1, see Remark~\ref{remark: different modifications}.

Note that all the examples of this section are available as Python code for reproducibility and for providing a further point of view, in~\url{https://github.com/HofitVardi/Hermite-Interpolation}.

\subsection{Comparison between different modifications of LR1}

In the first example, we compare the IHB-scheme with the modification of LR1 in~\cite{lipovetsky2021subdivision}; see also Remark~\ref{alternative average} and Remark~\ref{remark: different modifications}. We apply both schemes to data sampled from spirals in $\mathbb{R}^{2}$ and $\mathbb{R}^{3}$. 

Figure~\ref{fig:spiral} shows the limit curves obtained from  samples with increasing density. As demonstrated in the figure, for a high density of samples, the performances of the two schemes are almost identical. However, as the number of samples decreases, the performance of the IHB-scheme is superior. Note that as mentioned in Remark~\ref{alternative average}, in the 2D case and for a sufficiently dense sampling, the two methods coincide. This example indicates similar behavior in 3D.

\begin{figure}[hbt!]
    \centering
    \includegraphics[width=0.26\textwidth]{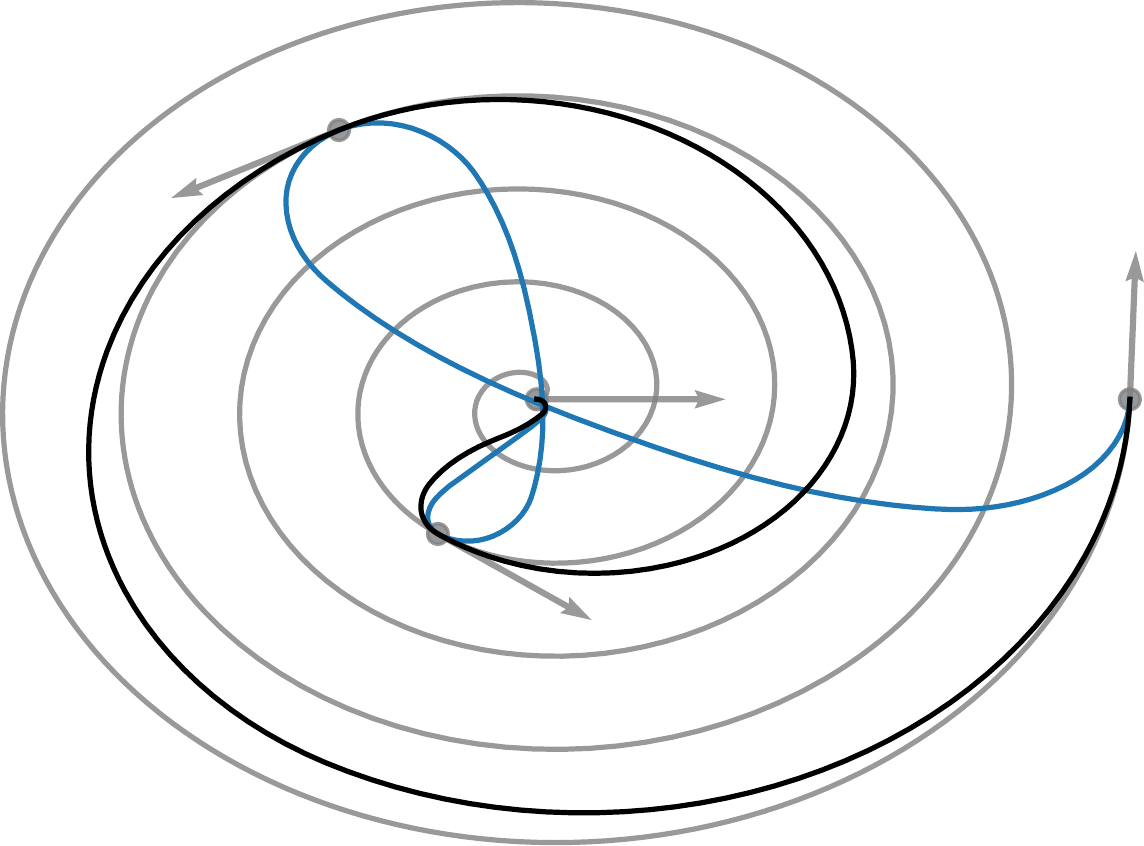}
    \quad\quad
    \includegraphics[width=0.26\textwidth]{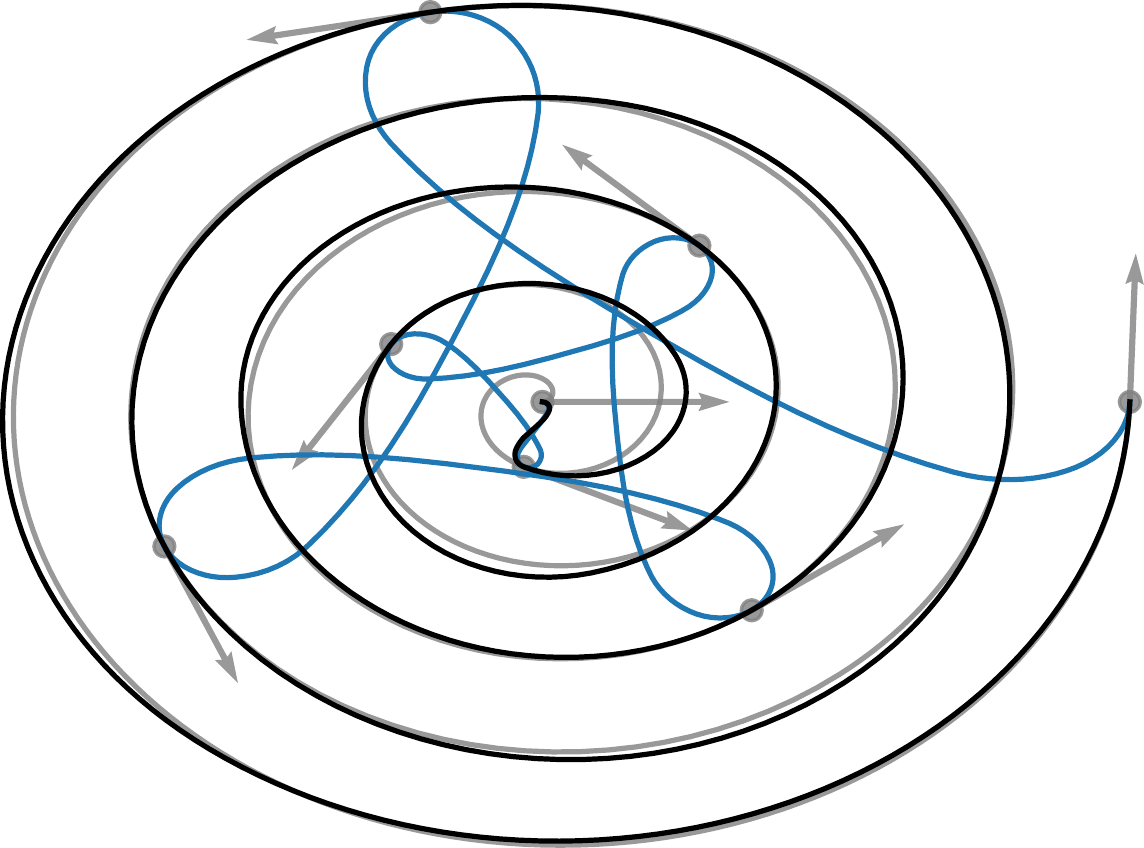}
    \quad\quad
    \includegraphics[width=0.26\textwidth]{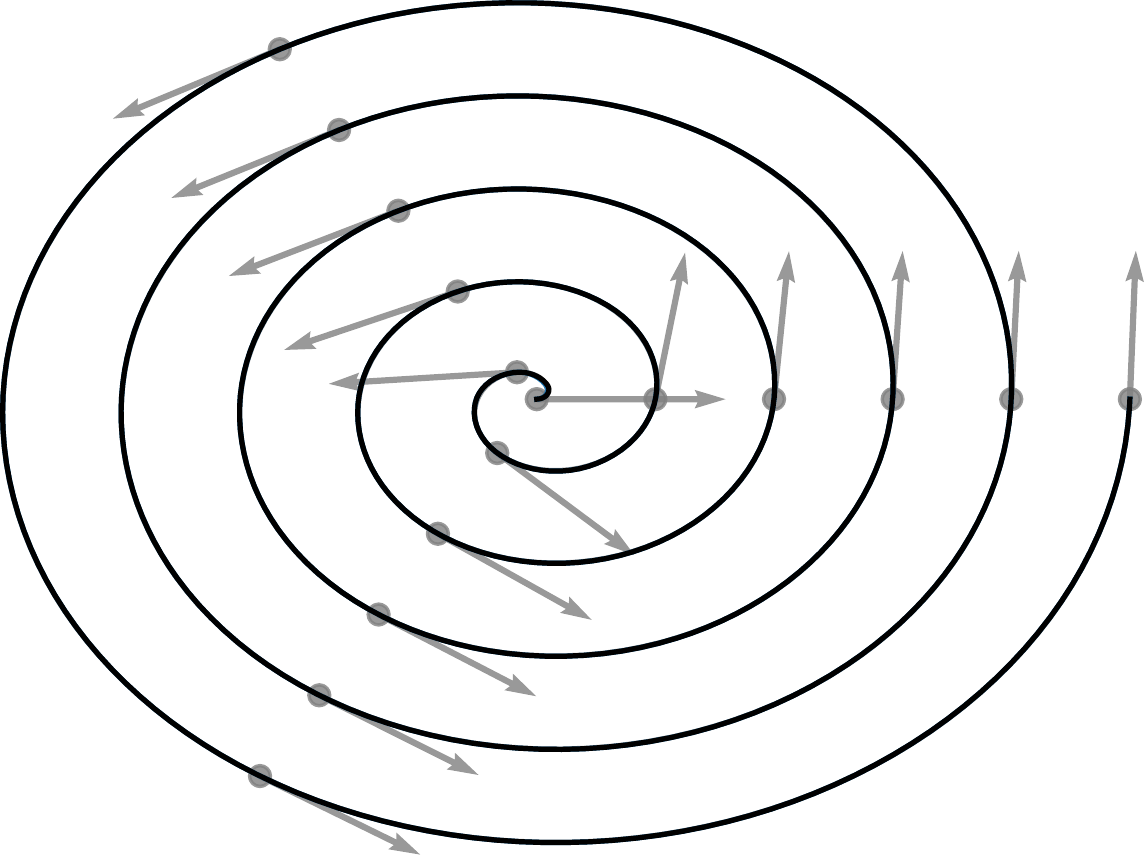}
    \\ \vspace{\baselineskip}
    \includegraphics[width=0.3\textwidth]{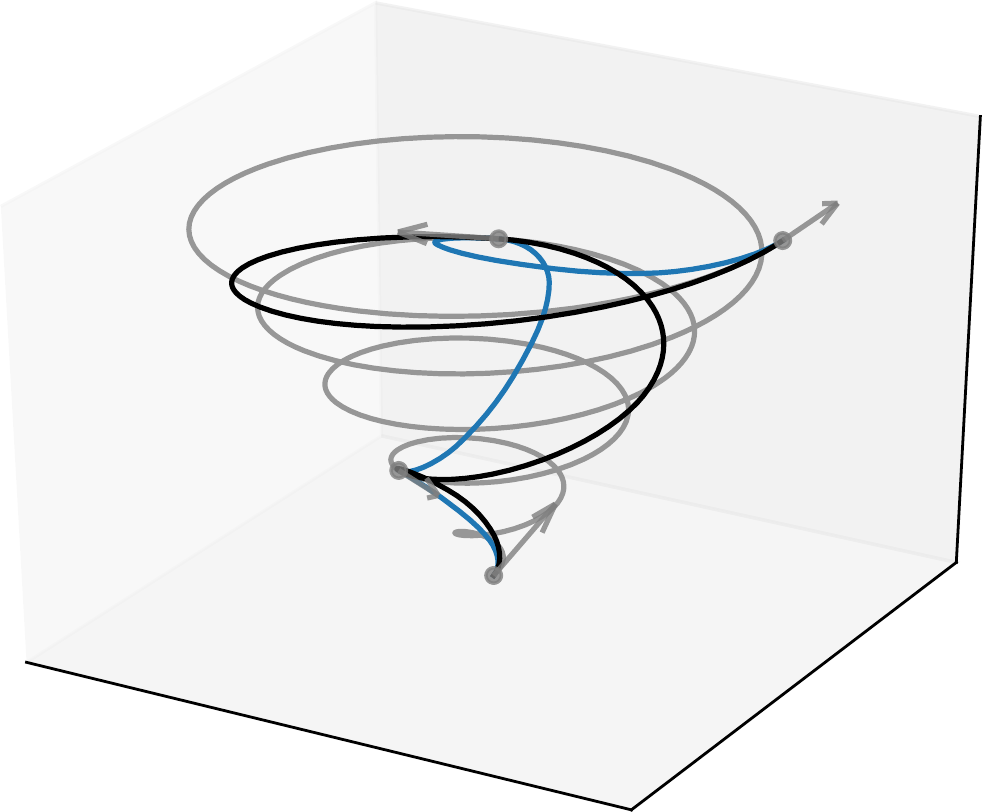}
    \quad
    \includegraphics[width=0.3\textwidth]{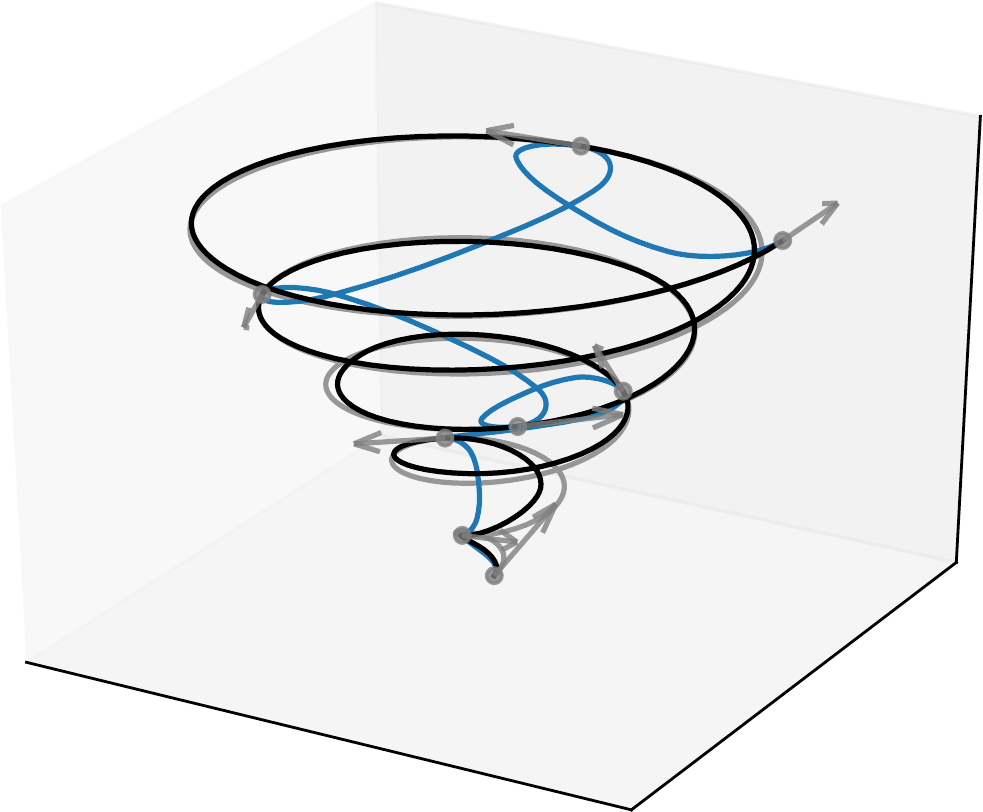}
    \quad
    \includegraphics[width=0.3\textwidth]{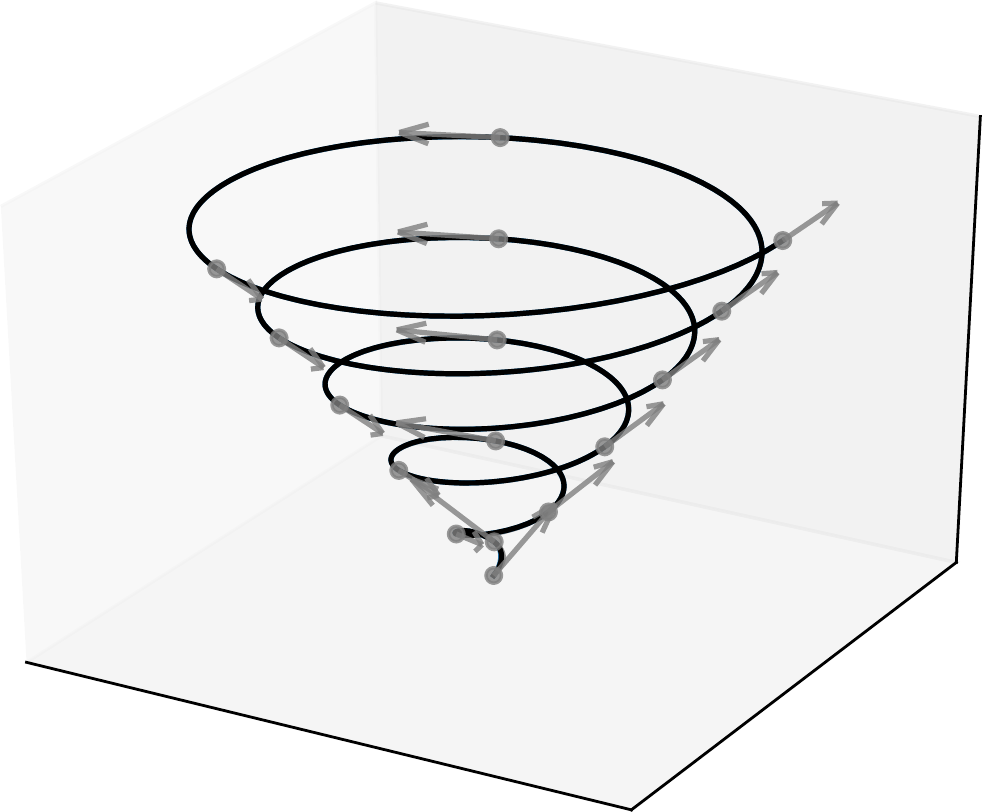}
    \caption{Comparison between two modifications of LR1. Spiral approximation in 2D (upper row) and 3D (lower row) by the IHB-scheme (black), by the modification of LR1 as in~\cite{lipovetsky2021subdivision} (blue), initial data and spirals (gray).}
    \label{fig:spiral}
\end{figure}

\subsection{Geometric versus linear Hermite interpolation}

We interpolate geometric Hermite data, sampling the $2D$ curve $\gamma(t)=(t,\sin{t})$, by the IHB-scheme and by a linear Hermite subdivision scheme. The latter subdivision was introduced in~\cite{merrien1992family}, and we term it Merrien-scheme. The Merrien-scheme is based on cubic Hermite interpolation, and uses point-tangent pairs with tangents that are not normalized as its initial data. Nevertheless, from the reasoning we present next, we compare the above two schemes when applied to geometric Hermite data, with different sampling
rates.

\begin{figure}[!hbp]
    \centering
    \begin{subfigure}[t]{0.33\textwidth}
                 \includegraphics[width=\textwidth]{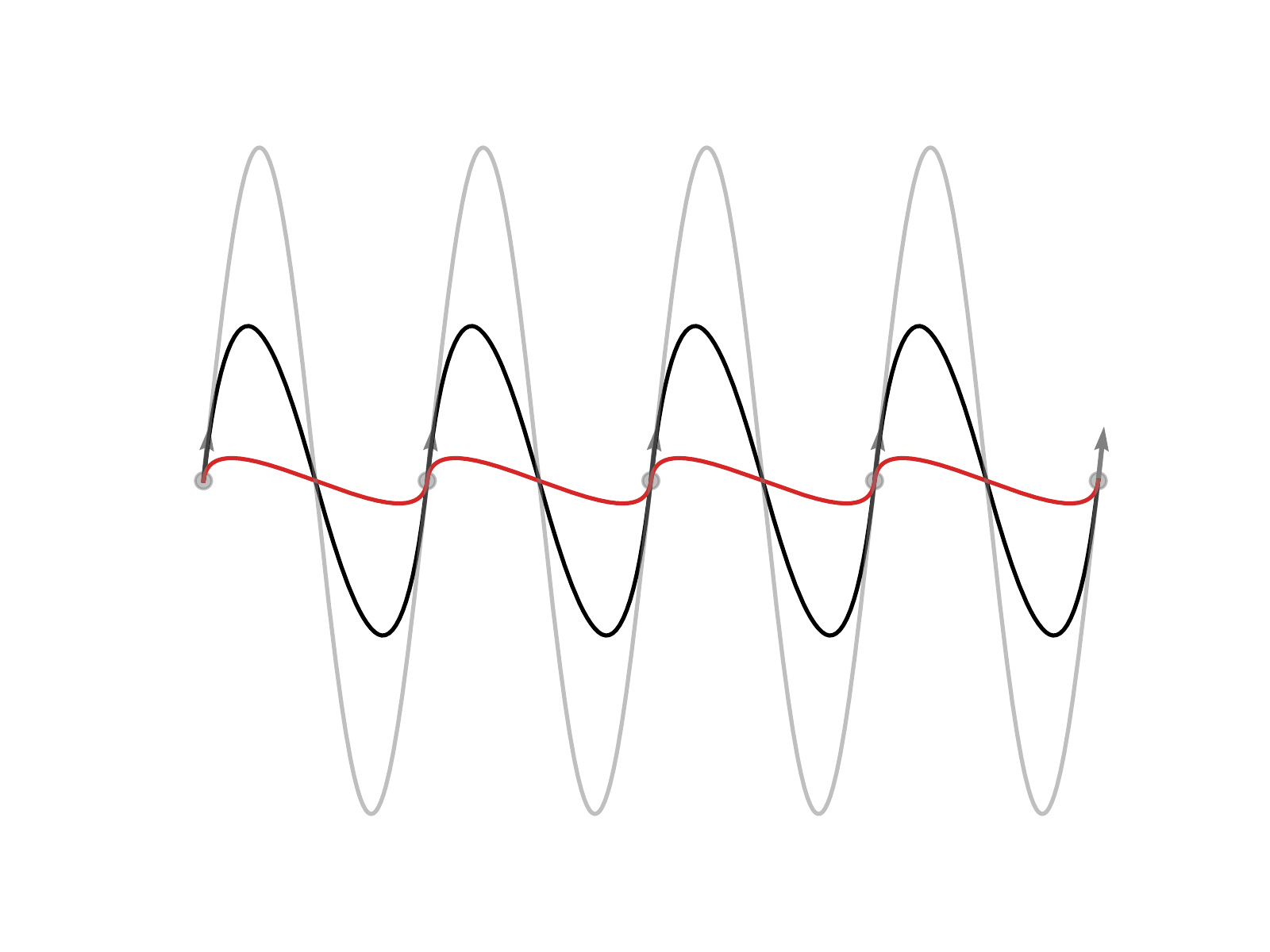}
                 \caption*{$h=2\pi$}
     \end{subfigure} \quad\quad
     \begin{subfigure}[t]{0.33\textwidth}
                 \includegraphics[width=\textwidth]{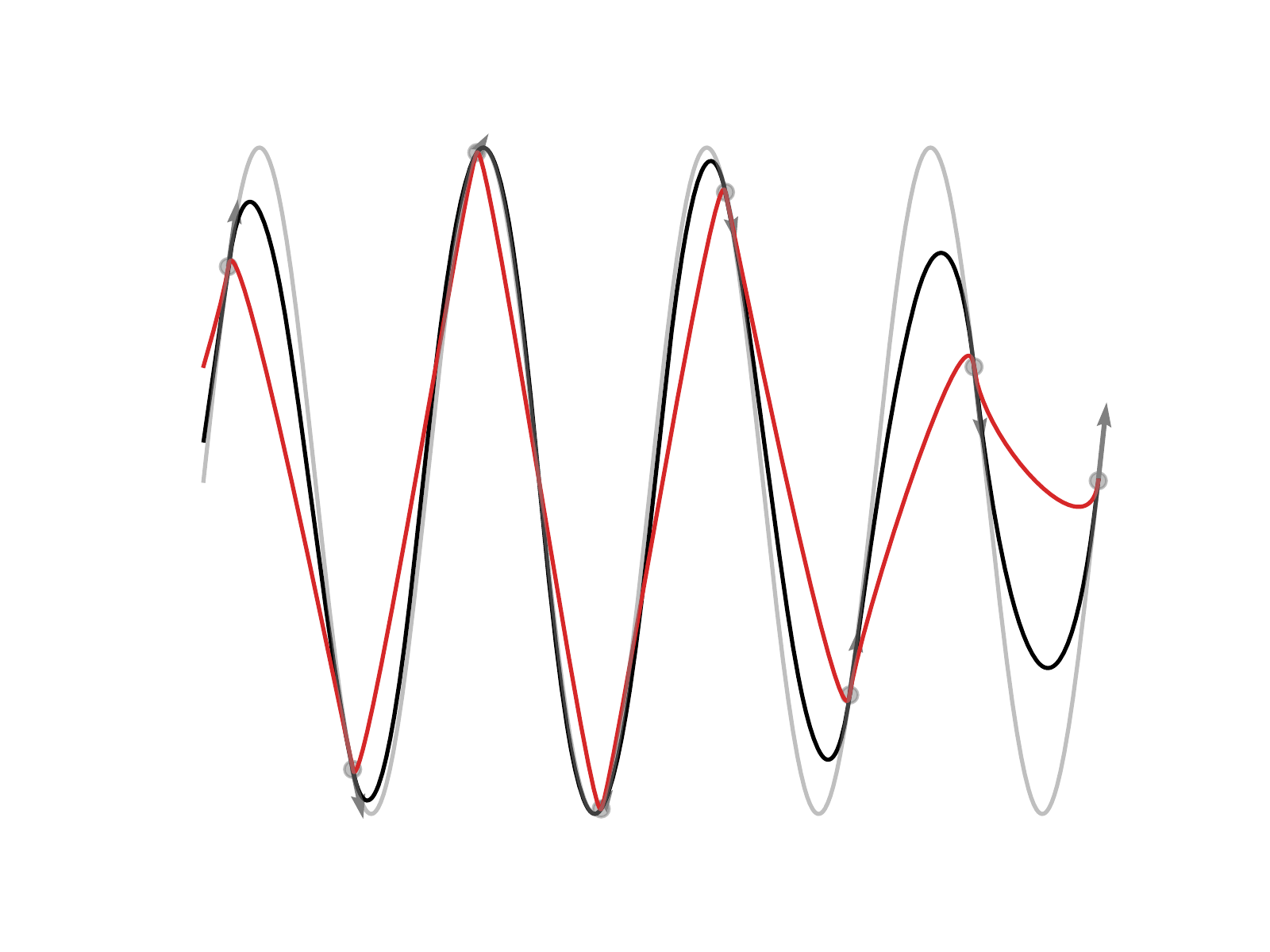}
                 \caption*{$h=\frac{10}{9}\pi$}
    \end{subfigure} \\
    \begin{subfigure}[t]{0.33\textwidth}
                 \includegraphics[width=\textwidth]{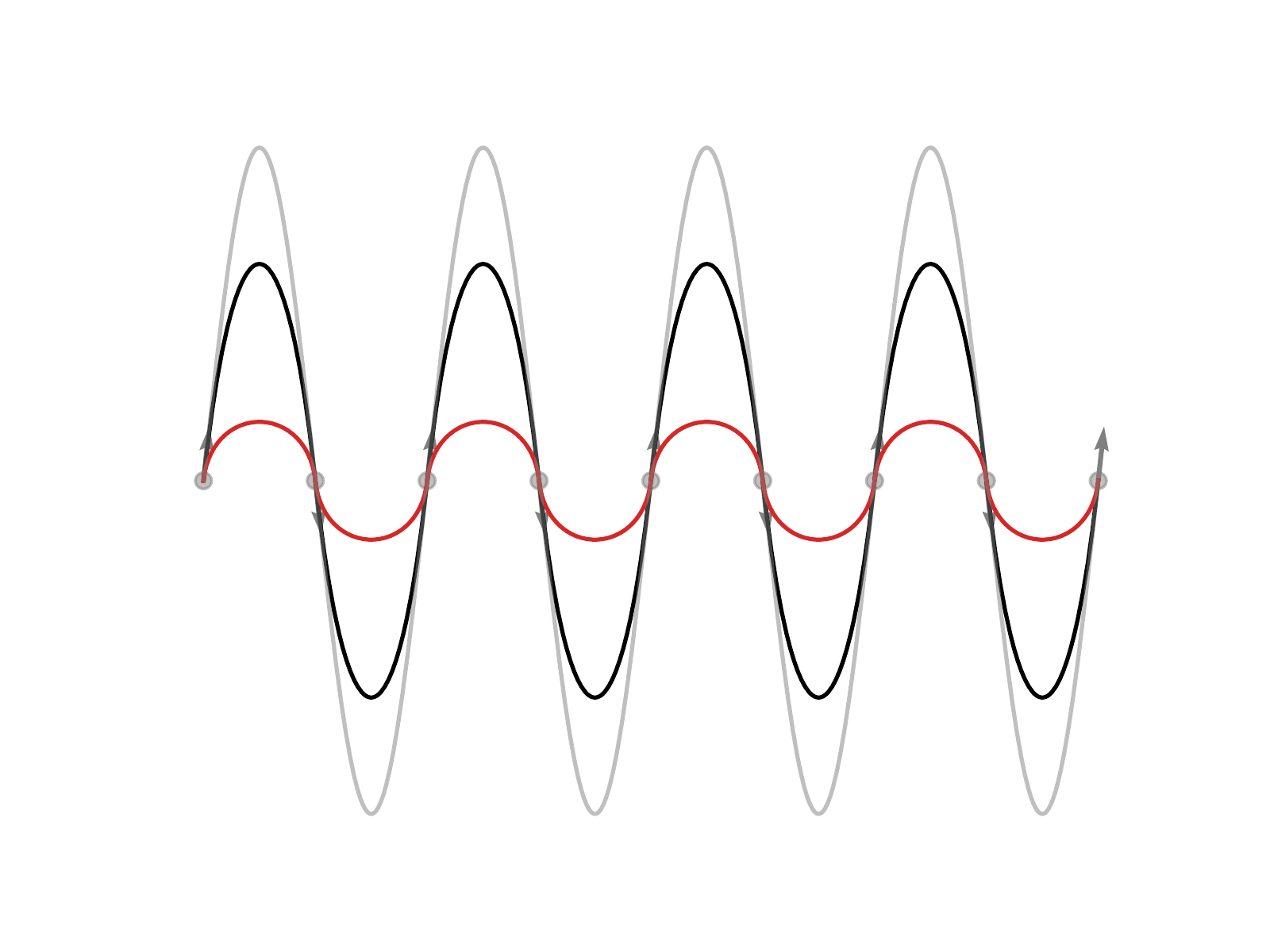}
                 \caption*{$h=\pi$}
    \end{subfigure}  \quad\quad
    \begin{subfigure}[t]{0.33\textwidth}
                 \includegraphics[width=\textwidth]{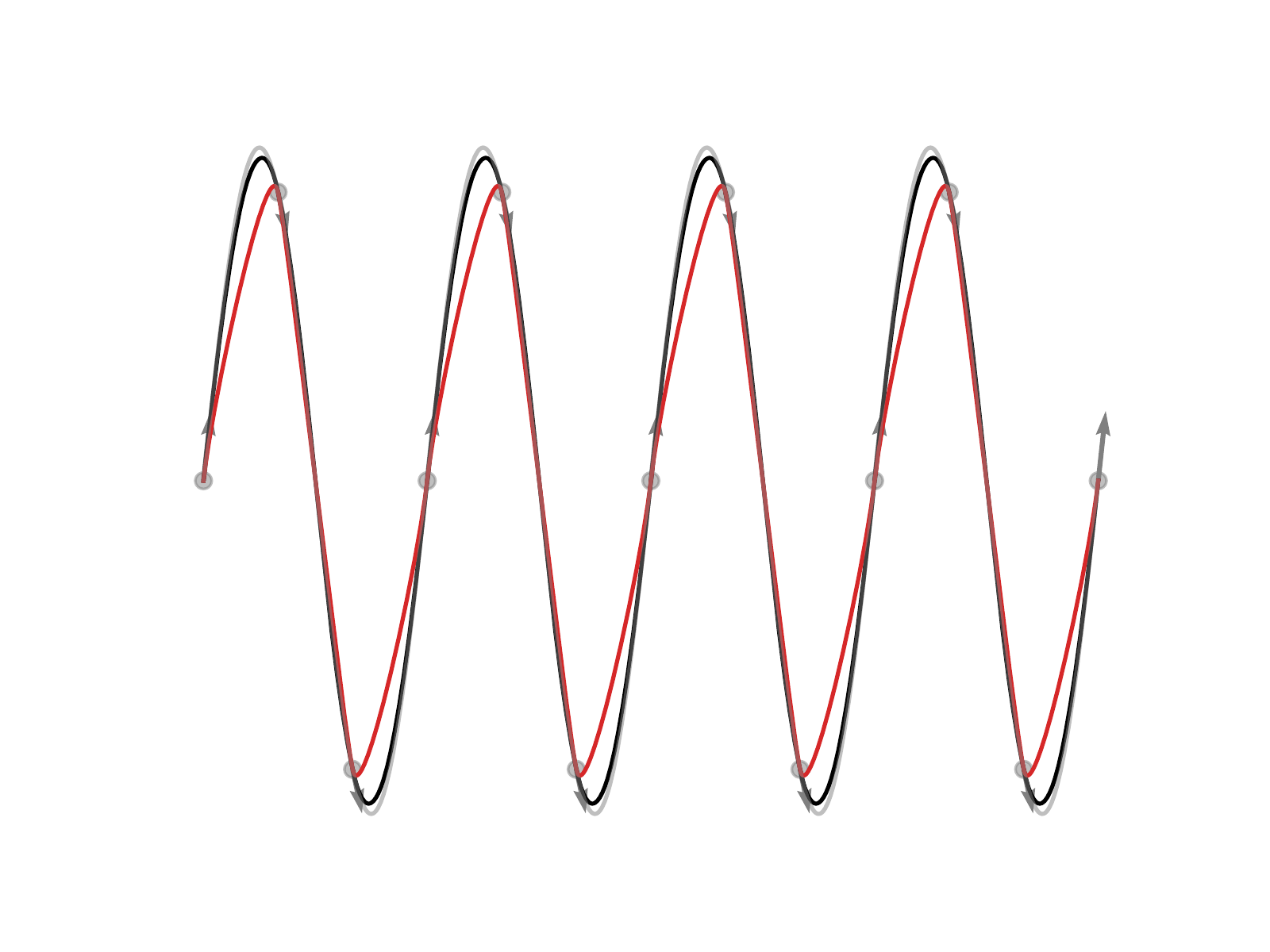}
                 \caption*{$h=\frac{2}{3}\pi$}
    \end{subfigure} \\
    \begin{subfigure}[t]{0.65\textwidth}
                 \includegraphics[width=\textwidth]{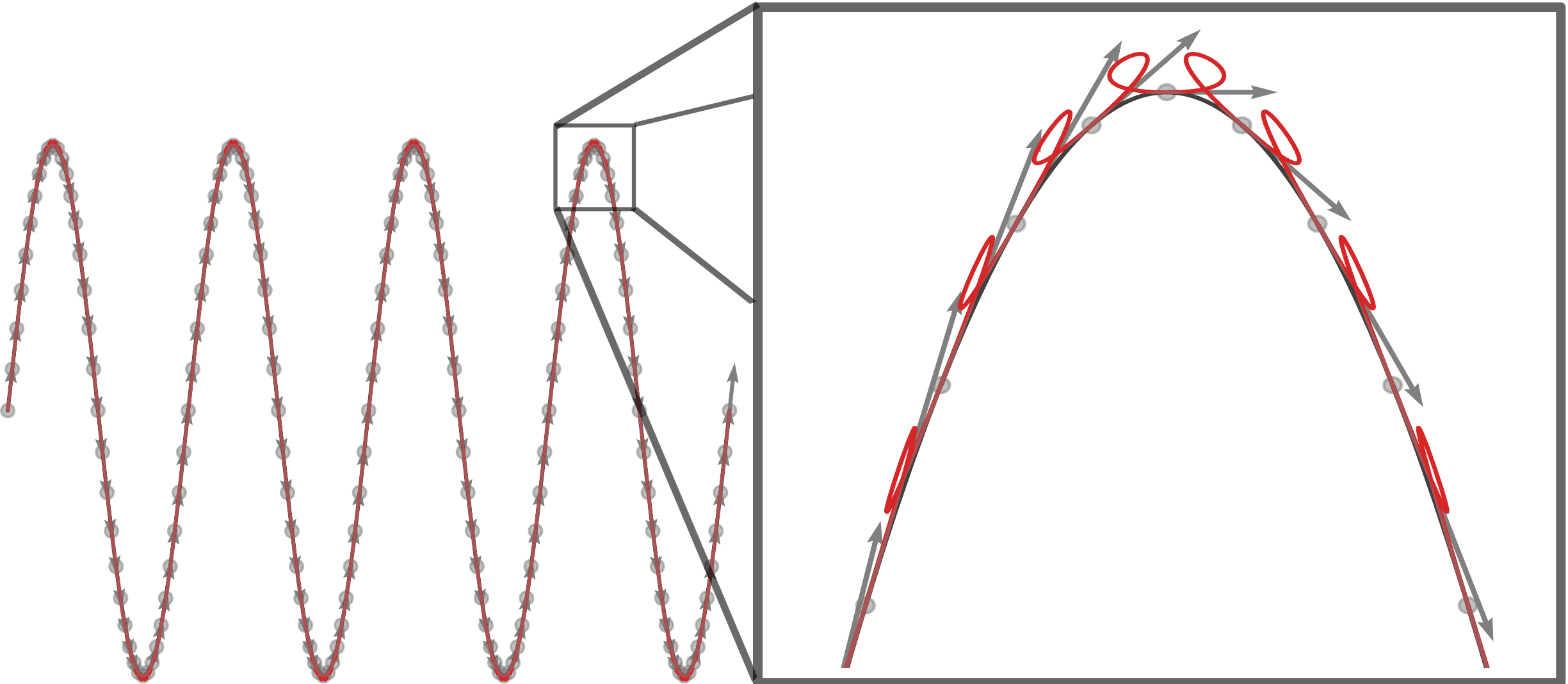}
                 \caption*{$h=0.05\pi$}
    \end{subfigure}
    \caption{Hermite Interpolation for varying sampling densities by the IHB-scheme (black) and by Merrien-scheme (red), both applied to geometric Hermite data. The sampled curve $\gamma(t)=(t,\sin{t})$ and the initial data are gray. Here $h$ is the distance between consecutive sampling points in the parameter domain.}
    \label{fig:Schemes comparison}
\end{figure}

For approximating a function by subdivision schemes, we assume that the points where we sample the function are equidistant. For curves approximation, the samples have to be close to equidistant with respect to the arc-length parameterization, see, e.g.,~\cite{floater2006parameterization}. To imitate such an ideal sampling in the Hermite setting, one must consider the distance between the points and the tangents' magnitude. Therefore, to show the effect of sampling, we fix the tangents to be normalized and apply different sampling rates in the current comparison. 

Figure~\ref{fig:Schemes comparison} demonstrates the way linear Merrien-scheme and the IHB-scheme address sparse sampling (top and middle left sub-figures) and dense one (bottom sub-figure). We observe that the sampling rate significantly affects the approximation quality of the linear scheme. In particular, sparse samples yield a ``flattened'' curve, and dense samples generate unnecessary loops. On the other hand, the B\'{e}zier average, being adaptive to the distance between the two points, provides a remedy.

\subsection{The contribution of the tangent directions to the approximation}

In the following example, we test the contribution of the tangent directions to the quality of the approximation. This contribution is clear when sampling a periodic curve in parameter distances which are equal to its period. In this case, the only reasonable approximation in the absence of tangent directions is a straight line. On the other hand, our scheme when applied to Hermite data sampled from the curve $\gamma(t)=(t,\sin{t})$, in parameter distances which are equal to its period, preserves the oscillations of $\gamma$ (see the first sub-figure of Figure~\ref{fig:naive vectors}). Note that this cannot be the case when the sampled points are extremum points. 

\begin{figure}[!htp]
    \centering
    \begin{subfigure}[t]{0.35\textwidth}
        \includegraphics[width=\textwidth]{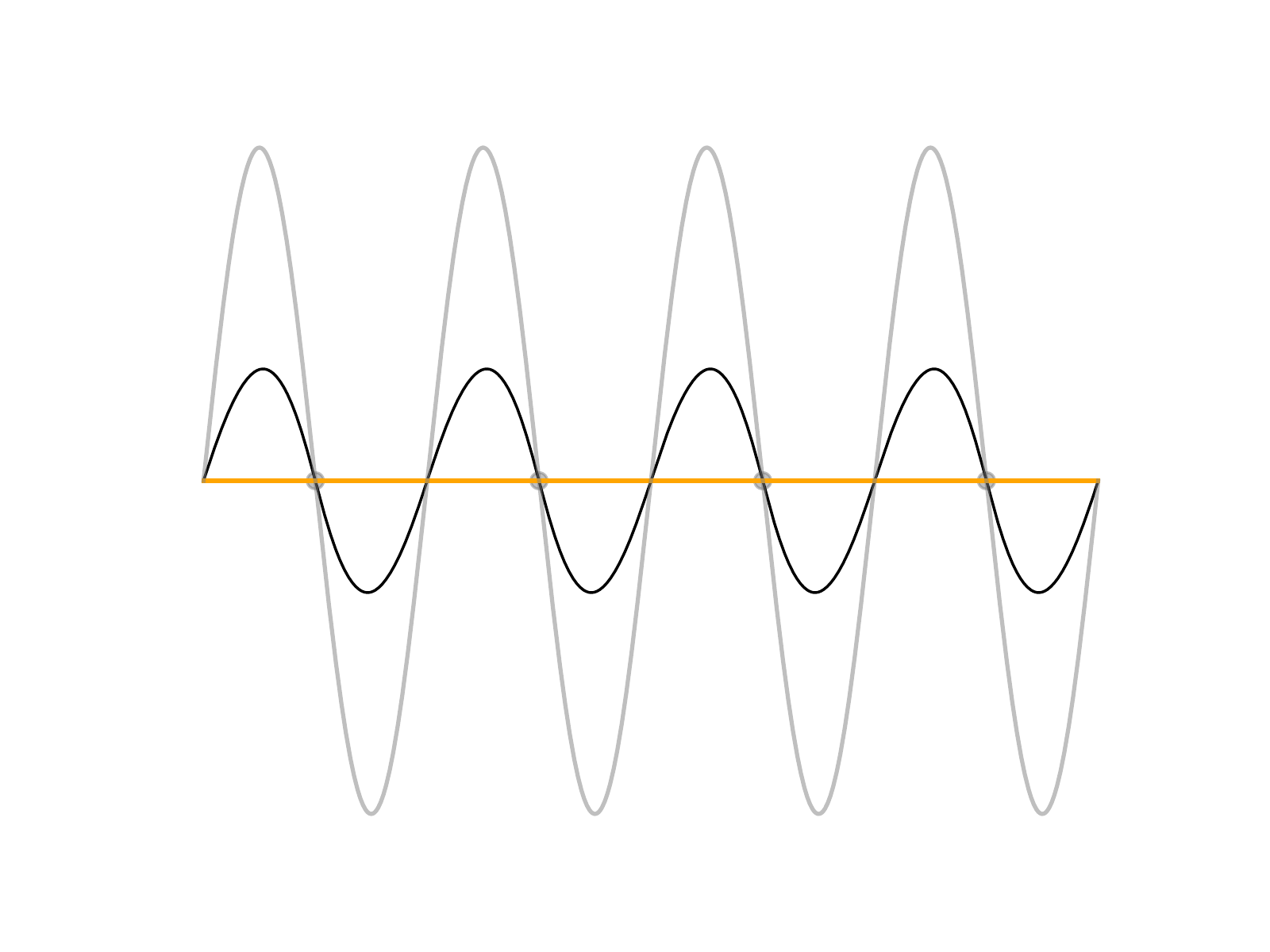}
        \caption*{$h=2\pi$}    
    \end{subfigure}\quad\quad
    \begin{subfigure}[t]{0.35\textwidth}
        \includegraphics[width=\textwidth]{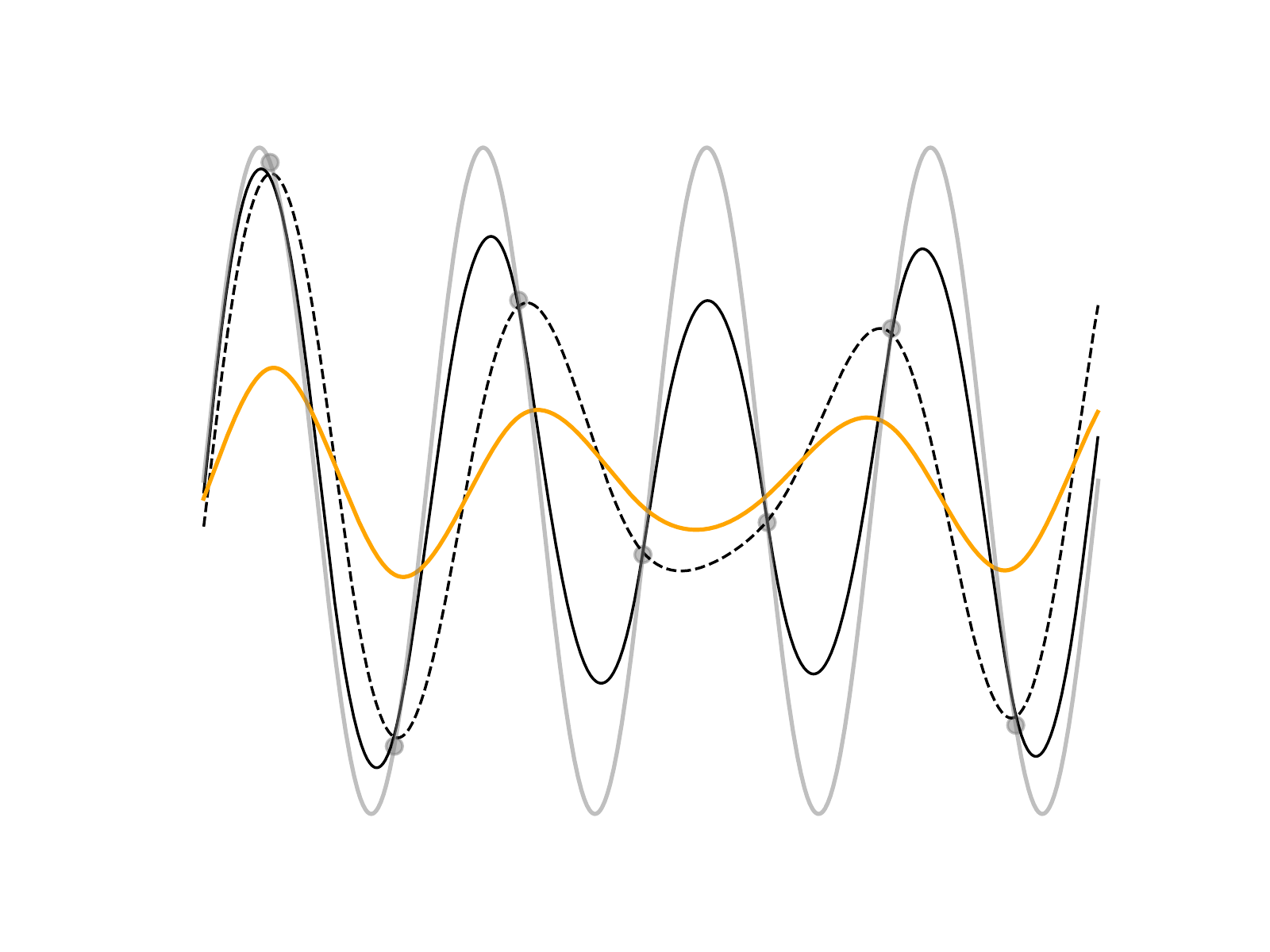}
        \caption*{$h=\frac{10}{9}\pi$}    
    \end{subfigure}\\
    \begin{subfigure}[t]{0.35\textwidth}
        \includegraphics[width=\textwidth]{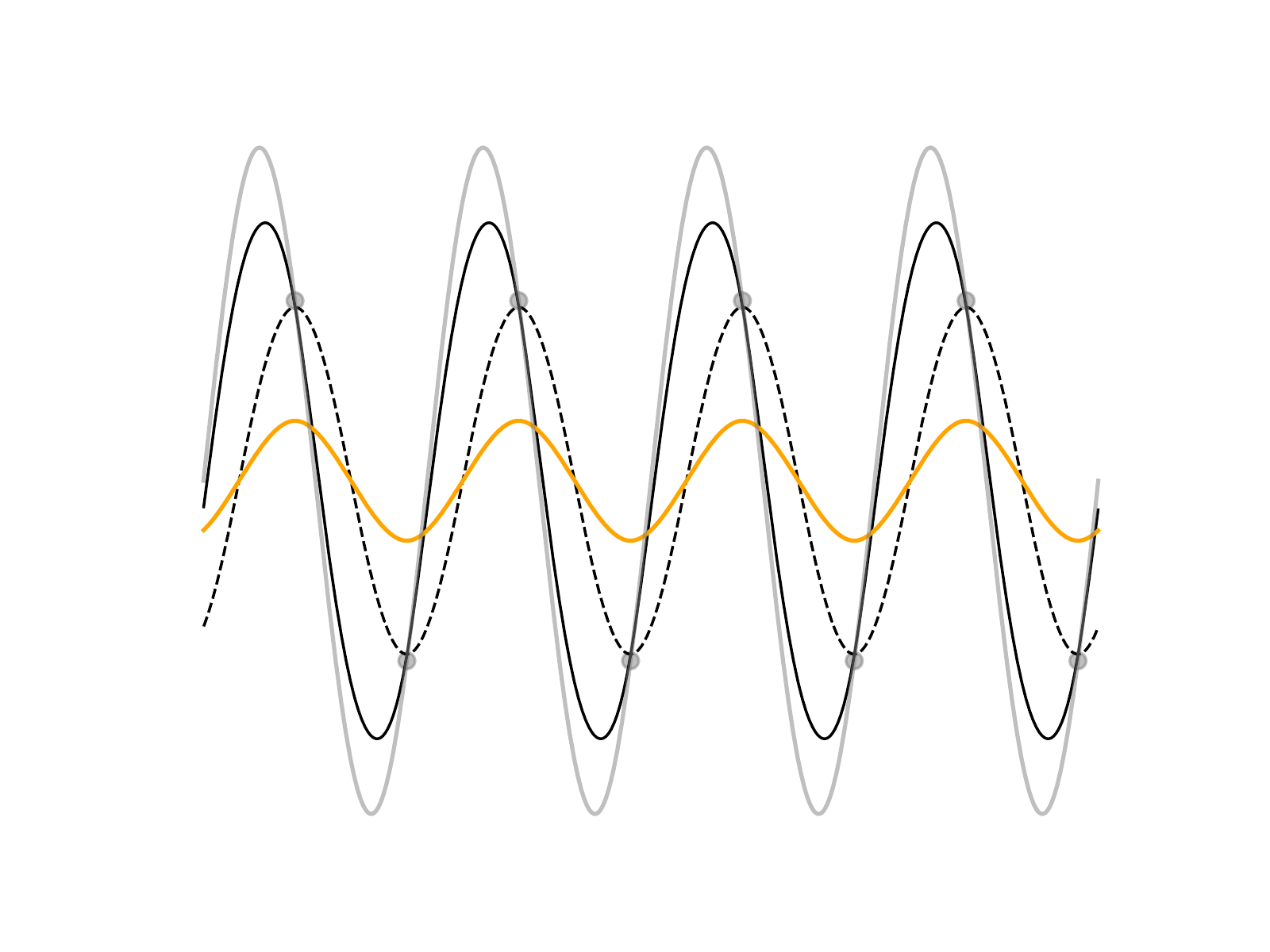}
        \caption*{$h=\pi$}    
    \end{subfigure}\quad\quad
    \begin{subfigure}[t]{0.35\textwidth}
        \includegraphics[width=\textwidth]{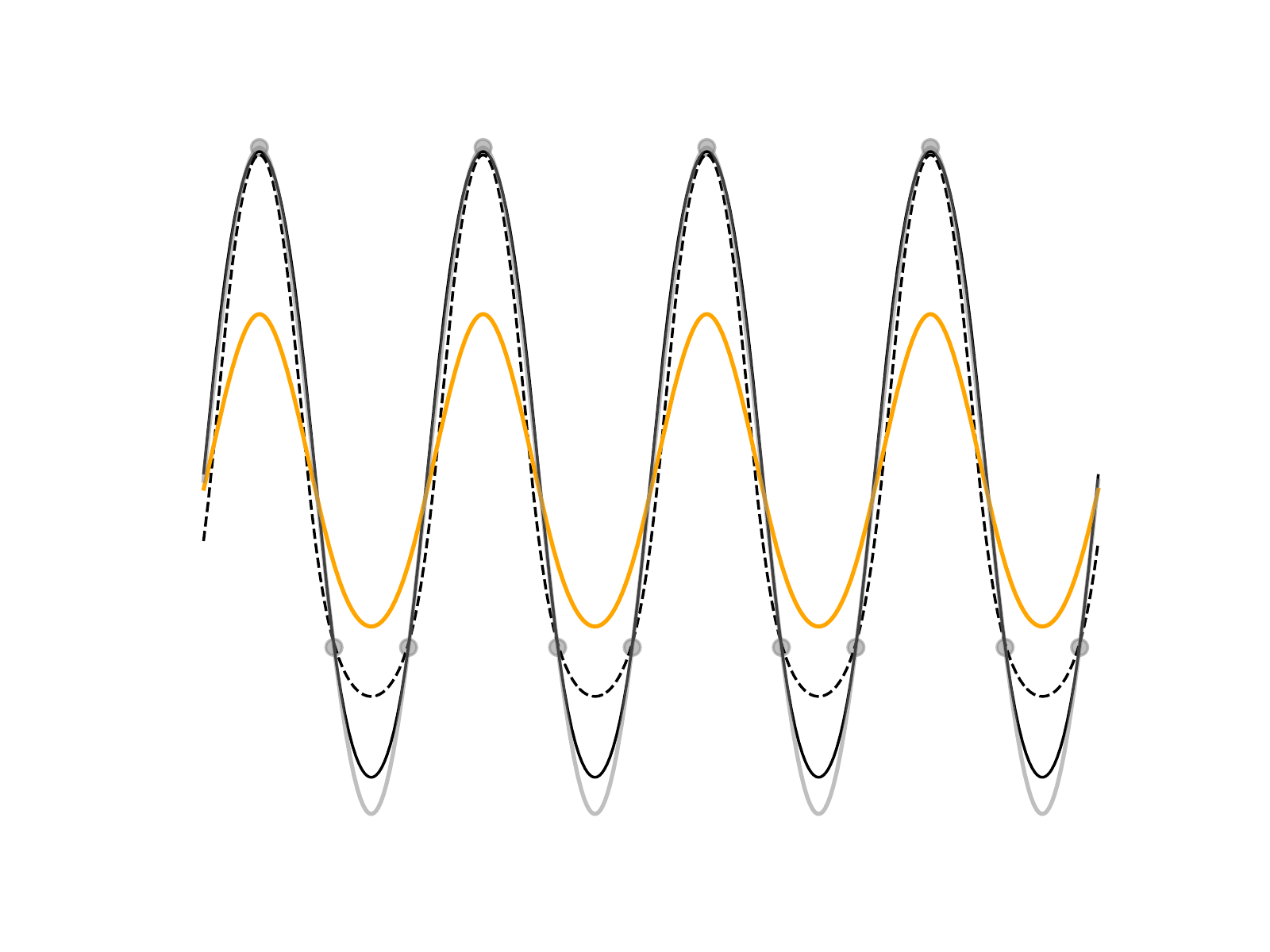}
        \caption*{$h=\frac{2}{3}\pi$}    
    \end{subfigure}\\
    \caption{Non-interpolatory approximation by LR3 applied to initial point values only (orange), the HB-LR3 applied to the same point values with estimated  tangent directions (dashed line) and the HB-LR3 applied to the corresponding Hermite data (solid black). The sampled curve $\gamma=(t,\sin{t})$ and the initial points are in gray. $h$ represents the distance between parameter values corresponding to consecutive samples. } 
    \label{fig:naive vectors}
\end{figure}

We demonstrate the contribution of the tangent directions in two ways: How well information about the initial tangent directions affects the quality of the approximation and how the Hermite setting can be effective even in the absence of information on the initial tangent directions. 

We compare the HB-LR3 scheme, as discussed in Section~\ref{subsec: IHB}, with the linear LR3 scheme. We do so by applying the HB-LR3 once to real Hermite data sampled from $\gamma=(t,\sin t)$, and once to Hermite data, obtained from data consisting of points only, by an algorithm which estimate a tangent direction at a point using the point and its immediate neighboring points; see, e.g.~\cite{yang2006normal} and Remark~\ref{rem:guesing_tangents} below. 

The results of our experiments are presented in Figure~\ref{fig:naive vectors} and indicate the clear superiority of the approximation based on the right tangent directions. Moreover, in the absence of such information, the HB-LR3 scheme still produces a better approximation than the linear LR3 scheme, which in some cases fails to describe the curve under the limited available data. Note that this comparison follows a similar experiment done in \cite{lipovetsky2016weighted} and also indicates the limitation in estimating the tangent directions, see, e.g., the bottom left subfigure. 

\begin{remark} \label{rem:guesing_tangents}
The estimation of tangent directions from the initial points is explained below.

Given points $\{p_{j}\}_{j\mathbb{\in Z}}\subseteq\mathbb{R}^ n$, the estimated tangent direction $\tilde{v}_j$ at the point $p_j$ is the weighted geodesic average of the two normalized vectors $\frac{p_{j} - p_{j-1}}{\left\Vert p_{j} - p_{j-1}\right\Vert }$ and $\frac{p_{j+1} - p_{j}}{\left\Vert p_{j+1} - p_{j}\right\Vert}$ with weight $\omega=\frac{\left\Vert p_{j} - p_{j-1}\right\Vert}{\left\Vert p_{j} - p_{j-1}\right\Vert+\left\Vert p_{j+1} - p_{j}\right\Vert}$. 

The computation of the tangent direction follows the computation of a ``naive normal" in \cite{lipovetsky2016weighted}. 
\end{remark}

\subsection{Approximation order} \label{subsec:app_order}

The last two examples emphasize visually the superior quality of the outputs of the HB-LRm schemes with $m=1$ and $m=3$ as approximants. Namely, the distance between the sampled curve and the curves generated by the aboive two HB-LRm schemes is smaller than the distance obtained by the other methods. In this section, we further examine the approximation quality via the concept of approximation order: how small is the distance between the approximation and the sampled curve as a function of the distance between the equidistant parameter values corresponding to the sampled data, as this distance becomes small.

In the functional setting, classical Hermite interpolation to the data
$(x_i,f^{(j)}(x_i)),\ i,j=0,1$ gives fourth-order approximation over $[x_0,x_1]$ if $\frac{d^4f}{dt^4}$ is continuous on $[x_0,x_1]$. Namely, under these conditions the error decays like $C (x_1-x_0)^4$, for some constant $C$, when $x_1 -x_0$ tends to zero. In particular, methods like piecewise cubic interpolation are known to have a fourth-order approximation see, e.g.,~\cite{conti2015ellipse}. Nevertheless, this claim is more limited when it comes to curves as it depends on a particular choice of samples. Specifically, when one samples a curve in the non-Hermite settings, the sampling must be taken close to equidistant with respect to the arc-length parameterization. In the Hermite setting, we must take into account also the tangents to form the analogue ``equidistant sampling'', see~\cite{floater2006parameterization}. 

In our numerrical tests  we observe that the IHB-scheme preserves the classical fourth-order approximation, independently of any particular sampling policy. We relate this observation to the fact that the B\'{e}zier average depends on the distance between the points and the angle between the tangent directions, and believe that due to this property of our average, the IHB-scheme is more robust to sampling. 

Consider a curve $\gamma$ and its approximation $\tilde{\gamma}$. A common method to measure the approximation error is the Hausdorff distance between the two curves,
\begin{equation} \label{eqn:hausdorf_dist}
    D(\gamma,\tilde{\gamma})=\max\left(\sup_{p\in\gamma}\inf_{q\in\tilde{\gamma}} d(p,q),\sup_{q\in\tilde{\gamma}}\inf_{p\in\gamma} d(q,p)\right).
\end{equation}
Given a set of Hermite samples $\left( (p_{j},v_j) \right)_{j\in J}$ of $\gamma$, as in~\eqref{eqn:sampled_hermite_data}, we denote by $h$ the maximal distance between consecutive samples. One may choose the parametric distance between adjacent samples if a parameterization is known. A natural parameterization is the arc length, which also plays a significant role in the context of approximation order, see~\cite{floater2006parameterization}. As we usually do not have access to parameterization, a possible solution is to set $h = \max_{j \in J} d(p_j,p_{j+1})$. With $h$, we denote the approximation based on the samples by $\gamma_h$. For estimating the approximation order, we consider the approximation error as a function of $h$,
\begin{equation*}
    e(h) = D(\gamma,\gamma_h).
\end{equation*}
Our ansatz is that when $h$ is small, the error function behaves as $C h^\beta$, for some constants $C$ and $\beta$.

\begin{figure}[ht]
    \centering
    \includegraphics[width=0.45\textwidth]{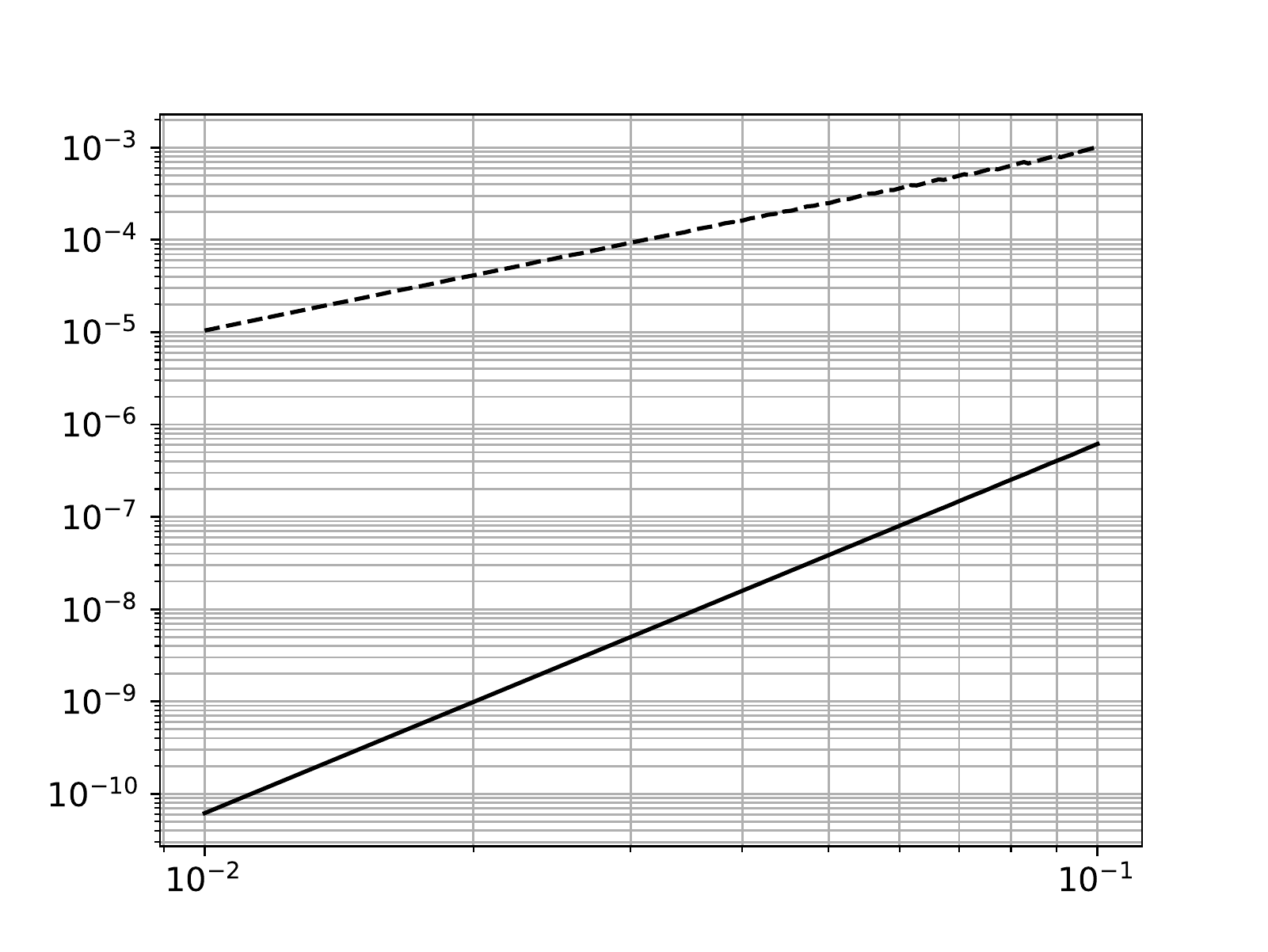}
    \caption{Order of approximation. A log-log plot of the maximal error of approximation to the curve~\eqref{eqn:gamma_example}, as obtained by IHB (solid) and by piecewise linear interpolation (dashed), as a function of the parametric distance between consecutive samples.}
    \label{fig:order}
\end{figure}

For simplicity, we consider in our numerical test a functional curve $\gamma$,
 \begin{equation} \label{eqn:gamma_example}
     \gamma(t) = (t, P(t)) \in \mathbb{R}^2 ,
 \end{equation}
where $P$ is a polynomial. In this case, we sample $\gamma$ equidistantly according to the  parameter $t$, that is $h$ is taken with respect to the first coordinate. Then, we use the maximal pointwise distance at the second coordinate of $\gamma$ as a simpler form for the error.

Next, to estimate numerically $\beta$, we use a log-log plot, presenting $\log(e(h))$ as a function of $\log(h)$ for decreasing values of $h$. In view of our ansatz we expect to obtain a line whose slope is an estimate of $\beta$. Figure~\ref{fig:order} presents the above log-log plot. This figure shows a nearly straight line with a slope equal to $4$. In other words, $e(h) \approx h^4$, that is $\beta=4$. For comparison, we also compute and plot the maximal error of the approximation of $\gamma$ by the linear LR1. The latter yields a line of slope $2$ in accordance with the known rate of approximation of LR1.

\section{Conclusions and future work}

This paper proposes a method for geometric Hermite interpolation by a G1 convergent subdivision scheme. The scheme is a part of a more general approach for addressing the problem of geometric Hermite approximation by modifying the Lane-Riesenfeld family of subdivision schemes~\cite{lane1980theoretical} by replacing the linear binary average in these schemes with an appropriate Hermite average. We give a general definition for such an average and design a specific example, the B\'{e}zier average. We also point to the problem of finding a proper metric over $\mathbb{R}^n\times S^{n-1}$. By having such a metric, we aim to identify the subspace consisting of Hermite samples of curves and adapt the property of metric-boundedness with respect to a given Hermite average. Finally, we mention the particular benefit that might arise once a metric that fulfills the so-called metric-property with respect to a given Hermite average is available. Such a Hermite-metric can significantly facilitate the convergence analysis of the modified LRm schemes.

We show how geometric properties, like circle-preserving of a Hermite average, are inherited by the adapted HB-LRm subdivision schemes, and present numerically the advantage in using a geometric method for approximating curves from their Hermite samples.

While we treat the problem of approximating curves in $\mathbb{R}^n$, the generality of our approach naturally leads to the solution of a more general problem, namely to the approximation of curves in certain non-linear spaces by a similar method. For example, the definition of the B\'{e}zier average has a natural and intrinsic generalization that fits inputs from the tangent bundle of some Riemannian manifolds. The latter is the subject of an ongoing research by the authors.

\bibliography{hermite_bib}

\begin{appendices}

\section{Validation of Lemma~\ref{angles contraction lemma}}
\label{app:sigma proof}

In this section we discuss the claim in Lemma~\ref{angles contraction lemma}. First, by the symmetry of the B\'{e}zier average (see Definition~\ref{Hermite average}), it is sufficient to investigate only $\sigma\begin{pmatrix}\begin{pmatrix}
p_{0} \\ 
v_{0} \\ 
\end{pmatrix},\begin{pmatrix}
p_{\frac{1}{2}} \\ 
v_{\frac{1}{2}} \\ 
\end{pmatrix}
\end{pmatrix}$.
We start by providing an explicit form of $\sigma\begin{pmatrix}\begin{pmatrix}
p_{0} \\ 
v_{0} \\ 
\end{pmatrix},\begin{pmatrix}
p_{\frac{1}{2}} \\ 
v_{\frac{1}{2}} \\ 
\end{pmatrix}
\end{pmatrix}$ in terms of elementary functions of $\theta_0,\theta_1$ and $\theta$. Then, we present numerical indications for the correctness of the claim by evaluating appropriate functions. Finally, we outline the required steps to make the above numerical study into a formal proof.

\subsection{Expressing \texorpdfstring{$\sigma$}{sigma} as a function of \texorpdfstring{$\theta_0,\theta_1$}{theta0, theta1} and \texorpdfstring{$\theta$}{theta}}

Let \begin{equation}
    \label{def theta00 theta01}
    \tilde{\theta }_{0,0} :=\arccos\left(\frac{\left\langle v_{0},p_{\frac{1}{2}} - p_{0}\right\rangle }{\left\Vert p_{\frac{1}{2}} - p_{0}\right\Vert }\right),\quad
\tilde{\theta }_{0,1} :=\arccos\left(\frac{\left\langle v_{\frac{1}{2}},p_{\frac{1}{2}} - p_{0}\right\rangle }{\left\Vert p_{\frac{1}{2}} - p_{0}\right\Vert }\right).
\end{equation} Then,
\begin{equation}
    \sigma\begin{pmatrix}\begin{pmatrix}
p_{0} \\ 
v_{0} \\ 
\end{pmatrix},\begin{pmatrix}
p_{\frac{1}{2}} \\ 
v_{\frac{1}{2}} \\ 
\end{pmatrix}
\end{pmatrix}=\sqrt{\tilde{\theta }_{0,0}^{2} +\tilde{\theta }_{0,1}^{2}}.
\end{equation}

Some algebraic manipulations lead to the following expressions for $\tilde{\theta}_{0,0}$ and $\tilde{\theta}_{0,1}$:
\begin{equation}
	\resizebox{.9\linewidth}{!}{\ensuremath{
\label{theta00}
\tilde{\theta }_{0,0} = \arccos\left(\frac{4\cos^{2}\left(\frac{\theta_{0} + \theta_{1}}{4}\right)\cos\left(\theta_{0}\right) + 1 - \cos\left(\theta\right)}{\sqrt{16\cos^{4}\left(\frac{\theta_{0} + \theta_{1}}{4}\right) + 2 - 2\cos\left(\theta\right) + 8\cos^{2}\left(\frac{\theta_{0} + \theta_{1}}{4}\right)\left(\cos\left(\theta_{0}\right) - \cos\left(\theta_{1}\right)\right)}}\right), }}
\end{equation}
\begin{equation}
	\resizebox{.9\linewidth}{!}{\ensuremath{
\label{theta01}
\tilde{\theta }_{0,1} = \arccos\left(\frac{\frac{\cos^{2}\left(\frac{\theta_{0} + \theta_{1}}{4}\right)\left(12\cos^{2}\left(\frac{\theta_{0} + \theta_{1}}{4}\right) + \cos\left(\theta_{0}\right) - 5\cos\left(\theta_{1}\right)\right)}{\sqrt{18\cos^{4}\left(\frac{\theta_{0} + \theta_{1}}{4}\right) + 1 + \cos\left(\theta\right) - 6\cos^{2}\left(\frac{\theta_{0} + \theta_{1}}{4}\right)\left(\cos\left(\theta_{0}\right) + \cos\left(\theta_{1}\right)\right)}}}{\sqrt{8\cos^{4}\left(\frac{\theta_{0} + \theta_{1}}{4}\right) + 1 - \cos\left(\theta\right) + 4\cos^{2}\left(\frac{\theta_{0} + \theta_{1}}{4}\right)\left(\cos\left(\theta_{0}\right) - \cos\left(\theta_{1}\right)\right)}}\right). }}
\end{equation}
Thus, $\sigma\begin{pmatrix}\begin{pmatrix}
p_{0} \\ 
v_{0} \\ 
\end{pmatrix},\begin{pmatrix}
p_{\frac{1}{2}} \\ 
v_{\frac{1}{2}} \\ 
\end{pmatrix}
\end{pmatrix}$ is a function of $\theta_0,\theta_1$ and $\theta$. Recall that by its definition, see~\eqref{sigma definition}, the quantity $\sigma\begin{pmatrix}\begin{pmatrix}
p_{0} \\ 
v_{0} \\ 
\end{pmatrix},\begin{pmatrix}
p_{1} \\ 
v_{1} \\ 
\end{pmatrix}
\end{pmatrix}$ is also a function of $\theta_0 $ and $\theta_1$.

To prove the inequality of Lemma~\ref{angles contraction lemma} we consider two auxiliary functions, the quotient and the difference:
\begin{align}
     Q\left(\theta_{0} ,\theta_{1},\theta\right) & = \frac{\tilde{\theta }_{0,0}^{2} +\tilde{\theta }_{0,1}^{2}}{\theta_{0}^{2} + \theta_{1}^{2}}
     \label{Q} \\
      D\left(\theta_{0} ,\theta_{1},\theta\right) & = 0.9\left(\theta_0^2+\theta_1^2\right)-\tilde{\theta }_{0,0}^{2}-\tilde{\theta }_{0,1}^{2}
     \label{D}
\end{align}
In particular, we have to show that for all optional choices of $\theta_0,\theta_1,\theta$, we obtain that $Q$ of~\eqref{Q} is not greater than $0.9$ or equivalently that $D$ of~\eqref{D} is non negative.

By the assumption of Lemma~\ref{angles contraction lemma}, $\sqrt{\theta_0^2+\theta_1^2}\leq\frac{3\pi }{4}$. Also, since $\theta_0,\theta_1$ and $\theta$ are the angular distances between three vectors, we have by the triangle inequality that $\lvert\theta_{1} - \theta_{0}\rvert \leq \theta\leq\theta_{0} + \theta_{1}$. Therefore, we define the set of all relevant values of $\theta_0,\theta_1$ and $\theta$ to be 
\begin{equation}     \label{Omega}
    \Omega  :=\{ (\theta_{0} ,\theta_{1},\theta)\in [0,\pi]^3 \colon \sqrt{\theta_{0}^{2} + \theta_{1}^{2}}\leq\frac{3\pi }{4}, \, \lvert\theta_{1} - \theta_{0}\rvert \leq \theta\leq\theta_{0} + \theta_{1}
\}
\end{equation}
As a first illustration, we present a 4D visualization of $Q$ and $D$ over $\Omega$ in Figure~\ref{fig:4d}. The values of the functions appear in color and indicate that indeed $Q$ is less than $0.9$ ans $D$ is non-negative when considering the domain $\Omega$.
\begin{figure}
    \centering
    \begin{subfigure}[t]{0.45\textwidth}
         \includegraphics[width=\textwidth]{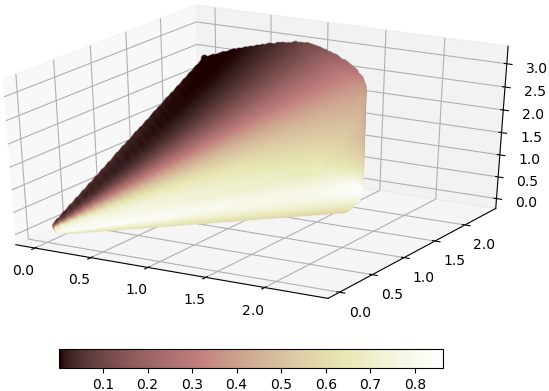}
         \caption*{The quotient $Q$ over $\Omega$}
     \end{subfigure} \quad
     \begin{subfigure}[t]{0.45\textwidth}
         \includegraphics[width=\textwidth]{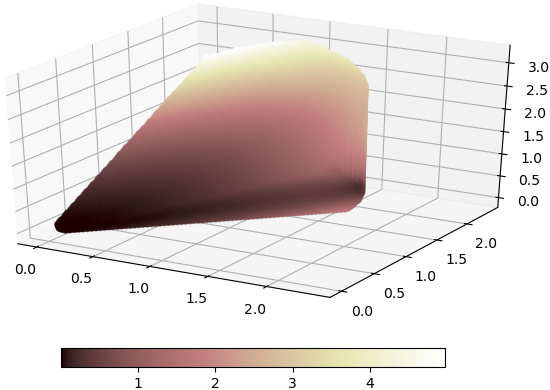}
         \caption*{The difference $D$ over $\Omega$}
     \end{subfigure}
    \caption{A 4D visualization of the functions $Q$ of~\eqref{Q} and $D$ of~\eqref{D} over $\Omega \subset \mathbb{R}^3$ of~\eqref{Omega}. The functions values are given in color, see the horizontal colorbars, and show that the samples satisfy the desired bounds, namely that $Q\le 0.9$ and $D\ge 0$.}
    \label{fig:4d}
\end{figure}

\subsection{Ensuring non-negativity of the difference function}

Here, we provide a more comprehensive numerical method for ensuring the non-negativity of $D(\theta_0,\theta_1,\theta)$ over $\Omega$. The basic idea is to compute $D$ over a grid of points covering $\Omega$, which is dense enough to guarantee non-negativity at all the
points of $\Omega$, if $D$ is non-negative at the points of the grid. 

In more detail, since the function $D(\theta_0,\theta_1,\theta)$ vanishes at the origin,
we first determine a neighborhood of the origin where the minimum of $D$ is zero.Then we sample the function $D$ on the rest of $\Omega$ using steps with a length based on a bound of the analytic form of the gradient, derived from the explicit form of $D$. Taking into account also the finite precision of the machine we use, such a procedure proves that $D$ is positive over
$\Omega$.

We recall two basic results of classical analysis. The first result is concerned with possible locations of points where a minimum value of a function is achieved.
\begin{result} \label{2}
The minimum of $f$ in $\Omega$ is obtained on the boundaries of $\Omega$ ($\partial\Omega$), or in the set of critical points of $f$ in the interior of $\Omega$, i.e., $\{ x\in\Omega\setminus\partial\Omega : \nabla f(x)=0 \}$.
\end{result}
The second result implies how to sample $\Omega$ outside the neighbourhood of the origin.
\begin{result}
    \label{1}
Assume a gradient bound $\| {\nabla f} \|_\infty \leq M$ over $\Omega$ and a local function bound $\min_B(f)=m$, $B \subseteq \Omega$. Then, if $d(x,B )\leq d^\ast$ for any $x\in \Omega$, we have 
\[ f(x)\geq m-M\times d^\ast , \quad x \in B.\] 
In particular, if $d^\ast=\frac{m-\varepsilon}{M}$ then $f(x)\geq \varepsilon$ for any $x\in B$.
\end{result}

The strategy of our proof consists of the following three main stages:
\begin{enumerate}
	\item Split $\Omega$ into two domains, $\Omega = \Omega_1\bigcup\Omega_2$. In particular, find  a radius $r>0$ of an open ball centerd in the origin $B_r(0)$ such that in
	 $\Omega_1=B_r(0)\bigcap\Omega$ the minimum of $D$ is at the origin. \label{r}
	\item Find M such that $\norm{\nabla D}_\infty\leq M$ in $\Omega$ (enough in $\Omega_2$).
	\label{M}
	\item Construct a grid $G$ over $\Omega_2 = \Omega \setminus \Omega_1$ such that 
	\[  d(x,G) \leq \frac{f(g^{\ast}) - \varepsilon}{M}, \quad g^\ast = \arg\min_{g \in G} d(x,g), \quad x \in \Omega_2.  \] 
	Here $\varepsilon$ is a bound on the roundoff error, related to the machine double precision ($2^{-52}\approx2.2\times 10^{-16}$).
	\label{G}
\end{enumerate}
The last step above guarantees, in view of Result~\ref{1}, the non-negativity of $D$ over $\Omega_2$. As to $\Omega_1$, we show that the gradient is non vanishing and deduce by Result~\ref{2} that the non-negativity of $D$ depends on the values of $D$ on the boundaries of $\Omega_1$, which take the form
	\begin{equation*}
		\partial\Omega_1 = \left(\left(P_1\bigcup P_2\bigcup P_3\right)\bigcap \Omega_1\right)\bigcup \left(C\bigcap\Omega\right).
	\end{equation*}
Here $P_i, i=1,2,3$ are three  planes in $\mathbb{R}^3$ and $C=\partial B_r(0)$. The next lemma presents values of $r$ and $M$ allowing  for the application of our strategy of proof.
 
\begin{lemma}  \label{lemma:radius_and_bound}
\phantom{The following}
\begin{enumerate}[label=(\roman*)]
    \item 
    If $x\in B_{0.1}(0)$ and $x\neq 0$  then $\nabla D(x)\neq 0$.
    \item
     $\norm{\nabla D}_\infty\leq 6.3\times 10^4$.
\end{enumerate}
\end{lemma}
The proof of Lemma~\ref{lemma:radius_and_bound} is highly technical, and so we omit it. Nevertheless, it is important to note that the values $r=0.1$ and $M= 6.3\times 10^4$ are not optimal. Specifically, numerical tests imply that the actual bound $M$ is much lower, and we conjecture that $M$ can be decreased down to $10$. This issue is fundamental as higher $M$ means a much denser grid $G$ in step~\ref{G} above. 

To illustrate the last point above, we ran a test with $M=10$, which yielded an exhaustive search within less than $0.1$ of a second. In this case, the grid $G$ consisted of $3,747,179$ points scattered over $\Omega_2$, and the exhaustive search ended up with success; that is, non-negativity was validated. To compare, the same test for $M=100$ took $68$ seconds and required a search over $3,469,028,128$ points. The tests were executed on a standard desktop (Windows-based, runs on Intel i7-9700 CPU with 32G of RAM), and the Matlab script is available at \url{https://github.com/nirsharon/Hermite_Interpolation}. We extended the above test scenario up to $M=1000$ on a designated server. Like the preceding ones, this test also yields a success, meaning the non-negativity was retained. However, we have not reached with our tests the value of $M$, as computed analytically in Lemma~\ref{lemma:radius_and_bound}. Nevertheless, the current test reassures with high confidence that the non-negativity, in this case, is indeed obtained.

\section{Proofs of the B\'{e}zier average properties}
\label{proofs}

\begin{proof}[\textbf{Proof of Lemma~\ref{non vanishing vectors}}]
One can verify that 
\begin{equation*}
    b\left(\frac{1}{2}\right)=\frac{1}{2}(p_0+p_1)+\frac{3}{8}\alpha(v_0-v_1)
\quad \text{ and } \quad
    \frac{d}{dt}b\left(\frac{1}{2}\right)=\frac{3}{2}\left(p_1-p_0-\frac{1}{2}\alpha(v_0+v_1)\right).
\end{equation*}
Assume by contradiction that $\frac{d}{dt}b\left(\frac{1}{2}\right)=0$, and observe that in view of~\eqref{eqn:alpha}, we have that $ \left\Vert p_{1} - p_{0} - \frac{1}{2}\alpha\left(v_{0} + v_{1}\right)\right\Vert^{2}$ is equal to
\begin{align*}
    &\left\Vert p_{1} - p_{0}\right\Vert^{2} +\frac{\left\Vert p_{1} - p_{0}\right\Vert^{2}\left\Vert v_{0} + v_{1}\right\Vert^{2}}{36\cos^{4}\left(\frac{\theta_{0} + \theta_{1}}{4}\right)} -\frac{\left\Vert p_{1} - p_{0}\right\Vert }{3\cos^{2}\left(\frac{\theta_{0} + \theta_{1}}{4}\right)}\left\langle p_{1} - p_{0},v_{0} + v_{1}\right\rangle \\
    & =\left\Vert p_{1} - p_{0}\right\Vert^{2}\left(1 +\frac{1 + \cos\left(\theta\right)}{18\cos^{4}\left(\frac{\theta_{0} + \theta_{1}}{4}\right)} -\frac{\cos\left(\theta_{0}\right) + \cos\left(\theta_{1}\right)}{3\cos^{2}\left(\frac{\theta_{0} + \theta_{1}}{4}\right)}\right).
\end{align*}
Since $p_1\neq p_0$ we get $ 1 +\frac{1 + \cos\left(\theta\right)}{18\cos^{4}\left(\frac{\theta_{0} + \theta_{1}}{4}\right)} -\frac{\cos\left(\theta_{0}\right) + \cos\left(\theta_{1}\right)}{3\cos^{2}\left(\frac{\theta_{0} + \theta_{1}}{4}\right)} = 0$.
Thus,
\[ 18\cos^{4}\left(\frac{\theta_{0} + \theta_{1}}{4}\right) + 1 + \cos\left(\theta\right) - 6\cos^{2}\left(\frac{\theta_{0} + \theta_{1}}{4}\right)\left(\cos\left(\theta_{0}\right) + \cos\left(\theta_{1}\right)\right) = 0 . \]
However, 
\begin{align*}
    & 18\cos^{4}\left(\frac{\theta_{0} + \theta_{1}}{4}\right) + 1 + \cos\left(\theta\right) - 6\cos^{2}\left(\frac{\theta_{0} + \theta_{1}}{4}\right)\left(\cos\left(\theta_{0}\right) + \cos\left(\theta_{1}\right)\right) \\
    &\geq 18\cos^{4}\left(\frac{\theta_{0} + \theta_{1}}{4}\right) - 6\cos^{2}\left(\frac{\theta_{0} + \theta_{1}}{4}\right)\left(\cos\left(\theta_{0}\right) + \cos\left(\theta_{1}\right)\right) \\
    & = 6\cos^{2}\left(\frac{\theta_{0} + \theta_{1}}{4}\right)\left(3\cos^{2}\left(\frac{\theta_{0} + \theta_{1}}{4}\right) - \cos\left(\theta_{0}\right) - \cos\left(\theta_{1}\right)\right)\\
    & = \resizebox{0.85\linewidth}{!}{\ensuremath{
    		6\cos^{2}\left(\frac{\theta_{0} + \theta_{1}}{4}\right)\left(\frac{3\left(\cos\left(\frac{\theta_0+\theta_1}{2}\right)+1\right)}{2}-2\cos\left(\frac{\theta_0+\theta_1}{2}\right)\cos\left(\frac{\theta_0-\theta_1}{2}\right)\right) }}\\
    & = 6\cos^{2}\left(\frac{\theta_{0} + \theta_{1}}{4}\right)\left(\frac{3}{2}-\cos\left(\frac{\theta_0+\theta_1}{2}\right)\left(2\cos\left(\frac{\theta_0-\theta_1}{2}\right)-\frac{3}{2}\right)\right)>0 ,
\end{align*} 
The last inequality holds, since by~\eqref{theta j definition} we have $\abs{\theta_0-\theta_1}\in[0,\pi]$, and therefore we also get $\abs{2\cos\left(\frac{\theta_0-\theta_1}{2}\right)-\frac{3}{2} }\leq \frac{3}{2}$, with equality only when $\abs{\theta_0-\theta_1 }=\pi$. However in this case $\abs{\cos\left(\frac{\theta_0+\theta_1}{2}\right)}=0$.

The last displayed inequality contradicts our assumption that $\frac{d}{dt}b\left(\frac{1}{2}\right)=0$,
and therefore, we obtain that $ v_{\frac{1}{2}}\neq 0$.
\end{proof}

\begin{proof}[\textbf{Proof of Proposition~\ref{Bezier average is Hermite average}}]
By the definition of $X$, it remains to show that properties (1)-(3) in Definition~\ref{Hermite average} holds for the B\'{e}zier average.
For identity on the limit diagonal, observe that from~\eqref{lines reconstraction} (which is proved next independently) we get:
\begin{equation*}
    B_{\omega }\left(\begin{pmatrix}
    p \\ 
    v \\ 
    \end{pmatrix},\begin{pmatrix}
    p+tv \\ 
    v \\ 
    \end{pmatrix}\right)=\begin{pmatrix}
    p+t\omega v \\
    v \\
    \end{pmatrix}\xrightarrow{t\longrightarrow 0^{ + }}\begin{pmatrix}
    p\\ 
    v\\ 
    \end{pmatrix}.
\end{equation*}
For symmetry, note that each  $\left(\begin{pmatrix}
p_{0} \\ 
v_{0} \\ 
\end{pmatrix},\begin{pmatrix}
p_{1} \\ 
v_{1} \\ 
\end{pmatrix}\right)$ and $\left(\begin{pmatrix}
p_{1} \\ 
 - v_{1} \\ 
\end{pmatrix},\begin{pmatrix}
p_{0} \\ 
 - v_{0} \\ 
\end{pmatrix}\right)$ defines the same control points in reverse order and therefore define the same B\'{e}zier curve corresponding to a reversed parametrization.
End points interpolation is easily verified by \eqref{bezier curve} and \eqref{tangents curve}.
\end{proof}

\begin{proof}[\textbf{Proof of Lemma~\ref{lemma: similarity invariance}}]
Equation~\eqref{similarity 2} is easy to verify by Definition~\ref{def: average def} and specifically by looking at~\eqref{eqn:alpha}, \eqref{bezier curve} and~\eqref{tangents curve}.

To observe that~\eqref{similarity1} holds, we exploit the fact that B\'{e}zier curves are invariant under affine transformations. Let $\Phi$ be an isometry over $\mathbb{R}^{n}$. In particular, $\Phi$ is an affine transformation. It is therefore sufficient to show that the control points associated with $B_\omega\left(\tilde{\Phi}\begin{pmatrix}
p_0\\v_0\\
\end{pmatrix},\tilde{\Phi}\begin{pmatrix}
p_1\\v_1\\
\end{pmatrix}\right)$ are given by applying $\Phi$ to 
$$p_0, p_0+\alpha v_0, p_1-\alpha v_1, p_1.$$
The control points associated with $B_\omega$ 
are: 
\begin{equation*}
    \Phi(p_0),\Phi(p_0)+\tilde{\alpha}\left(\Phi(v_0)-\Phi(0)\right),\Phi(p_1)-\tilde{\alpha}\left(\Phi(v_1)-\Phi(0)\right),\Phi(p_1),
\end{equation*}
where $\tilde{\alpha}=\alpha\left(\Phi(p_0),\Phi(p_1),\Phi(v_0)-\Phi(0),\Phi(v_1)-\Phi(0)\right)$.

It is clear that the first and last control points are obtained by applying $\Phi$ to the original points. As for the second and third control points, since $\Phi$ is an isometry, $\tilde{\alpha}=\alpha$ and since $\Phi$ is an affine transformation it is of the form $\Phi(x)=Rx+t$ with $R$  an orthogonal matrix, and $t\in\mathbb{R}^{n}$. Therefore, $t=\Phi(0)$ and
\begin{equation*}
    \Phi\left(p_0+\alpha v_0\right)= Rp_0+\alpha Rv_0+t=\Phi(p_0)+\alpha\left(\Phi(v_0)-t\right)=\Phi(p_0)+\tilde{\alpha}\left(\Phi(v_0)-\Phi(0)\right),
\end{equation*}
\begin{equation*}
    \Phi\left(p_1-\alpha v_1\right)= Rp_1-\alpha Rv_1+t=\Phi(p_1)-\alpha\left(\Phi(v_1)-t\right)=\Phi(p_1)-\tilde{\alpha}\left(\Phi(v_1)-\Phi(0)\right).
\end{equation*}
This completes the proof of~\eqref{similarity1}.
\end{proof}

\begin{proof}[\textbf{Proof of Lemma~\ref{line reconstruction lemma}}]
    Let $p_0,p_1 \in\mathbb{R}^{n}$ s.t. $p_0\neq p_1$, and let $v_{0}=v_1=u$. Then, we have that
    $\alpha\left(p_0,p_1, u,u\right)
    =\frac{1}{3}\|p_{1} - p_{0}\|$.
    By Definition~\ref{def: average def}
    \begin{equation*}
    B_{\omega }\left(\begin{pmatrix}
    p_{0} \\ 
    u \\       
  \end{pmatrix},\begin{pmatrix}
    p_{1} \\ 
    u \\      
    \end{pmatrix}\right)= \begin{pmatrix}
    b(w)\\
    \frac{b'(w)}{\|b'(w)\|}\\
    \end{pmatrix},
    \end{equation*}
    and by simple calculations based on~ \eqref{bezier curve}  and on~\eqref{tangents curve},
    we arrive at the claim of the lemma.
\end{proof}

\begin{proof}[\textbf{Proof of Lemma \ref{circle reconstruction lemma}}]
By lemma \ref{lemma: similarity invariance}, it is sufficient to show that for data sampled from the unit circle in $\mathbb{R}^{2}$ , the B\'{e}zier average with $\omega=\frac{1}{2}$ the mid-point of the circular arc determined by the data, and its ntangent. 

Assume $\left(\begin{pmatrix}
p_0\\v_0\\ \end{pmatrix}\begin{pmatrix}
p_1\\v_1\\ \end{pmatrix}\right)\in\mathbb{R}^{2}\times S^{1}$ are sampled from a unit circle.  We further assume w.l.o.g that $p_0=\begin{pmatrix}1\\0\\\end{pmatrix}, v_0=\begin{pmatrix}0\\1\\ \end{pmatrix}$ and denote $p_1=\begin{pmatrix}\cos(\varphi)\\ \sin(\varphi)\end{pmatrix}, v_1=\begin{pmatrix}-\sin(\varphi)\\ \cos(\varphi)\end{pmatrix}$ where $0<\varphi<2\pi$.
Basic geometrical reasoning shows that $\theta_0=\theta_1=\frac{\varphi}{2}$, and
that if $p_1$ lies on the upper semicircle, namely if $\varphi\leq\pi$, then $\varphi=\theta$, while otherwise, $\varphi=2\pi-\theta$. 
Using known trigonometric identities we obtain, 
\begin{equation}
    \alpha=\frac{d\left(p_0,p_1\right)}{3\cos^2\left(\frac{\theta_0+\theta_1}{4}\right)}=\frac{4}{3}\tan\left(\frac{\varphi}{4}\right),
\label{circle_alpha}
\end{equation} 
which is the parameter used in~\cite{dokken1990good}. Therefore, by Theorem~1 in \cite{dokken1990good}, the point of the B\'{e}zier average with weight $\frac{1}{2}$ of $\begin{pmatrix}p_0\\v_0\\ \end{pmatrix}$ and $\begin{pmatrix}p_1\\v_1\\ \end{pmatrix}$ is the midpoint of the circular arc defined by the two points and their ntangents. 

To verify that the average tangent is tangent to the unit circle at $p_{\frac{1}{2}}$, note that the tangent to the unit circle at $p_{\frac{1}{2}}$ is in the direction of $p_1-p_0$. Observe as well that $v_0+v_1$ linearly depends on $p_1-p_0$. In fact, it implies that:
\begin{equation} \label{eq: dot product of u an the sum of tangent vectors}
    \frac{\langle v_0+v_1,u\rangle}{\norm{v_0+v_1}}=\begin{cases}
    1 \quad &\text{if} \quad \varphi<\pi ,\\
    0 \quad &\text{if} \quad \varphi=\pi , \\
    -1 \quad &\text{if} \quad \varphi>\pi .\\
\end{cases}
\end{equation}
and that $\norm{v_0+v_1}=\sqrt{2+2\cos(\theta)}=2\cos\left(\frac{\theta}{2}\right)$. See Figure~\ref{fig:circle demo} for a graphical demonstration.
\begin{figure}
    \centering
    \includegraphics[width=0.275\textwidth]{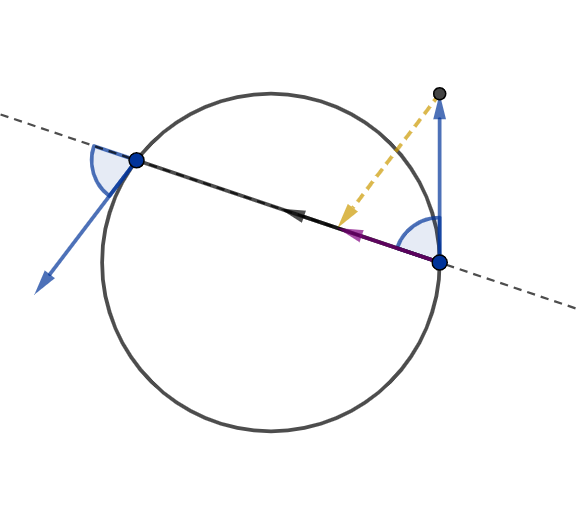}
    \quad
    \includegraphics[width=0.31\textwidth]{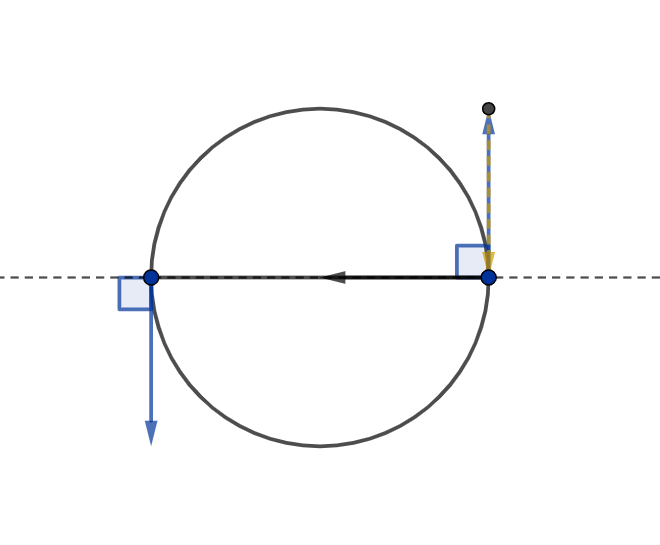}
    \quad
    \includegraphics[width=0.26\textwidth]{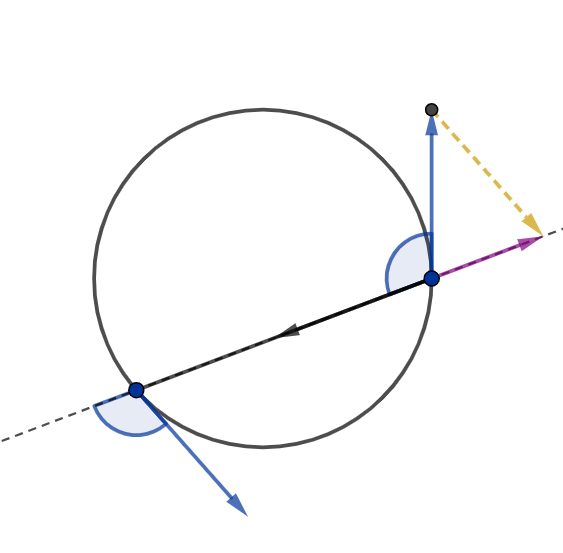}
    \caption{The sum of two directed tangent vectors to the unit circle linearly depends on the vector of difference between the two points of tangency. The yellow vector is the other tangent moved to form the sum of the two tangents which appears in red. The three sub-figures demonstrate the three possibilities mentioned in~\eqref{eq: dot product of u an the sum of tangent vectors}.}
    \label{fig:circle demo}
\end{figure}
The above means that \begin{equation}
\label{equation: sum of vectors}
    v_0+v_1=
    \begin{cases}
    2\cos\left(\frac{\theta}{2}\right)\cdot u\quad &\text{if}\quad \varphi\leq\pi ,\\
    -2\cos\left(\frac{\theta}{2}\right)\cdot u\quad &\text{if}\quad \varphi>\pi  . \\
\end{cases}
\end{equation}
Considering the numerator of $v_{\frac{1}{2}}$ given by~\eqref{mid-vector}, and applying~\eqref{circle_alpha} and~\eqref{equation: sum of vectors}, we get
\begin{align*}
       p_1-p_0-\frac{1}{2}\alpha\left(v_0+v_1\right) 
    & = \begin{cases}
        \left(\frac{1}{3}+\frac{1}{3\cos^{2}\left(\frac{\theta}{4}\right)}\right)\left(p_1-p_0\right) \quad &\text{if} \quad \varphi\leq\pi , \\
        \left(1+\frac{\sqrt{2+2\cos(\theta)}}{6\sin^{2}\left(\frac{\theta}{4}\right)}\right)(p_1-p_0) \quad &\text{if} \quad \varphi>\pi . \\
    \end{cases}\\
\end{align*}
Since both $  \left(\frac{1}{3}+\frac{1}{3\cos^{2}\left(\frac{\theta}{4}\right)}\right),\left(1+\frac{\sqrt{2+2\cos(\theta)}}{6\sin^{2}\left(\frac{\theta}{4}\right)}\right)>0$, we get that $v_{\frac{1}{2}}$ is in the direction of $p_1-p_0$ and thus is the tangent to the circle at $p_{\frac{1}{2}}$.
\end{proof}

\begin{proof}[\textbf{Proof of Lemma~\ref{lemma: reaveraging}}]
We prove that $\left(\begin{pmatrix}
p_0\\v_0\\
\end{pmatrix},\begin{pmatrix}
p_{\frac{1}{2}}\\v_{\frac{1}{2}}
\end{pmatrix}\right)$ and $\left(\begin{pmatrix}
p_{\frac{1}{2}}\\v_{\frac{1}{2}}\\
\end{pmatrix},\begin{pmatrix}
p_1\\v_1\\
\end{pmatrix}\right)$ satisfy conditions~\eqref{eqn:conditions_on_pts}, namely,
\begin{enumerate}[label=(\roman*)]
    \item $p_{\frac{1}{2}}\neq p_j, \quad j=0,1$.
    \item Let $u_0=u\left(\begin{pmatrix}
p_0\\v_0\\
\end{pmatrix},\begin{pmatrix}
p_{\frac{1}{2}}\\v_{\frac{1}{2}}\\
\end{pmatrix}\right)$ and $u_1=u\left(\begin{pmatrix}
p_{\frac{1}{2}}\\v_{\frac{1}{2}}\\
\end{pmatrix},\begin{pmatrix}
p_1\\v_1\\
\end{pmatrix}\right)$. Then $v_j=v_\frac{1}{2}=u_j$ or $v_j,v_\frac{1}{2},u_j$ are not pair-wise linearly dependent for $j=0,1$.
\end{enumerate}

First we prove (i). By algebraic manipulations we obtain that $\norm{p_\frac{1}{2}-p_0}$ is equal to
\begin{equation*}
	\resizebox{\linewidth}{!}{\ensuremath{
    \frac{\norm{p_1-p_0}}{\sqrt{32}\cos^2\left(\frac{\theta_0+\theta_1}{4}\right)}\left(8\cos^4\left(\frac{\theta_0+\theta_1}{4}\right)+1-\cos(\theta)+4\cos^2\left(\frac{\theta_0+\theta_1}{4}\right) ( \cos(\theta_0)-\cos(\theta_1) ) \right)^\frac{1}{2} . }}
\end{equation*}
Then, we argue that
\begin{equation}
\label{eq:average point different than initials}
    z := 8\cos^4\left(\frac{\theta_0+\theta_1}{4}\right)+1-\cos(\theta)+4\cos^2\left(\frac{\theta_0+\theta_1}{4}\right) ( \cos(\theta_0)-\cos(\theta_1) ) > 0,
\end{equation}
and hence $\norm{p_\frac{1}{2}-p_0}>0$. In other words, $p_\frac{1}{2}\neq p_0$. 

To show that \eqref{eq:average point different than initials} holds, we observe that
\begin{align*}
    z &\geq 8\cos^4\left(\frac{\theta_0+\theta_1}{4}\right)+1-\cos(\theta_0-\theta_1)+4\cos^2\left(\frac{\theta_0+\theta_1}{4}\right)(\cos(\theta_0)-\cos(\theta_1))\\
    &=
    \resizebox{0.95\linewidth}{!}{\ensuremath{
    8\cos^{4}\left(\frac{\theta_0+\theta_1}{4}\right) + 2 - 2\cos^2\left(\frac{\theta_0-\theta_1}{2}\right)- 16\cos^{3}\left(\frac{\theta_0+\theta_1}{4}\right)\sin\left(\frac{\theta_0+\theta_1}{4}\right)\sin\left(\frac{\theta_0-\theta_1}{2}\right) . }}
\end{align*}
Let \begin{equation*}  
    t=\frac{\theta_{0} + \theta_{1}}{4},\quad    
    s=\frac{\theta_{0}-\theta_1}{2}.
\end{equation*}
Then, $ 0\leq t < \frac{\pi}{2}$, and $-\frac{\pi}{2}< s < \frac{\pi}{2}$. We define,
\begin{equation*}
    h\left(t,s\right):=8\cos^{4}(t) + 2 - 2\cos^2(s)- 16\cos^{3}(t)\sin(t)\sin(s),
\end{equation*}
and obtain
\begin{align*}
    \frac{\partial h}{\partial t} & =-32\cos^3(t)\sin(t)+48\cos^2(t)\sin^2(t)\sin(s)-16\cos^4(t)\sin(s), \\
    \frac{\partial h}{\partial s} & =4\cos(s)\sin(s)-16\cos^{3}(t)\sin(t)\cos(s) .
\end{align*}
To derive the critical points of $h$, we observe that $\frac{\partial h}{\partial t}=0$ iff $\sin(s)=\frac{\sin(2t)}{3\sin^2(t)-\cos^2(t)}$ and similarly $\frac{\partial h}{\partial s} = 0$ iff $ \sin(s)=2\cos^2(t)\sin(2t)$. In other words,
\begin{align*}
    \nabla h=0 \iff \begin{cases}
        \sin(s)=2\cos^2(t)\sin(2t) , \\
        2\cos^2(t)\sin(2t)=\frac{\sin(2t)}{3\sin^2(t)-\cos^2(t)} .
    \end{cases}
\end{align*}
However, the only solution of the latter system of equations over $\left[0,\frac{\pi}{2}\right)\times\left(-\frac{\pi}{2},\frac{\pi}{2}\right)$ is $(0,0)$. Thus, we conclude that the minimum value of $h$ is obtained on the boundary of $\left[0,\frac{\pi}{2}\right)\times\left(-\frac{\pi}{2},\frac{\pi}{2}\right)$ or in $(0,0)$, which is also on the boundary (see Result~\ref{2}). We further explore all candidates of extremal points on the boundary of $\left[0,\frac{\pi}{2}\right)\times\left(-\frac{\pi}{2},\frac{\pi}{2}\right)$. 

We consider the restriction of $h$ to the boundaries, which are functions of one variable, and get that
\begin{align*}
    h_{t,l} &:=h\left(t,-\frac{\pi}{2}\right)=8\cos^{4}(t) +2+16\cos^{3}(t)\sin(t)\geq 2 > 0 , \\
    h_{t,r} &:=h\left(t,\frac{\pi}{2}\right)=8\cos^{4}(t) +2-16\cos^{3}(t)\sin(t) \geq 0 , \\
    h_{s,l} &:=h\left(0,s\right)=10-2\cos^2(s)\geq 8 > 0 , \\
    h_{s,r} &:=h\left(\frac{\pi}{2},s\right)= 2 -2\cos^2(s) \geq 0 .
\end{align*}
Now, $h_{t,r}(t)=0$ iff $t=\frac{\pi}{4}$,namely at the point $(\frac{\pi}{4},\frac{\pi}{2})$, while $h_{s,r}(s)=0$ iff $s=0$, namely at the point $(\frac{\pi}{2},0)$. Thus the above two equalities to $0$ are not satisfied simultaneously.  Finally, by the continuity of $h$ and since $h(t\pm\epsilon_0,s\pm\epsilon_1)\neq h(t,s)$ for small enough $\epsilon_0$ and $\epsilon_1$, we get that $h(t,s)>0$ for any $(t,s)\in\left[0,\frac{\pi}{2}\right)\times\left(-\frac{\pi}{2},\frac{\pi}{2}\right)$ and conclude that $p_{\frac{1}{2}}\neq p_0$.
By symmetric reasoning we can show that also $p_\frac{1}{2}\neq p_1$.

For the second claim, assume $v_j,v_\frac{1}{2},u_j$ are pair-wise linearly dependent. That is, $v_j=\pm u_j$ and $v_\frac{1}{2}=\pm u_j$  for some $j\in\{0,1\}$. For simplicity we assume $j=0$. We follow De Casteljau's algorithm  for obtaining $p_\frac{1}{2}$ and $v_\frac{1}{2}$. As in~\eqref{control points}, consider the initial control points of the associated B\'{e}zier curve and denote:
    \begin{equation*}
            q_0^0=p_{0},\quad q_1^0=p_{0} + \alpha v_{0},\quad q_2^0=p_{1} - \alpha v_{1},\quad q_3^0=p_{1} .
    \end{equation*}
The following points are obtained by De Casteljau's algorithm:
    \begin{align*}
    &q_0^1=\frac{q_0^0+q_1^0}{2},\quad q_1^1=\frac{q_1^0+q_2^0}{2},\quad q_2^1=\frac{q_2^0+q_3^0}{2}\\
    &q_0^2=\frac{q_0^1+q_1^1}{2},\quad q_1^2=\frac{q_1^1+q_2^1}{2}\\
    &q_0^3=\frac{q_0^2+q_1^2}{2}.
    \end{align*}
Then, $p_\frac{1}{2}=q_0^3$ and $v_\frac{1}{2}=\frac{q_1^2-q_0^2}{\norm{q_1^2-q_0^2}}$. For further information about De Casteljau's algorithm, we refer the reader to \cite{prautzsch2002bezier}. Since $v_0=\pm u_0$, the points $q_0^0,q_1^0$ and $q_0^3$ lie on the same line. We denote this line by $l$. Since $v_{\frac{1}{2}}=\pm u_0$, we have that $q_0^2$ and $q_1^2$ also lie on $l$. $q_0^1$ is the average of $q_0^0,q_1^0$  and hence it also lies on $l$. Now, since $q_0^2,q_0^1$ and $q_1^1$ lie on the same line, it follows that $q_1^1$ is on $l$. Similarly, $q_2^0$ is on $l$. Now $q_2^1$ is on $l$ since it lies on the line passing through $q_1^2$ and $q_1^1$ and finally we get that $q_3^0$ is on $l$. Thus $v_0,v_1,u$ are pair-wise linearly dependent and hence, by our assumption, $v_0=v_1=u$.  It is shown in Lemma~\ref{line reconstruction lemma} that in this case $v_0=v_\frac{1}{2}=u_0$ as required.
\end{proof}

\begin{proof}[\textbf{Proof of Lemma~\ref{points distance contraction lemma}}]
For the first part,
\begin{align*}
    d\left(p_{\frac{1}{2}},p_{0}\right) &=\left\Vert p_{\frac{1}{2}} - p_{0}\right\Vert  =\left\Vert \frac{1}{2}\left(p_{1} - p_{0}\right) +\frac{3}{8}\alpha\left(v_{0} - v_{1}\right)\right\Vert \\
    & \leq \frac{1}{2} d\left(p_{0},p_{1}\right) +\frac{d\left(p_{0},p_{1}\right)}{8\cos^{2}\left(\frac{\theta_{0} + \theta_{1}}{4}\right)}\left(\left\Vert v_{0}\right\Vert  +\left\Vert v_{1}\right\Vert\right) \\
    & = d(p_{0},p_{1})\cdot\left(\frac{1}{2} +\frac{1}{4\cos^{2}\left(\frac{\theta_{0} + \theta_{1}}{4}\right)}\right) .
\end{align*}
Since $ \theta_{0} + \theta_{1}\leq \pi$ we have that $ \cos^{2}\left(\frac{\theta_{0} + \theta_{1}}{4}\right)\geq \frac{1}{2}$ and $ d\left(p_{\frac{1}{2}},p_{0}\right)\leq d\left(p_{0},p_{1}\right)$. In the same way we obtain $ d\left(p_{\frac{1}{2}},p_{1}\right)\leq d\left(p_{0},p_{1}\right)$. For the second claim of the lemma, simply set $  \mu  = \frac{1}{2} +\frac{1}{4\cos^{2}\left(\frac{\gamma }{4}\right)}$ for any $\gamma$ satisfying
$ \theta_{0} + \theta_{1}\leq \gamma <\pi$.
\end{proof}

\section{The set \texorpdfstring{$X$}{X} of the B\'{e}zier Average} \label{app:subsec_conjecture}

The set $X\subset\left(\mathbb{R}^n\times S^{n-1}\right)^2$ consists of all pairs where the B\'{ezier} average is a Hermite average by Definition~\ref{Hermite average}. Namely, for any pair in $X$, the B\'{e}zier average fulfills the conditions~\eqref{eqn:conditions_on_pts} and produces averaged vectors in $S^{n-1}$ for any weight in $[0,1]$. This section further investigates the properties of elements in $X$ to receive some perception of the size and significance of $X$.

Let $x = \left(\begin{pmatrix}
p_0\\v_0\\
\end{pmatrix},\begin{pmatrix}
p_1\\v_1\\
\end{pmatrix}\right) \in \left(\mathbb{R}^n\times S^{n-1}\right)^2$ be a pair that meets the conditions~\eqref{eqn:conditions_on_pts}. Then, following the notation of Section~\ref{sec:HermiteAveraging}, $x\in X$ iff $b(t)$ of~\eqref{bezier curve} is regular over $[0,1]$. Nevertheless, recall that the regularity of $b(t)$ depends on the parameters $\theta_0,\theta_1$ and $\theta$ of $x$. In particular, to obtain regularity, the linearly independence of $v_0,v_1,u$ is sufficient; therefore, if $p_0,p_0+v_0,p_1,p_1+v_1$ do not lie on the same plane, then we conclude that $x\in X$. Thus, it remains to characterize the case where $p_0,p_0+v_0,p_1,p_1+v_1$ are in a two dimensional space.

Consider the vectors of difference between each two consecutive control points, 
\begin{equation*}
    \alpha v_0, \quad p_1-p_0-\alpha (v_0+v_1), \quad \alpha v_1 .
\end{equation*}
By Theorem~2 in~\cite{lin2005regular}, if the angles between each pair of those vectors are acute, then $b(t)$ is regular. That is, if 
\begin{equation} \label{eqn:acute_angles}
\theta<\frac{\pi}{2} \quad \text{ and } \quad \langle v_j, p_1-p_0-\alpha (v_0+v_1)\rangle>0 \quad j=0,1.
\end{equation}
Then, $x\in X$. The inner products in~\eqref{eqn:acute_angles} are
\begin{equation}
\label{eq: angles}
    \langle v_j, p_1-p_0-\alpha (v_0+v_1)\rangle=\norm{p_1-p_0}\left(\cos(\theta_j)-\frac{1+\cos(\theta)}{3\cos^2\left(\frac{\theta_0+\theta_1}{4}\right)}\right), \quad j=0,1.
\end{equation}
Namely, if $\theta<\frac{\pi}{2}$ and $3\cos^2\left(\frac{\theta_0+\theta_1}{4}\right)\cos(\theta_j)-1-\cos(\theta)>0$ for any $j=0,1$, then $x\in X$.

In the plane we have the following relationship between $\theta_0,\theta_1$ and $\theta$:
\begin{equation}
\label{possibilities for theta}
\theta=\theta_0 + \theta_1  \quad \mathrm{or}\quad \theta= \abs{\theta_0-\theta_1}  \quad\mathrm{or}\quad  \theta=2\pi-\left(\theta_0+\theta_1\right)
\end{equation}
In Figure~\ref{fig:acute angles} we plot the domains of positivity of~\eqref{eq: angles} for $\theta=\theta_0+\theta_1$ and for $\theta=\abs{\theta_0-\theta_1}$. Note that the first case is equivalent in sense of the positivity of~\eqref{eq: angles} to the case where $\theta=2\pi-\left(\theta_0+\theta_1\right)$ and therefore the latter is omitted. We consider the intersection of the domains in question, for both $j=0,1$, together with $\theta<\frac{\pi}{2}$. We conclude that even in the case where the control points $p_0,p_0+v_0,p_1,p_1+v_1$ lie on a plane (a rare event in high dimensions) the domain where $x\in X$ is not minor. 

\begin{figure}
    \centering
    \begin{subfigure}[t]{0.3\textwidth}
        \includegraphics[width=\textwidth]{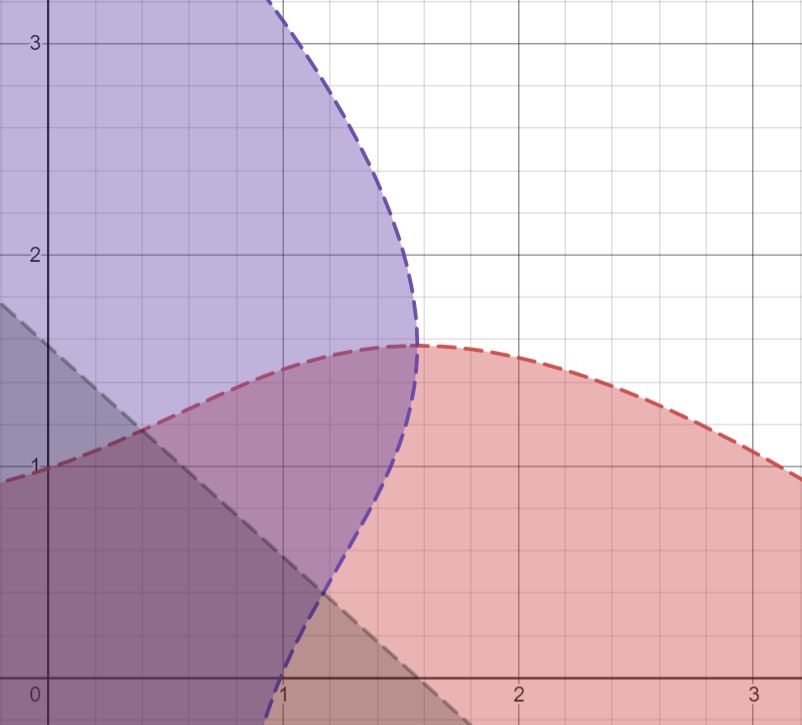}
        \caption*{$\theta = \theta_0 + \theta_1$}
    \end{subfigure}\quad\quad
    \begin{subfigure}[t]{0.3\textwidth}
        \includegraphics[width=\textwidth]{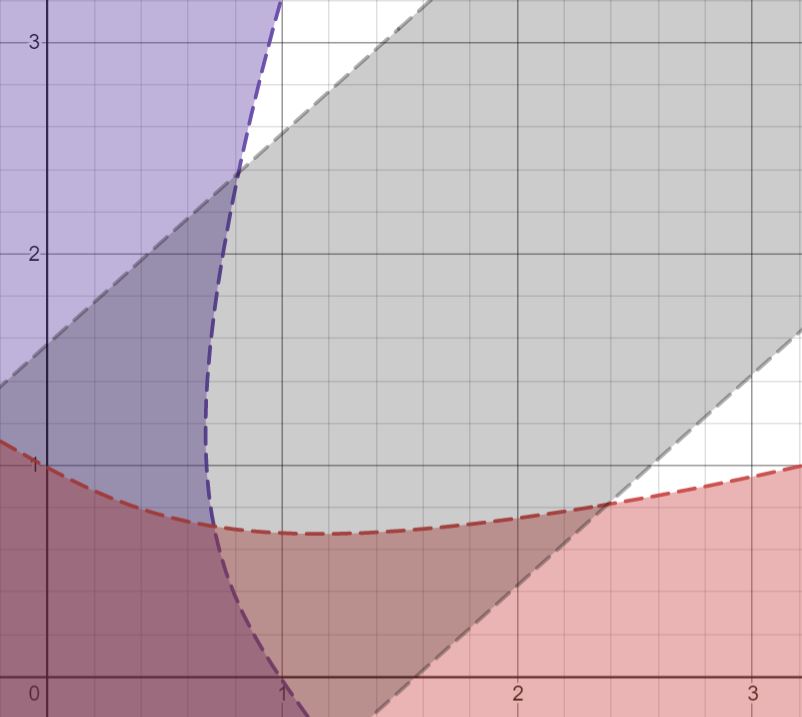}
        \caption*{$\theta = \abs{\theta_0 - \theta_1}$}
    \end{subfigure}

    \caption{Intersections as sufficient conditions for a pair to be in $X$ when control points lie on a plane. The graphs show the two fundamental cases of relations between the parameters. In each graph, we see the possible domain of the parameters $\theta_0$ (horizontal axis) and $\theta_1$ (vertical axis), where the blue and red subdomains indicate $\langle v_j,p_1-p_0-\alpha (v_0+v_1)\rangle>0$ for $j=0$ and $j=1$, respectively. The black domain is when $\theta<\frac{\pi}{2}$. The intersections are when all conditions are met.}
    \label{fig:acute angles}
\end{figure}

The conditions on $\theta_0,\theta_1$ defined by the intersections, as presented in Figure~\ref{fig:acute angles}, are sufficient to determine that $x\in X$. Nevertheless, these conditions are no necessary. To complete the picture, we briefly show a necessary conditions for a pair $x$ to be in $\left(\mathbb{R}^n\times S^{n-1}\right)^2\setminus X$. 

Recall that we consider only the case when the control points are planer. Therefore, we assume, w.l.o.g, that $p_0=\begin{pmatrix}
0\\0\\
\end{pmatrix}, p_1=\begin{pmatrix}
1\\0\\
\end{pmatrix}, v_0=\begin{pmatrix}
\cos{x}\\ \sin{x}\\
\end{pmatrix}, v_1=\begin{pmatrix}
\cos{y}\\ \sin{y}\\
\end{pmatrix}$ for some arbitrary $x,y \in \left[-\pi,\pi\right]$. In this case, $\theta_0=|x|$ and $\theta_1=|y|$. Assigning the above to~\eqref{tangents curve} yields:
\begin{equation*}
    \frac{d}{dt}b\left(t\right) = 6t\left(1 - t\right)\begin{pmatrix}
    1\\0\\
    \end{pmatrix} + \frac{\left(3t - 1\right)\left(t - 1\right)}{\cos^2\left(\frac{|x|+|y|}{4}\right)} \begin{pmatrix}
    \cos{x}\\ \sin{x}\\
    \end{pmatrix} + \frac{t\left(3t - 2\right)}{\cos^2\left(\frac{|x|+|y|}{4}\right)} \begin{pmatrix}
    \cos{y}\\ \sin{y}\\
    \end{pmatrix}.
\end{equation*}
Therefore, $\frac{d}{dt}b\left(t\right)=0$ iff 
\begin{equation} \label{eqn:quad_sys_t}
    \begin{cases}
        6t\left(1 - t\right)\cos^2\left(\frac{|x|+|y|}{4}\right)+\left(3t - 1\right)\left(t - 1\right)\cos{x}+t\left(3t - 2\right)\cos{y}=0 \\
        \left(3t - 1\right)\left(t - 1\right)\sin{x}+t\left(3t - 2\right)\sin{y}=0 .\\
    \end{cases}
\end{equation}
The equations in~\eqref{eqn:quad_sys_t} are quadratic in $t$ but transcendental in $x$,$y$ and their absulote value sum and so we conclude that $x \notin X$ iff the system~\eqref{eqn:quad_sys_t} has a solution in $t \in [0,1]$.

\end{appendices}

\end{document}